\documentclass[12pt]{amsart}

\usepackage[margin=1.25in]{geometry}

\usepackage{amsmath}
\usepackage{amsthm}
\usepackage{amssymb}
\usepackage{caption}
\usepackage[curve]{xypic}
\usepackage{cite}
\usepackage{enumitem}
\usepackage{url}
\usepackage{color}
\usepackage[usenames,dvipsnames]{xcolor}
\usepackage{graphicx}
\usepackage{tikz}
\usepackage{tikz-cd}
\usepackage{hyperref} 
\usepackage{verbatim}
\usepackage{multirow}
\usepackage{subcaption}
\usepackage{graphicx}
\usepackage{algorithm}
\usepackage{algorithmic}

\setlist{itemsep=4pt, topsep=4pt}

\makeatletter

\def\chaptermark#1{}

\def\chapter{%
\if@openright\cleardoublepage\else\clearpage\fi
\thispagestyle{plain}\global\@topnum\z@
\@afterindenttrue \secdef\@chapter\@schapter}

\def\@chapter[#1]#2{\refstepcounter{chapter}%
\ifnum\c@secnumdepth<\z@ \let\@secnumber\@empty
\else \let\@secnumber\thechapter \fi
\typeout{\chaptername\space\@secnumber}%
\def\@toclevel{0}%
\ifx\chaptername\appendixname \@tocwriteb\tocappendix{chapter}{#2}%
\else \@tocwriteb\tocchapter{chapter}{#2}\fi
\chaptermark{#1}%
\addtocontents{lof}{\protect\addvspace{10\p@}}%
\addtocontents{lot}{\protect\addvspace{10\p@}}%
\@makechapterhead{#2}\@afterheading}
\def\@schapter#1{\typeout{#1}%
\let\@secnumber\@empty
\def\@toclevel{0}%
\ifx\chaptername\appendixname \@tocwriteb\tocappendix{chapter}{#1}%
\else \@tocwriteb\tocchapter{chapter}{#1}\fi
\chaptermark{#1}%
\addtocontents{lof}{\protect\addvspace{10\p@}}%
\addtocontents{lot}{\protect\addvspace{10\p@}}%
\@makeschapterhead{#1}\@afterheading}
\newcommand\chaptername{Chapter}

\def\@makechapterhead#1{\global\topskip 7.5pc\relax
\begingroup
\fontsize{\@xivpt}{18}\bfseries\centering
\ifnum\c@secnumdepth>\m@ne
  \leavevmode \hskip-\leftskip
  \rlap{\vbox to\z@{\vss
      \centerline{\normalsize\mdseries
          \uppercase\@xp{\chaptername}\enspace\thechapter}
      \vskip 3pc}}\hskip\leftskip\fi
 #1\par \endgroup
\skip@34\p@ \advance\skip@-\normalbaselineskip
\vskip\skip@ }
\def\@makeschapterhead#1{\global\topskip 7.5pc\relax
\begingroup
\fontsize{\@xivpt}{18}\bfseries\centering
#1\par \endgroup
\skip@34\p@ \advance\skip@-\normalbaselineskip
\vskip\skip@ }

\newcounter{chapter}

\newif\if@openright

\makeatother

\makeatletter 
\def\@cite#1#2{{\m@th\upshape\bfseries%
[{#1\if@tempswa{\m@th\upshape\mdseries, #2}\fi}]}}
\makeatother 

\theoremstyle{plain}
\newtheorem{thm}{Theorem}[section]
\newtheorem{cor}[thm]{Corollary}

\newtheorem{prop}[thm]{Proposition}
\newtheorem{lem}[thm]{Lemma}
\newtheorem{sublem}[thm]{Sublemma}
\theoremstyle{definition}

\theoremstyle{remark}
\newtheorem{rem}[thm]{Remark}

\numberwithin{equation}{subsection}

\newcommand{\nc}{\newcommand}
\newcommand{\rnc}{\renewcommand}

\nc\bA{\mathbb{A}}
\nc\bB{\mathbb{B}}
\nc\bC{\mathbb{C}}
\nc\bD{\mathbb{D}}
\nc\bE{\mathbb{E}}
\nc\bF{\mathbb{F}}
\nc\bG{\mathbb{G}}
\nc\bH{\mathbb{H}}
\nc\bI{\mathbb{I}}
\nc{\bJ}{\mathbb{J}} 
\nc\bK{\mathbb{K}}
\nc\bL{\mathbb{L}}
\nc\bM{\mathbb{M}}
\nc\bN{\mathbb{N}}
\nc\bO{\mathbb{O}}
\nc\bP{\mathbb{P}}
\nc\bQ{\mathbb{Q}}
\nc\bR{\mathbb{R}}
\nc\bS{\mathbb{S}}
\nc\bT{\mathbb{T}}
\nc\bU{\mathbb{U}}
\nc\bV{\mathbb{V}}
\nc\bW{\mathbb{W}}
\nc\bY{\mathbb{Y}}
\nc\bX{\mathbb{X}}
\nc\bZ{\mathbb{Z}}
\nc\cA{\mathcal{A}}
\nc\cB{\mathcal{B}}
\nc\cC{\mathcal{C}}
\rnc\cD{\mathcal{D}}
\nc\cE{\mathcal{E}}
\nc\cF{\mathcal{F}}
\nc\cG{\mathcal{G}}
\rnc\cH{\mathcal{H}}
\nc\cI{\mathcal{I}}
\nc{\cJ}{\mathcal{J}} 
\nc\cK{\mathcal{K}}
\rnc\cL{\mathcal{L}}
\nc\cM{\mathcal{M}}
\nc\cN{\mathcal{N}}
\nc\cO{\mathcal{O}}
\nc\cP{\mathcal{P}}
\nc\cQ{\mathcal{Q}}
\rnc\cR{\mathcal{R}}
\nc\cS{\mathcal{S}}
\nc\cT{\mathcal{T}}
\nc\cU{\mathcal{U}}
\nc\cV{\mathcal{V}}
\nc\cW{\mathcal{W}}
\nc\cY{\mathcal{Y}}
\nc\cX{\mathcal{X}}
\nc\cZ{\mathcal{Z}}
\nc\bfA{\mathbf{A}}
\nc\bfB{\mathbf{B}}
\nc\bfC{\mathbf{C}}
\nc\bfD{\mathbf{D}}
\nc\bfE{\mathbf{E}}
\nc\bfF{\mathbf{F}}
\nc\bfG{\mathbf{G}}
\nc\bfH{\mathbf{H}}
\nc\bfI{\mathbf{I}}
\nc{\bfJ}{\mathbf{J}} 
\nc\bfK{\mathbf{K}}
\nc\bfL{\mathbf{L}}
\nc\bfM{\mathbf{M}}
\nc\bfN{\mathbf{N}}
\nc\bfO{\mathbf{O}}
\nc\bfP{\mathbf{P}}
\nc\bfQ{\mathbf{Q}}
\nc\bfR{\mathbf{R}}
\nc\bfS{\mathbf{S}}
\nc\bfT{\mathbf{T}}
\nc\bfU{\mathbf{U}}
\nc\bfV{\mathbf{V}}
\nc\bfW{\mathbf{W}}
\nc\bfY{\mathbf{Y}}
\nc\bfX{\mathbf{X}}
\nc\bfZ{\mathbf{Z}}

\newcommand{\xra}{\mathop{\longrightarrow}^}
\newcommand{\del}{\partial}

\nc{\dmo}{\DeclareMathOperator}
\nc{\wt}{\widetilde}
\rnc{\Re}{\operatorname{Re}}
\rnc{\Im}{\operatorname{Im}}
\rnc{\span}{\operatorname{span}}
\dmo{\rank}{rank}
\dmo{\End}{End}
\dmo{\Hom}{Hom}
\dmo{\Jac}{Jac}
\dmo{\Id}{Id}
\dmo{\Ann}{Ann}
\dmo{\Area}{Area}
\dmo{\CP}{\bC P^1}
\dmo{\rk}{rk}
\dmo{\rel}{rel}
\dmo{\ra}{\rightarrow}
\dmo{\AGL}{\mathrm{AGL}}
\dmo{\AO}{\mathrm{AO}}
\dmo{\Sym}{\mathrm{Sym}}
\dmo{\Hur}{\mathrm{Hur}}
\dmo{\Aut}{\mathrm{Aut}}
\dmo{\Mod}{\mathrm{Mod}}

\rnc{\Col}{\operatorname{Col}}

\nc{\ColOne}{\Col_{\bfC_1}}
\nc{\ColOneX}{\ColOne(X,\omega)}

\nc{\ColTwo}{\Col_{\bfC_2}}
\nc{\ColTwoX}{\ColTwo(X,\omega)}

\nc{\ColThree}{\Col_{\bfC_3}}
\nc{\ColThreeX}{\ColThree(X,\omega)}

\nc{\ColOneTwo}{\Col_{\bfC_1, \bfC_2}}
\nc{\ColOneTwoX}{\ColOneTwo(X,\omega)}

\nc{\ColOneThree}{\Col_{\bfC_1, \bfC_3}}
\nc{\ColOneThreeX}{\ColOneThree(X,\omega)}

\nc{\MOne}{\cM_{\bfC_1}}
\nc{\MTwo}{\cM_{\bfC_2}}
\nc{\MOneTwo}{\cM_{\bfC_1, \bfC_2}}

\nc{\MThree}{\cM_{\bfC_3}}

\nc{\MOneThree}{\cM_{\bfC_1, \bfC_3}}

\dmo{\For}{\cF}

\nc{\GL}{\mathrm{GL}^+(2, \bR)}

\title[Components of strata and their orbit closures]{Components of strata of $k$-differentials and their orbit closures}
\author[Apisa]{Paul Apisa}
\author[Aygun]{Juliet Aygun}

\begin{document}

\begin{abstract}
We obtain a complete classification of components of strata of holomorphic and meromorphic $k$-differentials. We show that, when genus is at least two and outside of explicit exceptions when $k \leq 3$, there is one primitive nonhyperelliptic component unless $k$ is odd and all singularities have even order, in which case there are two distinguished by their Arf invariant. The exceptions include new sporadic components of strata of cubic differentials. Our work provides a new proof of earlier results of Kontsevich-Zorich, Boissy, Lanneau, and Chen-Gendron when $k = 1, 2$. The proofs are almost purely algebraic, relying on the multiscale compactification of Bainbridge-Chen-Gendron-Grushevsky-M\"oller. This answers a question of Chen-Yu. 

We also show that for any component of a stratum of finite area $k$-differentials of positive genus, the smallest $\mathrm{GL}(2, \bR)$-orbit closure containing all of its holonomy covers is as big as possible. This answers the positive genus analogue of a ``long term goal'' of Mirzakhani-Wright. The result also holds in genus zero strata that contain surfaces with Euclidean cylinders, thus addressing infinitely many cases of the original question of Mirzakhani-Wright.
\end{abstract}

\maketitle

\setcounter{tocdepth}{1} 
\tableofcontents

\vspace{-5mm}

\section{Introduction}\label{S:Intro}

Let $k$ be a positive integer. If $X$ is a closed Riemann surface of genus $g$, then a \emph{k-differential} $\omega$ is a tensor on $X$ that can be written in any local holomorphic coordinate $z$ on $X$ as $f(z)dz^k$ where $f$ is meromorphic. In particular, if $K_X$ is the canonical bundle and $P = \sum_{p} |n_p| p$ is the sum of poles weighted by their order, then $\omega \in H^0(K_X^{\otimes k}(P))$. The \emph{order multiset} $\kappa$ of orders of zeros and poles is an integer partition of $k(2g-2)$. Given such a partition, the \emph{stratum} $\Omega^k \cM_g(\kappa)$ is the set of all pairs $(X, \omega)$ where $X$ belongs to the moduli space $\cM_g$ of smooth genus $g$ curves and $\omega$ is a $k$-differential on $X$ with order multiset $\kappa$. Given such an order multiset $\kappa = (k_i)_{i=1}^n$, $\Omega^k \cM_g(\kappa)$ is a $\bC^\times$-bundle over
\[ \bP \Omega^k \cM_g(\kappa) := \{ (X; p_1, \hdots, p_n) \in \cM_{g,n} : \sum_i k_i p_i \sim k K_X \}. \]
Perhaps surprisingly, components of strata have remained unclassified when $k > 2$. The question of classifying them has been around for many years and was formally posed by Chen-Yu \cite[Problem 3.3]{chen2026algebrogeometric}. For $k=1$, the classification of strata without poles was achieved by Kontsevich-Zorich \cite{KZ} and then by Boissy \cite{Boissy-Components} in general. For $k=2$, strata with at most simple poles were classified by Lanneau \cite{Lconn} and then by Chen-Gendron \cite{ChenGendron} in general. 

A component of a stratum is \emph{primitive} if one and hence every $k$-differential in it is not a power of a lower order differential. A component is \emph{hyperelliptic} if every element $(X, \omega)$ has the property that $X$ is hyperelliptic with the divisor of $\omega$ invariant under the hyperelliptic involution. Hyperelliptic components of strata were classified by Chen-Gendron \cite{ChenGendron}. In light of these observations, in order to classify components of strata, it suffices to classify primitive nonhyperelliptic ones. 

Given a $k$-differential $(X, \omega)$, let $\Sigma_{g,n}$ be the underlying topological genus $g$ surface punctured at the $n$ singularities. A $k$-differential has an associated winding number $w \in H^1(UT\Sigma_{g,n}, \bZ)$ where $UT\Sigma_{g,n}$ is the unit tangent bundle. If $k$ is odd and all singularities of the $k$-differential are of even order, $w$ mod $2$ belongs to $H^1(UT\Sigma_g; \bZ/2)$ and defines a $\bZ/2$-valued \emph{Arf invariant}. See Section \ref{S:Arf} for a description of this invariant and its ``relative'' version. The Arf invariant was studied by Atiyah \cite{Atiyah-spin}, Mumford \cite{Mumford-theta}, and Johnson \cite{Johnson-Spin}. 

The main theorem of this paper is the following. By our previous remarks, this provides a complete classification of components of strata.

\begin{thm}\label{T1}
If $g \geq 2$ and $k > 3$, $\Omega^k \cM_g(\kappa)$ has a unique primitive nonhyperelliptic component unless $k$ is odd and all singularities have even order, in which case there are two such components distinguished by their Arf invariants. When $g \geq 2$ and $k \leq 3$, the same holds except:
\begin{itemize}
    \item $\Omega^1 \cM_g(\kappa)$ is empty if the multiset of poles in $\kappa$ is $(-1)$.
    \item There are no such components of $\Omega^1 \cM_2(\kappa)$ with $\kappa = (2), (1,1)$ and of $\Omega^2 \cM_2(\kappa)$ with $\kappa = (4), (3,1), (2,2), (2,1,1), (1,1,1,1)$. 
    \item There is only one such component of $\Omega^1 \cM_3(\kappa)$ and $\Omega^1 \cM_2(\kappa, -2)$ where $\kappa = (4), (2,2)$.
    \item There is only one such component of $\Omega^3 \cM_2(\kappa)$ where $\kappa = (6), (4,2), (2,2,2)$.
    \item There are two such components of $\Omega^1 \cM_g(\kappa, -1, -1)$ distinguished by the relative Arf invariant when $g \geq 3$ and $\kappa$ is a positive even partition of $2g$.
    \item There are two such components of $\Omega^2 \cM_3(\kappa, -1)$ with $\kappa = (9), (6,3), (3,3,3)$ and $\Omega^2 \cM_4(\kappa)$ with $\kappa = (12), (9,3), (6,6), (6,3,3), (3,3,3,3)$.
    \item There are three such components of $\Omega^3 \cM_3(\kappa)$ with $\kappa = (12), (8,4), (4,4,4)$ distinguished by the Arf invariant and the generic value of $h^0(\cO_X(P))$ where $4P$ is the divisor of the cubic differential $X$.
\end{itemize}
\end{thm}

It is classical that genus zero strata are connected. Components of genus one strata were classified by Boissy \cite{Boissy-Components} when $k=1$ and Chen-Gendron \cite{ChenGendron} in general. It will be rederived from scratch in Section \ref{S:GenusOneComponents}. Finally, Bogatyrev-Gendron \cite{Bogatyrev-Gendron} classified components of $\Omega^k \cM_2(2k)$ for all $k$. 

We emphasize that the existence of sporadic components of $\Omega^3 \cM_3(\kappa)$ with $\kappa = (12), (8,4), (4,4,4)$ was previously unknown. They are constructed in Section \ref{S:SporadicCubics}. This parallels the algebro-geometric characterization of the sporadic components of strata of quadratic differentials discovered by Chen-M\"oller \cite{ChenMoller-Exceptional}. Section \ref{S:SporadicCubics} fundamentally uses their work.

The classification in Theorem \ref{T1} solves a question of Chen-Yu \cite[Problem 3.3]{chen2026algebrogeometric}. Its proof is almost purely algebraic, relying most heavily on the multiscale compactification of \cite{BCGGM-k-diff}, which we recall in Section \ref{S:BCGGM}. Curiously, the global $k$-residue condition makes almost no appearance in the sequel. Our proof provides a new proof of results of Kontsevich-Zorich, Boissy, Lanneau, and Chen-Gendron when $k = 1, 2$. The only geometric, non-algebraic inputs to the proof are: (1) the noncompactness of strata that do not project to points in $\cM_{g,n}$, (2) horizontal nodes in the multiscale compactification correspond to Euclidean cylinders, (3) nonseparating Euclidean cylinders are generically simple, and (4) collapsing a simple Euclidean cylinder splits a zero or merges two singularities (see Figure \ref{F:SplitMerge}). Removing this geometric input would yield a purely algebraic proof and hence resolve another question of Chen-Yu \cite[Problem 3.1]{chen2026algebrogeometric}.

\subsection{Geometric structures, orbit closures, and counting problems.}

A \emph{flat cone metric} on $\Sigma_{g,m}$ is a set of $n$ points $P \subseteq \Sigma_{g,m}$, called \emph{cone points}, and a metric on $\Sigma_{g,m}$ that is smooth and flat on $\Sigma_{g,n+m} := \Sigma_{g,m} - P$. This determines a $(G, X)$-structure on $\Sigma_{g,n+m}$ where $X = \bC$ and $G = S^1 \ltimes \bC$ is the group of rigid motions. Each such surface admits a \emph{holonomy homomorphism} from $\pi_1(\Sigma_{g,n+m})$ to $S^1$. The cover associated to its kernel is the \emph{holonomy cover}. It is characterized as the smallest cover under which the metric pulls back to one induced from a meromorphic $1$-form. 

A $k$-differential is equivalently a flat cone metric with holonomy valued in the $\bZ/k$ subgroup of $S^1$. It has finite area if all poles have (signed) order strictly greater than $-k$. Any finite area flat cone metric can be described as a finite union of Euclidean triangles in $\bR^2$ with sides glued together by rigid motions.

The holonomy cover of a finite area $k$-differential is a finite area abelian differential. Strata of abelian differentials admit an action of $\mathrm{GL}(2, \bR)$ generated by scalar multiplication and Teichm\"uller geodesic flow. By work of Eskin-Mirzakhani \cite{EM} and Eskin-Mirzakhani-Mohammadi \cite{EMM}, the orbit closure of any point in the stratum is a submanifold and, by work of Filip \cite{Fi1}, a variety. Such orbit closures are called \emph{invariant subvarieties}.

Given a component $\cC$ of a stratum of finite-area $k$-differentials, let $\cM_\cC$ be the smallest invariant subvariety containing every holonomy cover of a $k$-differential in the component. When $k$ is even, the subvariety $\cM_\cC$ is naturally contained in a \emph{quadratic double}, which is an invariant subvariety consisting of holonomy covers of a component of a stratum of quadratic differentials. Similarly, if $k$ is odd and $\cC$ is hyperelliptic, $\cM_{\cC}$ is contained in a codimension one hyperelliptic locus. Say that $\cM_\cC$ is \emph{as big as possible} if these containments are equalities or, when $k$ is odd and $\cC$ is nonhyperelliptic, if $\cM_\cC$ coincides with the ambient stratum. 

A deep body of work of Eskin-Masur \cite{EMa}, Eskin-Masur-Zorich \cite{EMZboundary}, Eskin-Kontsevich-Zorich \cite{EKZbig}, and others shows that the asymptotics (in $L$) of geometric counting functions, like the number of saddle connections between two fixed points or the number of cylinders of circumference at most $L$, for almost-every element of $\cC$ is governed by $\cM_\cC$. Partially motivated by this, Mirzakhani-Wright \cite{MirWri2} wrote that their ``long term goal is to compute'' $\cM_\cC$ for components $\cC$ of genus zero strata. Another motivation is that $\cM_\cC$ can furnish exotic orbit closures. These include Teichm\"uller curves \cite{V, Ward, Vo, KS} and the seven higher rank examples discovered by Eskin-McMullen-Mukamel-Wright \cite{MMW, EMMW}. We show that such unusual behavior is confined to genus zero. 

\begin{thm}\label{T2}
If $\cC$ is a component of a stratum of surfaces of genus $g \geq 1$, then $\cM_{\cC}$ is as big as possible. The same holds if $\cC = \Omega^k \cM_0(\kappa)$ and $\kappa$ contains a subset summing to $-k$.
\end{thm} 

Previous known cases of Theorem \ref{T2} include Aygun \cite{Aygun2025_PrimeOrderkDifferentials}, when $k$ is prime and $g >2$, and Apisa \cite[Theorems 1.9, 1.10]{Apisa-Codim1Hyp} who proved a special case of the genus zero part of the result. Given a multiset $\theta$ of $n$ rational multiples of $\pi$ summing to $(n-2)\pi$, let $\cP(\theta)$ be the set of $n$-gons whose vertices have angles given by $\theta$.

\begin{cor}\label{C:BilliardConsequence}
If a subset $S$ of $\theta$ sums to $(|S|-1)\pi$, then almost every element of $\cP(\theta)$ has an unfolding with an orbit closure as big as possible. This assumption holds for $\theta$ coming from cyclic $n$-gons when $n$ is even. 
\end{cor}

Given a $k$-differential, a pole will be called \emph{invisible} if its order is $-k+d$ where $1 \leq d \mid k$. When $k$ is odd, say that a stratum $\Omega^k \cM_g(\kappa)$ is \emph{disconnected type} if every singularity has even order or if the multiset of zeros in $\kappa$ has the form $(a)$ or $(a,a)$ where $\gcd(a,k) = 1$ and all other entries are invisible poles. If $k$ is odd, the stratum will be called \emph{connected type} if it is not of disconnected type. Primitive components of strata of connected type contain $k$-differentials whose holonomy covers belong to connected strata of abelian differentials. By the deep body of work mentioned above and the large-genus asymptotics of Siegel-Veech constants due to Aggarwal \cite{Aggarwal-LargeGenus}, we have the following. 

\begin{cor}\label{C:Asymptotics}
Let $k$ be odd and let $\cC$ be as in Theorem \ref{T2}. Suppose that $\cC$ is primitive, nonhyperelliptic, and connected type. Given $X \in \cC$ and distinct $p,q \in X$, let $N_X(L)$ be the number of saddle connections from $p$ to $q$ of length at most $L$. Let $\theta_p$ and $\theta_q$ be the cone angles around $p$ and $q$. For almost every $X \in \cC$ and any $p,q \in X$, 
\[ N_X(L) \sim \left(\frac{\theta_p \theta_q}{4\pi} + O\left( \frac{1}{k(g+1)} \right) \right) \frac{L^2}{\mathrm{area}(X)}. \]
\end{cor}

For similar asymptotics when $\cC$ is hyperelliptic, see Apisa \cite{Apisa-Codim1Hyp}, which uses work of Athreya-Eskin-Zorich \cite{AEZ}. The $O(1/(k(g+1)))$ term in the asymptotics can be computed explicitly by work of Eskin-Okounkov \cite{EskinOkounkov}. By combining Corollaries \ref{C:BilliardConsequence} and \ref{C:Asymptotics}, one can compute the asymptotic number of billiard trajectories between any two points on a full measure set of polygons in $\cP(\theta)$ where $\theta$ has a subset $S$ summing to $(|S|-1)\pi$. The only obstacle to upgrading Corollary \ref{C:Asymptotics} to handle the case of $k$ even, strata of disconnected type, or to deduce the asymptotic number of closed geodesics is that analogues of Aggarwal's results are not yet fully available in those settings.

\subsection{Cylinders, fundamental groups, and the proof of Theorem \ref{T1}}

Two simple disjoint saddle connections on a flat cone surface are \emph{hat-homologous} if they have a locally constant ratio of lengths in the stratum. In Section \ref{S:MasurZorich}, we show the following.

\begin{thm}\label{T:MZ}
Two simple disjoint saddle connections on a $k$-differential that is not a finite area translation surface are hat-homologous if and only if their complement has a component that is a finite area translation surface. 
\end{thm}

This is the direct analogue of work of Masur-Zorich \cite{MZ} in strata of finite area quadratic differentials. As in \cite{MZ}, a consequence is that, up to perturbing to a nearby surface in the stratum, every Euclidean cylinder has one of five types (see \cite[Section 4.1]{ApisaWrightDiamonds}). 

Recall that a \emph{Euclidean cylinder} of height $h$ and circumference $\ell$ on a $k$-differential is an isometrically embedded copy of $[0,\ell] \times (0,h) / \sim$ where $(0,y) \sim (\ell, y)$ for all $y$. The boundary of a cylinder is a union of saddle connections. A cylinder is \emph{simple} if each boundary component is a single saddle connection that appears exactly once on the boundary\footnote{The second part of the condition rules out envelopes.}. In Section \ref{S:Cylinder}, we build on Theorem \ref{T:MZ} to show the following. 

\begin{thm}\label{T:Cyl}
Every component of strata of positive genus surfaces contains a surface with a simple Euclidean cylinder. The same holds for $\Omega^k \cM_0(\kappa)$ if and only if $\kappa$ can be partitioned into two subsets summing to $-k$ neither of which is $(-\frac{k}{2}, -\frac{k}{2})$. 
\end{thm}

In genus zero, when $\kappa$ admits a partition into two subsets summing to $-k$ and one of them is $(-\frac{k}{2}, -\frac{k}{2})$, the stratum is either $\Omega^2 \cM_0(-1^4)$ or contains a surface with a simple Euclidean envelope. Theorem \ref{T:Cyl} is subtler than one might expect; Tahar \cite{Tahar-Chamber} exhibited components of strata of $k$-differentials with open subsets in which a Euclidean cylinder exists and open subsets in which no Euclidean cylinder exists. Another subtlety is that the result is false in general in genus zero. Theorem \ref{T:Cyl} may be the closest analogue of Masur's theorem \cite{Masur-cylinder} on the existence of cylinders that holds for $k$-differentials when $k>2$.

We are interested in simple Euclidean cylinders for the following reason. A simple cylinder either has a unique zero on both boundaries or distinct singularities on opposite boundaries. Collapsing the cylinder, i.e. sending its height and length of a cross curve to zero, splits the zero or merges the two singularities respectively. Call these deformations \emph{splits} and \emph{merges} respectively. See Figure \ref{F:SplitMerge} for the resulting stable curves. Showing that components of strata of positive genus $k$-differentials have splits or merges in their boundary is the only non-algebraic part of the proof of Theorem \ref{T1}.

\begin{figure}[h]
    \centering
    \begin{subfigure}{0.46\linewidth}
        \centering
        \includegraphics[width=\linewidth]{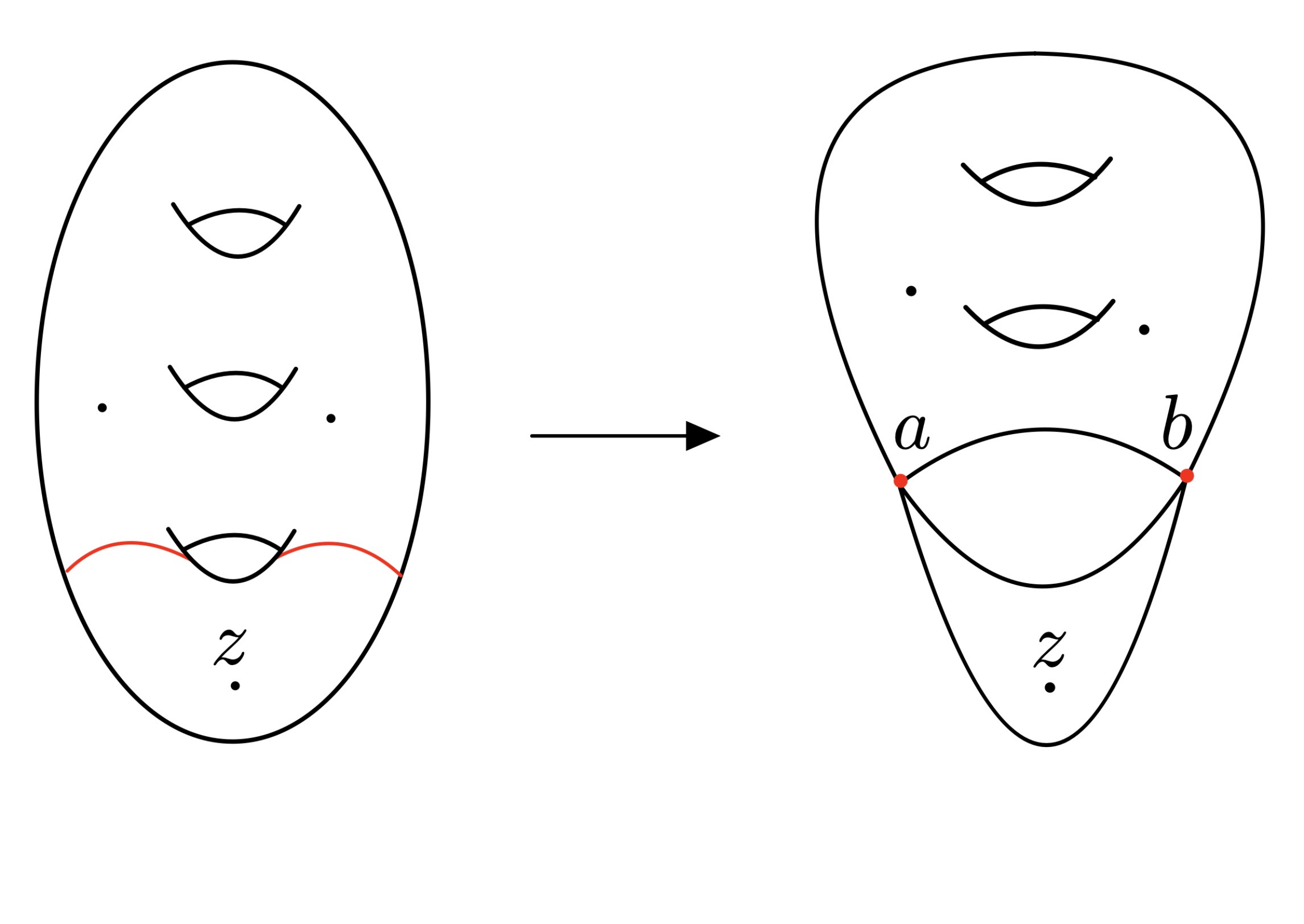}
        \caption{Splitting the zero $z$.}
        \label{F:Split}
    \end{subfigure}
    \hfill
    \begin{subfigure}{0.43\linewidth}
        \centering
        \includegraphics[width=\linewidth]{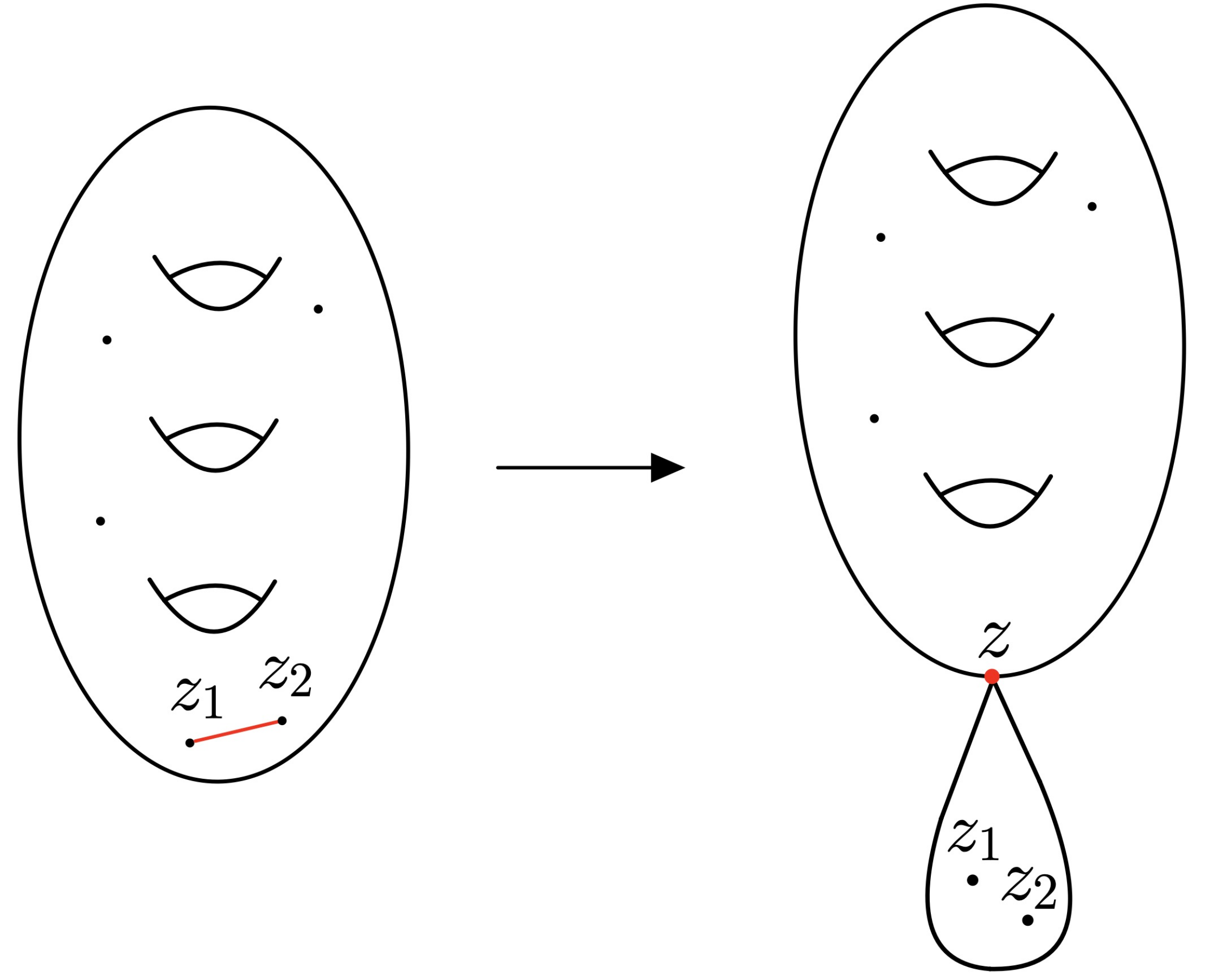}
        \caption{Merging singularities $z_1$ and $z_2$.}
        \label{F:Merge}
    \end{subfigure}
    \caption{Splitting and merging. In a split, $z \geq 2$, $a+b=z-2k$, and $a,b > 1-k$. For a merge, $z = z_1 + z_2$.}
    \label{F:SplitMerge}
\end{figure}

Each irreducible component of a split or merge is equipped with a nonzero $k$-differential. Smoothing the nodes to return to a smooth $k$-differential requires a \emph{prong matching}, which is recalled in Section \ref{S:BCGGM}. Two prong matchings are \emph{component-equivalent} if the resulting smooth $k$-differentials are in the same component. For a merge, all prong matchings are equivalent. For a split, the situation is subtler.

\begin{thm}\label{T:ProngMatching}
Consider a split that produces a component $\cC$ of $\Omega^k \cM_g(a,b, \kappa)$. There is a homomorphism $\rho: \pi_1(\cC) \ra \bZ/d$ where $d = \gcd(|k+a|, |k+b|)$ and a surjection from its cokernel to the set of component-equivalence classes of prong matchings.
\end{thm}

In Section \ref{S:ProngMatchings}, we determine the cokernel of $\rho$ when $\cC$ consists of genus one surfaces. In Section \ref{S:Merge}, we use Theorem \ref{T:Cyl} to show that, in any component of a stratum of positive genus surfaces, any two zeros and any two poles can be merged while preserving nonhyperellipticity and primitivity. At that point, we will have a powerful inductive engine to classify components in Section \ref{S:ProofMainTheorem}.

\subsection{An open question}

Let $\cF_{g,n}$ (resp. $\cF_{g,n}(\kappa)$) be the moduli space of flat cone metrics on genus $g$ surfaces with $n$ cone points (resp. whose cone angles are given by $\kappa$). Troyanov \cite{troyanov} showed that $\cF_{g,n}(\kappa)$ can be identified with $\cM_{g,n}$ (up to scaling). The \emph{holonomy map} $\mathrm{Hol}: \cF_{g,n} \ra H^1(\Sigma_{g,n}; S^1)$ sends a flat surface to its holonomy homomorphism. It is a submersion away from the trivial homomorphism by Veech \cite{veech93}, in the finite area case, and by Apisa-Bainbridge-Wang \cite{ABW2}, in general. Fibers of this map were studied by Deligne-Mostow \cite{DM} and Thurston \cite{thurstonshapes} in genus zero and Veech \cite{veech93} in general. The preimage of $H^1(\Sigma_{g,n}; \bZ/k)$ is the set of $k$-differentials. So, from a holonomy perspective, $k$-differential are ``rational points'' in $\cF_{g,n}$. The asymptotic formula in Corollary \ref{C:Asymptotics} suggests an obvious analogue for ``irrational points,'' i.e. by setting $O(1/(k(g+1)))$ to $0$. It would be interesting to know when it holds.

\text{}

\noindent \textbf{Acknowledgements.} The first author was partially supported by NSF grant no. DMS-2304840. The second author was partially supported by the Simons Dissertation Fellowship. The second author thanks her advisor Ben Dozier for his ongoing support.

\section{Framings, spin, and the Arf invariant}\label{S:Arf}

Recall that $\Sigma_{g,n}$ is an oriented topological surface of genus $g$ with $n > 0$ points removed. Let $UT \Sigma_{g,n}$ be its unit tangent bundle after equipping the surface with a metric. This is a fiber bundle $S^1 \ra UT\Sigma_{g,n} \ra \Sigma_{g,n}$ which produces the following short exact sequence,
\[ 0 \ra H^1(\Sigma_{g,n}) \ra H^1(UT \Sigma_{g,n}) \ra H^1(S^1) \ra 0. \]
A \emph{$k$-framing} is an element of $H^1(UT\Sigma_{g,n})$ that sends the positively oriented generator of $H^1(S^1)$ to $k$. The set of $k$-framings is an $H^1(\Sigma_{g,n})$-torsor. We will also allow ourselves to use the terminology of $k$-framings when the coefficient group is any cyclic group.

Suppose that $(X, \omega)$ is a $k$-differential with a basepoint $p$ that is not a singularity. Fix a unit tangent vector $v$ at $p$. Let $\gamma: S^1 \ra X$ be a unit-speed smooth arc that begins at $p$. For $t \in [0,1]$, where $S^1 = [0,1]/(0 \sim 1)$, let $v(t)$ be the result of parallel transporting $v$ to $\gamma(t)$. The \emph{winding number} of $\gamma$ is then $WN(\gamma) = \frac{1}{2\pi} \int_0^1 d\angle( \gamma'(t), v(t)) dt$. It is straightforward to check that $WN$ takes values in $\frac{1}{k} \bZ$ and is \emph{twist-linear}, i.e. $WN(T_d c) = WN(c) + \langle c, d \rangle WN(d)$ for any simple closed curves $c$ and $d$ where $T_d$ is the Dehn twist about $d$ and where $\langle c, d \rangle$ is the intersection pairing between $c$ and $d$.

\begin{lem}
The function $k \cdot WN$ is a $k$-framing.
\end{lem}
\begin{proof}
By Humphries-Johnson \cite[Theorem 2.5]{HumphriesJohnson}, any twist-linear map from simple closed curves to $\bZ$ is given by an element of $H^1(UT \Sigma_{g,n})$.
\end{proof}

Call $w := k \cdot WN$ the \emph{$k$-framing associated to $(X, \omega)$}. Note that $w$ mod $k$ is the holonomy homomorphism. If a $k$-differential has a cone point $q$ of order $N$ and $\delta_q$ is a small positively oriented loop around $q$, then $w(\delta_q) = N+k$. 

\begin{lem}[Arf]
If $w$ is a $\bZ/2$-valued $1$-framing on a closed genus $g$ surface $\Sigma_g$ and $(a_1, \hdots, a_g, b_1, \hdots, b_g)$ is a symplectic basis of simple closed curves, then $\sum_{i=1}^g (w(a_i)+1)(w(b_i)+1)$ mod $2$ is independent of the choice of basis.
\end{lem}
\begin{proof}
Since any two bases are in the same mapping class group orbit, it suffices to show that the expression is invariant under a generating set of mapping classes. By twist-linearity, the claim is clear when the mapping class is a Dehn twist about an element of the symplectic basis. For $i \in \{1,\hdots, g-1\}$ let $c_i$ be a simple closed curve whose geometric intersection number with $a_i$ is $1$, with $a_{i+1}$ is $-1$, and with any other element of the symplectic basis is $0$. So $b_i, b_{i+1}, c_i$ form a pair of pants (see Figure \ref{F:Generators}). It follows that $w(c_i) + w(b_i) + w(b_{i+1}) = 1$. (This property is called \emph{homological coherence of a $1$-framing}). Dehn twisting by $c_i$ causes the expression in the claim to change by $w(c_i)(w(b_i)-w(b_{i+1}))$. Either $w(c_i) = 0$ or $w(c_i) = 1$ and $w(b_i) + w(b_{i+1}) = 0$ so the change is $0$. The two types of Dehn twists we described account for all Humphries generators of the mapping class group. See Farb-Margalit \cite[Figure 4.10 and Corollary 4.16]{FarbMargalit-Primer}. 
\end{proof}

The invariant in the previous lemma is called the \emph{Arf invariant}. Given a $k$-framing $w$ where $k$ is odd and $w(\delta_q)$ is odd for all punctures $q$, we have that $w$ mod $2$ is a $1$-framing on a closed surface. The \emph{spin} of the $k$-framing is the Arf invariant of $w$ mod $2$. The spin of a $k$-differential is the spin of its associated $k$-framing. 

There is also a variant of the Arf-invariant for a $\bZ/2$-valued $1$-framing $w$ on a surface $\Sigma_{g,2}$ where $w(\delta_q) = 0$ for $q$ either of the two marked points. As in Apisa-Salter \cite[Section 2.3]{ApisaSalter}, $w$ may be extended to a relative framing $\overline{w} \in H^1(UT\Sigma_g, P; \bZ/2)$ where $P$ is the set of two punctures. The \emph{relative Arf invariant (of $\overline{w})$} is then $\overline{w}(s) + \sum_i (w(a_i)+1)(w(b_i)+1)$ where $s$ is a simple arc joining the two points in $P$ and $(a_1, \hdots, a_g, b_1, \hdots, b_g)$ is a symplectic basis of simple closed curves in its complement. This was originally defined by Randal-Williams \cite[Proposition 2.8]{Randal-Williams-RelativeArf}.

\begin{lem}\label{L:RelativeArf}
The relative Arf invariant only depends on $\overline{w}$.
\end{lem}
\begin{proof}
We need only consider the one additional Humphries generator. This is the Dehn twist $T_\gamma$ about a simple closed curve whose geometric intersection number with $s$ and $a_1$ is $1$ and is $0$ with all other elements of the basis. In particular, we take $\gamma$ to be the result of pulling $b_1$ across the puncture $q$. The change of the proposed invariant is then by $w(\gamma) + w(\gamma)(w(b_1)+1) = w(\gamma)w(b_1)$.  Since $(\gamma, b_1, \delta_q)$ form a pair of pants, we're done by homological coherence. See Figure \ref{F:Generators}.
\end{proof}
\begin{figure}[H]
    \centering
    \includegraphics[width=.55\linewidth]{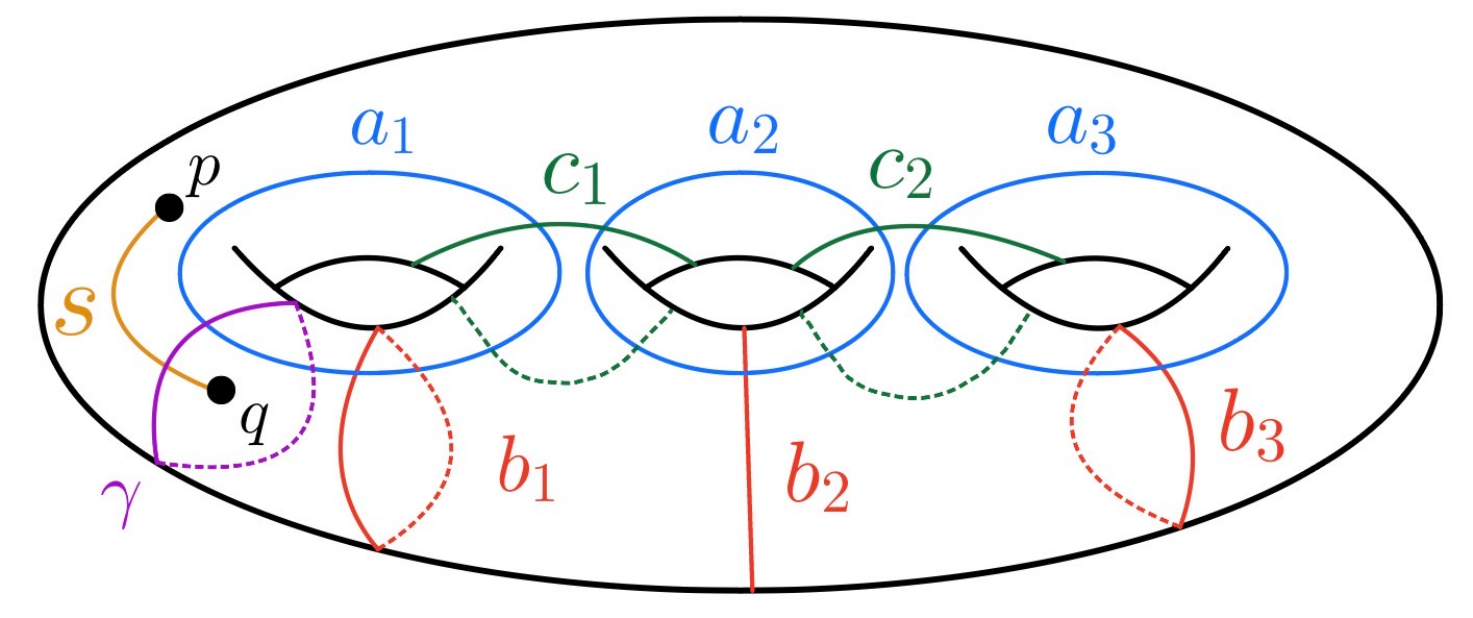}
    \caption{A collection of simple closed curves from Lemma \ref{L:RelativeArf} whose corresponding Dehn twists span the mapping class group of $\Sigma_{g,2}$.}
    \label{F:Generators}
\end{figure}
The construction of $\overline{w}$ involves blowing up the punctures and making a choice of a point and tangent vector on the resulting boundary circles, say $(p_1, v_1)$ and $(p_2, v_2)$. When $w$ comes from the winding number of a $k$-differential, the choices are made so that a smooth arc $\gamma : [0,1] \ra \Sigma_{g,2}$ with $\gamma(0) = p_1, \gamma'(0) = v_1, \gamma(1) = p_2$, and $\gamma'(1) = -v_2$ has a winding number in $\frac{1}{k} \bZ$. Since the extension $\overline{w}$ depends on $(p_1, v_1)$ and $(p_2, v_2)$ it is typically noncanonical. But when $w$ comes from the $1$-framing associated to a $1$-differential where the two punctures are simple poles and there are no other poles, there are canonical choices of $v_1$ and $v_2$ (and any choice of basepoints $p_1$ and $p_2$ determine the same extension). In particular, neighborhoods of the simple poles are infinite Euclidean cylinders whose boundaries are cylinder core curves. Choose $v_i$ to be normal vectors to these core curves pointing into the surface. The assumptions force all choices of $\gamma$ from $v_1$ to $v_2$ as above to have integral winding numbers. So the relative Arf invariant is an invariant of a component of $\Omega^1 \cM_g(\kappa, -1, -1)$ where $\kappa$ is an even positive partition of $2g$. This was described in different terms by Boissy \cite{Boissy-Components}. 

\section{Background}\label{S:BCGGM}

Fix integers $g \geq 0$, $k > 0$, and a partition $\kappa$ of $k(2g-2)$ into $n$ integers. The main tool in the sequel is the \emph{multiscale compactification} $\overline{\Omega^k \cM_g(\kappa)}$ of a stratum of $k$-differentials. This was introduced in \cite{BCGGM-k-diff} and further studied in \cite{CMZ}. 

The boundary points of this compactification, called \emph{multiscale $k$-differentials}, are nodal Riemann surfaces in $ \overline{\cM_{g,n}}$ equipped with $k$-differentials whose singularities only occur at the nodes and the $n$ marked points and which must satisfy the following conditions. First, the sum of the orders of vanishing of the $k$-differentials on both sides of a node must be $-2k$. The node is called \emph{vertical} if one side, called the \emph{higher level side}, has strictly larger order of vanishing than the other. Second, the notion of being at a higher level must define a partial order, called a \emph{level structure}, on irreducible components. Third, part of the data of a multiscale $k$-differential is a combinatorial object associated to each vertical node called a \emph{(local equivalence class of a) prong matching}, which we will define in a moment. Finally, a multiscale $k$-differential must satisfy the \emph{global $k$-residue condition} and the \emph{matching $k$-residue condition}. The matching $k$-residue condition is vacuous unless there are horizontal nodes and the global $k$-residue condition is vacuous unless there is a locally maximal irreducible component which is a $k$th power of a $1$-differential whose only poles occur at nodes. We will exploit these vacuities to mostly ignore the residue conditions. 

%

\begin{thm}[Costantini-M\"oller-Zachhuber \cite{CMZ}]
The multiscale compactification is a smooth projective variety.
\end{thm}

We will use the following result of Chen \cite[Theorem 1.1]{Chen-AffineGeometry} for strata of $k$-differentials with some pole of order at least $k$ and Aygun \cite[Lemma 3.9]{Aygun2025_PrimeOrderkDifferentials} for the remaining strata.

\begin{thm}[Chen, Aygun]\label{T:Chen-Aygun}
Each component of a stratum either projects to a point or to a noncompact subset of $\cM_{g,n}$.
\end{thm}

Before proceeding we note that, given a $k$-differential on a Riemann surface $X$ and a point $p$ that is not a singularity, the $k$-differential determines a nonzero map $T_p X \ra \bC$ that is the composition of a linear map to $\bC$ and the self-map of $\bC$ that sends $z$ to $z^k$. The $k$ preimages of $1$ are called the \emph{horizontal unit tangent vectors} at $p$. 

We will now define a prong matching at a vertical node. Let $a$ be the order of the node at the higher level. Deleting balls $B_+$ and $B_-$ around the node on the higher and lower level sides respectively produces a (possibly disconnected) surface whose two boundary components can be identified with circle sectors that sweep out an angle of $2\pi(k+a)$. A \emph{prong matching at a node} is the choice of normal horizontal vectors on $\del B_{\pm}$ that point into the surface. There are $k+a$ choices on each boundary component. A \emph{prong matching} is then a choice of prong matching at all vertical nodes. 

Given a multiscale $k$-differential, multiplying all $k$-differentials at a given level by a $k$th root of unity, i.e. rotating the $k$-differential by some multiple of $\frac{2\pi}{k}$, acts by permutations on prong matchings. If there are $L$ levels, this defines an action of $(\bZ/k)^L$ on prong matchings. Two prong matchings are called \emph{locally equivalent} if they are in the same orbit of this action. 

If $X$ is a multiscale $k$-differential with two irreducible components connected at two vertical nodes and whose $k$-differentials live in strata $\cH_+$ and $\cH_-$, then $\pi_1(\cH_+ \times \cH_-)$ acts on the set of prong matchings. Two prong matchings will be called \emph{globally equivalent} if they are in the same orbit. 

\begin{lem}\label{L:ProngsAndComps}
Smoothing two globally equivalent prong matchings produces two $k$-differentials in the same component of a stratum.
\end{lem}
\begin{proof}
Since the multiscale compactification of a stratum is smooth, the path between the two multiscale $k$-differentials in the boundary can be homotoped to a closed loop in a unique component of the stratum. 
\end{proof}

We will show the following strengthening of Theorem \ref{T:ProngMatching}.

\begin{thm}\label{T:ProngHom}
If the top-level $k$-differential of $X \in \cH_+ \times \cH_-$ has order $a$ and $b$ at the nodes, then there is a \emph{prong matching homomorphism} $\rho: \pi_1(\cH_+ \times \cH_-) \ra \bZ/\gcd(k+a, k+b)$ so that the cokernel of $\rho$ surjects onto the set of global equivalence classes of prong matchings. Moreover, $\rho$ is a change of winding number homomorphism.
\end{thm} 

Set $\delta := \gcd(k+a, k+b)$. We will first construct homomorphisms $\pi_1(\cH_\pm) \ra \bZ/\delta$ whose sum will be $\rho$. Suppose that surfaces in $\cH_+$ have genus $h$ and $m$ singularities. If $w$ is the $k$-framing of a $k$-differential, then $w$ mod $\delta$ gives winding number $0$ to loops around the nodes. As in Section \ref{S:Arf}, form the relative framing $\overline{w} \in H^1(UT\Sigma_{h,m-2}^2, P; \bZ/\delta)$, where we blow up the two singularities $a$ and $b$ and mark two points $P$ on the resulting boundary circles. The change of framing cocycle $A(f) := f^* \overline{w} - \overline{w}$ is a cocycle $A: \mathrm{Mod}_{h,m} \ra H^1(\Sigma_{h,m-2}^2, P; \bZ/\delta)$. (See Apisa-Salter \cite[Construction A.17]{ApisaSalter}). 

\begin{lem}
The restriction of $A$ to $\pi_1(\cH_+)$ is a \emph{(change of winding number) homomorphism} to $\widetilde{H}^0(P, \bZ/\delta) \cong \bZ/\delta$.
\end{lem}
\begin{proof}
Identifying the points in $P$ with their image on the surface, there is a short exact sequence
\[ 0 \ra \widetilde{H}^0(P; \bZ/\delta) \ra H^1(\Sigma_{h,m-2}^2, P; \bZ/\delta) \ra H^1( \Sigma_{h,m-2}^2; \bZ/\delta) \ra 0 \]
Composing $A(f)$ with the second map is $(f^*w - w)$ mod $\delta$. But elements of $\pi_1(\cH_+)$ stabilize the framing $w$. So the restriction of $A$ to $\pi_1(\cH_+)$ is valued in $\widetilde{H}^0(P; \bZ/\delta)$. Since $\mathrm{Mod}_{h,m}$ acts trivially on $\widetilde{H}^0(P, \bZ/\delta)$, the restriction of $A$ to $\pi_1(\cH_+)$ is a homomorphism. 
\end{proof}

\begin{proof}[Proof of Theorem \ref{T:ProngHom}:] The set of cyclically ordered normal vectors (that point into the surface) near the node on the higher (resp. lower level component) whose order is $a$ at the higher level admits a simply transitive order-preserving $\bZ/(k+a)$ action. Prong matchings at this node correspond to order preserving maps between these two sets, which again admits a simply transitive order-preserving $\bZ/(k+a)$ action. So the collection of all prong matchings is a $G := \bZ/(k+a) \times \bZ/(k+b)$ torsor. Local equivalence corresponds to orbits of the subgroup of $G$ generated by $(1,1)$. This subgroup is precisely the kernel of the map $\epsilon: G \ra \bZ/\delta$ where $\epsilon(x,y)=x-y$. The prong homomorphism $\pi_1(\cH_+ \times \cH_-) \ra G$ is the homomorphism that records how traveling along a loop acts on prong matchings. Let $\rho: \pi_1 \ra \bZ/\delta$ be its composition with $\epsilon$. 

Mark any two normal horizontal vectors at the two nodes on the higher level component and identify these vectors with $0$ in $\bZ/(k+a)$ and $\bZ/(k+b)$. The set of outward pointing normal horizontal vectors at the two nodes can now be identified with $\bZ/(k+a) \times \bZ/(k+b)$. Write the vectors corresponding to $(c,d)$ as $(v_c, w_d)$. By fixing two ``reference prongs,'' i.e. normal vectors on the other sides of the nodes on the lower level component, prong matchings are completely specified by the pairs in $\bZ/(k+a) \times \bZ/(k+b)$ that are matched with the reference prongs.

Following a loop $f$ in $\pi_1(\cH_+)$ sends $(0,0)$ to $(x,y)$, while fixing the reference prongs. So the prong homomorphism is $\rho(f) := x-y$ mod $\delta$. Let $s:[0,1] \ra X$ be any smooth arc so that $s'(0) = v_0$ and $s'(1) = -w_0$. Transporting the arc along the loop determined by $f$ while demanding that the derivative at $0$ and $1$ be an outward (resp. inward) pointing horizontal normal vector determines an arc, up to homotopy rel endpoints, $t: [0,1] \ra X$ that has the same winding number as $s$ and so that $t'(0) = v_x$ and $t'(1) = -w_y$. If $s_x$ (resp. $s_y$) are positively (resp. negatively) oriented circular arcs joining $v_0$ to $v_x$ (resp. $w_y$ to $w_0$) then the change of winding number cocycle is
\[ A(f) = w(s_x) + w(t) + w(s_y) - w(s) = x-y = \rho(f). \]
See Figure \ref{F:ProngHom}. 
\end{proof}

\begin{figure}[h]
    \centering    \includegraphics[width=.5\linewidth]{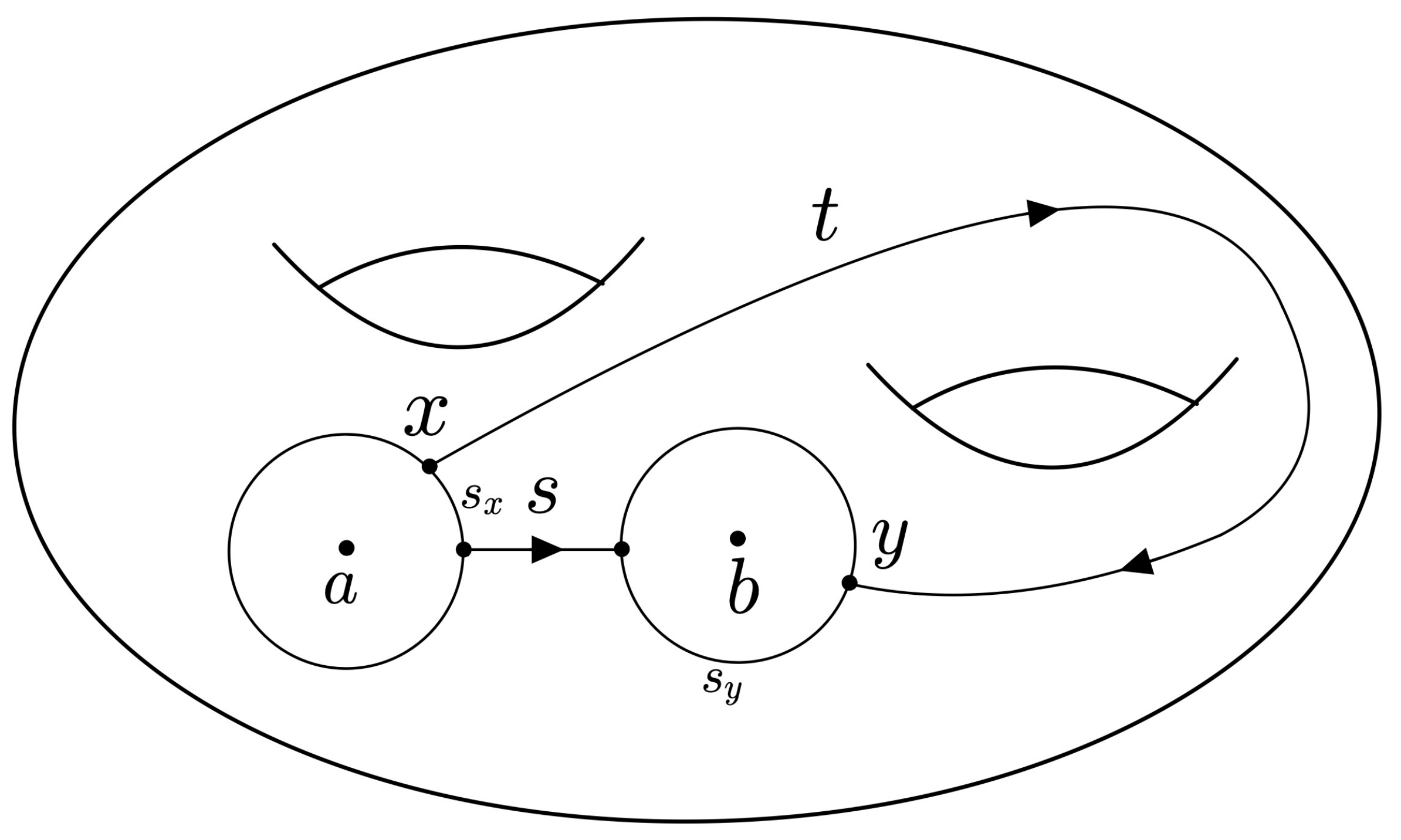}
    \caption{An illustration of the proof of Theorem \ref{T:ProngHom}.}
    \label{F:ProngHom}
\end{figure}

\section{Genus one components and merging}\label{S:GenusOneComponents}

The goal of this section is to classify components of strata of genus one surfaces and to study their boundary in the multiscale compactification. Most of these results are originally due to Boissy \cite{Boissy-Components} and Chen-Gendron \cite{ChenGendron}.

If $E$ is an elliptic curve and $dz$ is its invariant differential, then every $k$-differential on $E$ can be written as $f(z)dz^k$ where $f:E \ra \bP^1$ is some meromorphic function. If we fix the orders $\kappa := (k_i)_{i=1}^n$ of zeros and poles that prescribe a stratum, then the fiber of this stratum over $E \in \cM_1$ will be $\{ (p_i)_{i=1}^n \in E^n \setminus \Delta : \sum_i k_i p_i = 0 \}$ where $\Delta$ is the sublocus of $E^n$ where two points coincide. 

\begin{lem}\label{L:ComponentGenus1}
Let $A$ be an abelian group and $0 \ne \kappa := (k_i)_{i=1}^n$ integers summing to $0$ and so $\gcd(\kappa) = d$. Then $B := \{ (a_i)_{i=1}^n \in A^n : \sum_i k_i a_i = 0\} \cong A^{n-1} \times A[d]$, where $A[d]$ is the $d$-torsion of $A$. 
\end{lem}
\begin{proof}
Consider the homomorphism $s: \bZ^n \ra \bZ$ that sends $(a_i)_{i=1}^n \in \bZ^n$ to $\sum_i\frac{k_i}{d} a_i$. It is a surjection since $\gcd(\kappa) = d$ and its kernel is therefore a free $\bZ$-module of rank $n-1$. So there is a short exact sequence,
\[ 0 \ra \bZ^{n-1} \ra \bZ^n \xra{s} \bZ \ra 0. \]
Since the sequence splits, tensoring with $A$ preserves exactness, so
\[ 0 \ra A^{n-1} \ra A^n \xra{s} A \ra 0 \]
is split exact. Finally, $B$ is the preimage of $A[d]$.
\end{proof}

Let $\Delta_{ij} \subseteq A^n$ be the subgroup of $n$-tuples $(a_\ell)_{\ell=1}^n \in A^n$ so that $a_i = a_j$. Let $\kappa_{ij}$ be the tuple formed by deleting $k_i$ and $k_j$ from $\kappa$ and adding in $k_i + k_j$. 

\begin{cor}\label{C:ComponentGenus1Pt2}
For $B$ as in Lemma \ref{L:ComponentGenus1}, $B \cap \Delta_{ij} \cong A^{n-2} \times A[e]$ where $e = \gcd(\kappa_{ij})$ if $\kappa_{ij} \ne 0$ and $B \cap \Delta_{ij} = \Delta_{ij}$ otherwise.
\end{cor}
\begin{proof}
Without loss of generality, let $\{i,j\} = \{n-1,n\}$. Then,
\[ B \cap \Delta_{ij} = \left\{ (a_1, \hdots, a_{n-2}, a_{n-1}, a_{n-1}) \in A^n : \left( \sum_{i=1}^{n-2} k_i a_i \right) + (k_{n-1} + k_n)a_{n-1} = 0 \right\}. \] 
The claim now follows from Lemma \ref{L:ComponentGenus1}.
\end{proof}

The following was shown by Boissy \cite[Theorem 1.1]{Boissy-Components} when $k =1$ and Chen-Gendron \cite[Theorem 3.12]{ChenGendron} in general. 

\begin{prop}\label{P:Genus1Components}
The components of $\Omega^k \cM_1(\kappa)$ correspond to divisors of $\gcd(\kappa)$ except when $\kappa$ has two elements, in which case they correspond to divisors greater than $1$.
\end{prop}

Let $\Omega^k_{d/e} \cM_1(\kappa)$ be the component corresponding to $e \mid d := \gcd(\kappa)$. The number $d/e$ is called the \emph{rotation number}.

\begin{rem}\label{R:RotationNumberDefinition}
Given an elliptic curve $E$ with a $k$-differential with divisor $\sum_i k_i p_i$, $e = \mathrm{ord}\left( \sum_i \frac{k_i}{d} p_i \right)$. 
\end{rem}

\begin{proof}
Set $d = \gcd(\kappa)$. By Lemma \ref{L:ComponentGenus1}, $\Omega^k \cM_1(\kappa)$ can be completed to a fiber bundle over $\cM_1$ with fiber over $E \in \cM_1$ being $E^{n-1} \times E[d]$. Since $\mathrm{SL}(2, \bZ)$ acts transitively on the elements of $E$ of order exactly $e$, for any integer $e$, the components of the completion are indexed by divisors of $d$. To get the original stratum, we must delete elements of $\Delta$ in each fiber. By Corollary \ref{C:ComponentGenus1Pt2}, this is a positive codimension locus, and hence deleting it does not alter the components, except when $n = 2$ and, hence $\kappa_{12} = (0)$, in which case $\Delta$ is a component of $B$.
\end{proof}

\begin{rem}\label{R:GenusOneComplexAnalyticView}
An element of $\Omega_{d/e}^k \cM_1(\kappa)$ on $E \in \cM_1$ can be written as $f(z)^{d/e} dz^k$ where $f: E \ra \bP^1$ is meromorphic. In particular, the component is primitive if and only if $\gcd(k, \frac{d}{e}) = 1$.
\end{rem}

\begin{lem}\label{L:RotationNumberImplication}
If a $k$-differential in $\Omega^k \cM_1(\kappa)$ belongs to $\Omega^k_{d/e} \cM_1(\kappa)$, then the associated $k$-framing sends a simple closed curve that is nonseparating on the closed surface to a multiple of $d/e$.
\end{lem}
\begin{proof}
As in Remark \ref{R:GenusOneComplexAnalyticView}, the $k$-framing sends a simple closed nonseparating curve $\gamma$ to the degree of  $\frac{f(\gamma(t))^{d/e} \gamma'(t)^k}{|f(\gamma(t))^{d/e} \gamma'(t)^k|}$. If $g_1, g_2: S^1 \ra S^1$, then $\deg(g_1 \cdot g_2) = \deg(g_1) + \deg(g_2)$. Since $\gamma$ is simple closed and nonseparating the degree of $\frac{\gamma'(t)}{|\gamma'(t)|}$ is the winding number of $\gamma$ in the flat $dz$ metric, which is $0$. So the framings associated to $f(z)^{d/e} dz^k$ and $(f(z) dz)^{d/e}$ send $\gamma$ to the same value. The result now follows. 
\end{proof}

Say that two singularities $a$ and $b$ in a component of a stratum can be \emph{merged} if it is possible to make them coincide without pinching a curve other than the one which bounds a disc containing just $a$ and $b$.

\begin{prop}\label{P:GenusOneMerging}
The stratum $\Omega_r^k \cM_1(\kappa)$ contains $\Omega_{r'}^k \cM_1(\kappa_{ij})$ on its boundary where $r \mid r' \mid \gcd(\kappa_{ij})$ provided that the latter stratum is nonempty and $\kappa_{ij} \ne 0$. In particular, any two singularities can be merged except:
\begin{enumerate}
    \item The two singularities when $\kappa = (a,-a)$.
    \item The $a$ and $b$ singularities when $\kappa = (a,b,c)$ and $r = |c|$. 
    \item The $a$ and $-a$ singularities when $\kappa = (a,-a,b,-b)$ and $r = |b|$.
\end{enumerate}
\end{prop}
\begin{proof}
Suppose that $n > 2$. Set $d = \gcd(\kappa)$ and write $r = d/e$. Fix some surface $E \in \cM_1$. Write $\overline{k_i} := \frac{k_i}{d}$ and fix $g \in E$ of order $e$. It suffices to consider $B_g := \{ (a_1, \hdots, a_n) \in E : \sum_i  \overline{k_i} a_i = g \}$. Without loss of generality $\{i,j\} = \{n-1,n\}$. Then,
\[ B_g \cap \Delta_{ij} = \left\{ (a_1, \hdots, a_{n-2}, a_{n-1}, a_{n-1}) \in E^n : \left( \sum_{i=1}^{n-2} \overline{k_i} a_i \right) + (\overline{k_{n-1}} + \overline{k_n})a_{n-1} = g \right\}. \]
We identify this subset with
\[ C := \left\{ (a_1, \hdots, a_{n-2}, a_{n-1}) \in E^{n-1} : \left( \sum_{i=1}^{n-2} \overline{k_i} a_i \right) + (\overline{k_{n-1}} + \overline{k_n})a_{n-1} = g \right\}. \]
Since $n > 2$, $\kappa_{ij} \ne 0$. Write $ \gcd(\kappa_{ij}) = dm$. By Lemma \ref{L:ComponentGenus1}, the components of $C$ correspond to elements $g'$ of $E$ where $mg' = g$, which is a coset of $E[m]$. If $|g'| = em'$ where $m' \mid m$, then this component corresponds\footnote{This follows since multiscale $k$-differentials on nodal Riemann surfaces of genus $g$ with $n$ singularities and nodes are recoverable from the underlying element of $\overline{\cM_{g,n}}$ with marked points and nodes labeled by their orders. If a component of $C$ were contained in $\Delta_{\ell h}$, then either three singularities simultaneously collided while fixing the underlying elliptic curve or two pairs of singularities did. The first case does not specify a unique point in $\overline{\cM_{g,n}}$. In both cases, a curve other than one bounding a disk containing $a$ and $b$ is pinched.} to $\Omega^k_{rm/m'} \cM_1(\kappa_{ij})$ (see the proof of Proposition \ref{P:Genus1Components}) provided that the component is not contained in $\Delta_{\ell h}$ for some indices $\ell$ and $h$. If $C$ contains such a component then, for that component, $g = g' = 0$ and $\kappa_{ij}$ contains exactly two nonzero coefficients, which are indexed by $\ell$ and $h$. This would correspond to $\Omega_{\gcd(\kappa_{ij})}^k \cM_1(\kappa_{ij})$, which is empty. This establishes the first claim and the second claim is now straightforward to check. 
\end{proof}

\begin{cor}\label{C:GenusOneMerging}
In primitive nonhyperelliptic components of genus one strata of $k$-differentials without marked points, it is possible to merge any two zeros or any two poles while preserving primitivity and nonhyperellipticity. 
\end{cor}
\begin{proof}
By Proposition \ref{P:GenusOneMerging}, merging two zeros or two poles while preserving rotation number, and hence primitivity, is always possible. We claim that these merges preserve nonhyperellipticity. The hyperelliptic components of genus one strata are 
\[ \Omega^k_r \cM_1(r,r,-r,-r), \Omega^k_r \cM_1(\pm 2r, \mp r, \mp r), \text{ and } \Omega_r^k \cM_1(2r, -2r).\] 

Since our merges of two zeros or two poles preserve rotation number $r$, the merged singularity will have order $o$ where $|o| \geq 2r$. If the resulting surface lived in a hyperelliptic component, then $o$ is a merge of two singularities whose orders are both $r$ or both $-r$. So merging to a hyperelliptic component implies that the original component was also hyperelliptic.
\end{proof} 

\section{Rotation number and genus one prong matchings}\label{S:ProngMatchings}

The projectivization of $\Omega^k_r \cM_1(a,-a)$ is $\{(E,0,p) \in \cM_{1,2} : |p-0| = e\}$ where $\frac{a}{e} = r$. This is the modular curve $\bH/\Gamma_1(e)$ where $\Gamma_1(e)$ is the subgroup of matrices in $\mathrm{SL}(2, \bZ)$ that are unipotent upper triangular mod $e$. Let $\delta_q \in H_1(UT\Sigma_{1,2}; \bZ)$ be a positively oriented loop around $q$, where $q$ is $0$ or $p$, and let $f$ be the class of the fiber of $UT\Sigma_{1,2} \ra \Sigma_{1,2}$. When working topologically, we will identify $\Sigma_1$ with the square torus equipped with holomorphic $1$-form $\omega$ and where $p = (1/e, 0)$. Let $\gamma$ be the core curve of the horizontal cylinder and let $s$ be the horizontal path from $0$ to $p$. Intersections with these curves will be written using $\langle \cdot, \cdot \rangle$. They determine linear maps from $H_1(\Sigma_{1,2})$ to $\bZ$. 


\begin{prop}\label{P:WNofSCC}
For $a, e, k \in \bZ$ where $1 < e \mid |a|$, there is a unique $\Gamma_1(e)$-fixed element $w$ of $H^1(UT\Sigma_{1,2}; \bZ)$ where $w(f) = k$ and $w(\delta_0) = k+a$. Moreover, if $\eta$ is a curve in $\Sigma_{1,2}$ whose winding number with respect to $\omega$ is $0$, then $w(\eta) = r \langle \gamma, \eta \rangle - a \langle s, \eta \rangle$. 
\end{prop}
\begin{proof}
Uniqueness is clear since the difference of two such elements is a $\Gamma_1(e)$-invariant element of $H^1(\Sigma_1)$, which must be zero. When $k > 0$, existence amounts to taking the winding number function associated to a $k$-differential in $\Omega^k_{a/e} \cM_1(a, -a)$. When $k < 0$ we can reduce to the $k > 0$ case by multiplying by $-1$. When $k = 0$, Poincare duality reduces the problem to finding invariant elements of $H_1(\Sigma_1, \{0,p\})$. Since $\omega: H_1(\Sigma_1, \{0,p\}; \bZ) \ra \bC$ is $\Gamma_1(e)$-equivariant with a kernel generated by $\gamma - es$, it follows that elements of $\Gamma_1(e)$ send $\gamma - es$ to $\pm(\gamma-es)$. Invariance follows by considering pairings with $\delta_0$. 

For the second claim, note that $w' := w - r(\gamma - es)^*$, where the $*$ denotes the Poincare dual, is a $\Gamma_1(e)$-invariant element of $H^1(UT\Sigma_{1,2}; \bZ)$ where $w'(f) = w'(\delta_0) = w'(\delta_p) = k$. By uniqueness of $w'$, $w'/k$ must be the winding number of the flat metric prescribed by $\omega$, so the result follows.
\end{proof}

The following is due to Boissy \cite{Boissy-Components} when $k=1$ and Chen-Gendron \cite[Proposition 3.13]{ChenGendron} in general.

\begin{prop}
A $k$-differential in $\Omega^k \cM_1(\kappa)$ belongs to $\Omega^k_{r} \cM_1(\kappa)$ if and only if its associated $k$-framing sends the set $\cS$ of simple closed curve that are nonseparating on the closed surface surjectively to $r\bZ$.
\end{prop}
\begin{proof}
The claim that the $k$-framing on $\Omega^k_r \cM_1(\kappa)$ sends $\cS$ into $r \bZ$ is Lemma \ref{L:RotationNumberImplication}. Therefore, we must produce an element of $\cS$ on an element of $\Omega^k_r \cM_1(\kappa)$ that has rotation number exactly equal to $r$. When $\kappa = (-a,a)$ this follows from Proposition \ref{P:WNofSCC}. Using Proposition \ref{P:GenusOneMerging}, all other cases reduce to this one by merging all zeros and, separately, merging all poles, while preserving the invariant $r$.
\end{proof}

\begin{prop}\label{P:ProngHomG1TwoSing}
Given a primitive component $\Omega_r^k \cM_1(a,-a)$, the prong matching homomorphism $\rho: \Gamma_1(e) \ra \bZ/\gcd(k+a, k-a)$ is surjective unless $k$ and $a$ are odd, in which case its image has index $2$.
\end{prop}
\begin{proof}
Fix $g = \begin{pmatrix} \alpha & \beta \\ \epsilon & \zeta \end{pmatrix} \in \Gamma_1(e)$. Let $t$ be a closed curve on $\Sigma_{1,2}$ representing $g(s)s^{-1}$ which has winding number zero in the flat metric associated to $\omega$. Let $w$ be the $k$-framing associated to the $k$-differential. By Theorem \ref{T:ProngHom} (see Figure \ref{F:ProngHom}), $\rho(g) = w(t)$ mod $\gcd(k+a, k-a)$. By Proposition \ref{P:WNofSCC}, this is 
\[ r\langle \gamma, g(s) \rangle - a \langle s, t \rangle = \frac{r\epsilon}{e} - a \langle s, t \rangle. \]
The slope of $g(s)$ is $\frac{\epsilon}{\alpha}$ and $\langle s, t \rangle$ is the number of times that $g(s)$ crosses $s$. This is number of times that $n \frac{\alpha}{\epsilon}$ mod $1$ belongs to $[0, \frac{1}{e})$ for $n \in \{0, \hdots, \frac{\epsilon}{e}-1\}$. Write $m = \frac{\epsilon}{e}$. We have
\[ \langle s, t \rangle = \#\{ n \in \{0, \hdots, m-1\} : 0 \leq n\alpha  \text{ mod } \epsilon  < m \}. \]
Note that lower unipotents are sent to $(r-a)m$ and that $\delta := \gcd(k+a, k-a) = \gcd(k,e)$ unless the same powers of two divide $k$ and $a$, in which case $\gcd(k+a, k-a) = 2\gcd(k,e)$. (This uses that primitivity implies that $\gcd(k,r) = 1$ and that $a = re$). 

When $\delta \mid a$, $\rho$ is a surjection since $\gcd(r,k) = 1$. So suppose now that $\delta \nmid a$, which implies that $r$ is odd and $\delta = 2 \gcd(k,e)$. Since $\rho \mod \frac{\delta}{2}$ is surjective, the image of $\rho$ has index one or two and, in the latter case consists of even numbers. When $a$ is even, $r-a$ is odd so $\rho$ is surjective. 

Finally, let $k$ and $a$ be odd. Since $g$ is in the framed mapping class group and since it preserves the relative Arf invariant (by Lemma \ref{L:RelativeArf}) it must preserve $\overline{w}(s)$ mod $2$ where $\overline{w}$ is the relative framing associated to $w$. So $\rho(g) = \overline{w}(gs) - \overline{w}(s)$ mod $\delta$ is even. 
\end{proof}

\begin{cor}\label{C:GenusOneProngMatchings}
Split a singularity on a genus two $k$-differential to form a surface $X$ in a primitive component $\Omega^k_r \cM_1(a,b,\kappa)$ where $a$ and $b$ are the nodes formed by the split. There is one global equivalence class of prong matchings unless either (1) $k$, $a$, and $b$ are odd and all other singularities have even orders (in which case there are two) or (2) the component is hyperelliptic and the two singularities are exchanged by the hyperelliptic involution (in which case there are $|k+a|$). 
\end{cor}
\begin{proof}
In light of Theorem \ref{T:ProngHom} and Proposition \ref{P:ProngHomG1TwoSing}, suppose that $|\kappa| \geq 1$. Let $\Gamma$ be the image of the fundamental group of the stratum in the mapping class group $\mathrm{Mod}_{1,|\kappa|}$. Let $\rho: \Gamma \ra \bZ/\delta$ be the prong matching homomorphism where $\delta = \gcd(|k+a|, |k+b|)$. We will use the following.

\begin{sublem}\label{SL:MergeFTW}
Let $c \in \kappa$ and set $\kappa' = \kappa \setminus (c)$. If $a$ can be merged with $c$ to form $\Omega^k_r \cM_1(a+c, b, \kappa')$, then the image of $\rho$ contains $c = (k+a+c)$ mod $\delta$. Moreover, $\rho$ has an image with index at most that of the image of the prong homomorphism associated to $\Omega^k_r \cM_1(a+c, b, \kappa')$.
\end{sublem} 
\begin{proof}
Let $\gamma$ be the curve pinched by merging $a$ and $c$. The $k$-framing sends it to $k+a+c$. Let $\overline{w}$ be the relative framing associated to the $k$-framing $w$. The Dehn twist $T$ about this curve is contained in $\Gamma$. By twist-linearity, $\overline{w}(Ts) = \overline{w}(s) + (k+a+c)$. Moreover, by lifting paths in the boundary stratum to the original stratum, the image of $\rho$ mod $c$ contains the image of the prong homomorphism associated to $\Omega^k_r \cM_1(a+c, b, \kappa')$. The second claim in the statement now follows. 
\end{proof}

\noindent \textbf{Case 1: The component is hyperelliptic.} Let $J$ be the hyperelliptic involution on the surface $X$. If $a$ and $b$ are exchanged by the involution, then prong matchings correspond to a choice of two prongs (up to simultaneous rotation) at the singularity corresponding to $a$ on $X/J$. Up to choosing one of the prongs at $a$ on $X/J$.

So suppose without loss of generality that $J(a) \notin \{a,b\}$. This implies that $a = \pm r$ and $b = nr$ where $n \in \{\pm 1, \pm 2\}$. Since we can merge $a$ and $J(a)$, the image of $\rho$ contains $r$ (by Sublemma \ref{SL:MergeFTW}). Since the component is primitive, $\gcd(k,r) = 1$ and so $\gcd(k+nr,k\pm r, r) = 1$ so $\rho$ is surjective. We are done by Theorem \ref{T:ProngHom}. 

\noindent \textbf{Case 2: The component is nonhyperelliptic.} Induct on $|\kappa|$. By Corollary \ref{C:GenusOneMerging}, we can merge any two zeros and poles while retaining primitivity, nonhyperellipticity, and the rotation number. If $\kappa$ contains two zeros or two poles we merge them and conclude. So $\kappa$ has at most one pole and at most one zero.

\noindent \textbf{Case 2a: $a$ or $b$ can be merged with a third singularity $c$ while preserving nonhyperellipticity.} Say we merge $a$ and $c$. By the induction hypothesis, on the merged surface there is only one global equivalence class of prong matching unless $k,a+c,b$ are odd, when there are two. By Sublemma \ref{SL:MergeFTW} we are done unless $k, b, c$ are odd and $a$ is even. However, in that case there is no index two subgroup of $\bZ/\gcd(|k+a|, |k+b|)$. 

\noindent \textbf{Case 2b: Neither $a$ nor $b$ can be merged with a third singularity while preserving nonhyperellipticity.} In this case, $a$ and $b$ have the same sign and $\kappa$ contains exactly one singularity $c$ of the opposite sign. Note that $a$ and $b$ are not both equal to $\pm r$ since then the component would be hyperelliptic. Without loss of generality, suppose that $a \ne \pm r$. By Proposition \ref{P:GenusOneMerging}, we can merge $b$ and $c$ to form $\Omega_r \cM_1(a,-a)$. We are done by Sublemma \ref{SL:MergeFTW} except possibly if $k,a,c$ are odd and $b$ is even. But in that case there is no index two subgroup of $\bZ/\gcd(|k+a|, |k+b|)$
\end{proof} 

\section{Hat-homology and the proof of Theorem \ref{T:MZ}}\label{S:MasurZorich}

As in Masur-Zorich \cite{MZ}, two simple disjoint saddle connections $s$ and $t$ on a $k$-differential $X$ are called \emph{hat-homologous} if they have a ratio of lengths that is generically constant in the stratum. The saddle connections are called \emph{generically parallel} if any two lifts of these saddle connections to the holonomy cover have periods with a locally constant ratio of arguments in the locus of holonomy covers of surfaces in the stratum. 

\begin{lem}
Two simple disjoint saddle connections are generically parallel if and only if they are hat-homologous.
\end{lem}
\begin{proof}
For the forward direction, the ratio of periods is holomorphic and any holomorphic function that has constant argument is constant. For the reverse direction, the ratio of periods is holomorphic and unit modulus, hence it is constant.
\end{proof}

The main theorem of this section is the following, which is a restatement of Theorem \ref{T:MZ}. Note that two saddle connections are called \emph{disjoint} if their interiors are disjoint. 

\begin{thm}\label{T:HatHomology}
Two disjoint simple saddle connections $s$ and $t$ on a $k$-differential $X$ that is not a finite-area translation surface are hat-homologous if and only if $X - s \cup t$ has a component that is a finite area translation surface.
\end{thm}

We will deduce this result from the following one.

\begin{prop}\label{P:HatHomology}
Theorem \ref{T:HatHomology} holds when $X - s \cup t$ is connected.  
\end{prop}
\begin{proof}
Extend $s$ and $t$ to a maximal collection $S'$ of pairwise disjoint saddle connections. Since $X - s \cup t$ is connected, there is a subset $S \subseteq S'$ containing $s$ and $t$ so that $X - S$ is a topological disk with punctures at the set $P$ of poles of order $\leq -k$. These are the singularities for which every neighborhood has infinite area. For each $p \in P$ fix a neighborhood $D_p$ that is a disk that has empty intersection with $S$. Cut out each of these neighborhoods and mark a point $b_p$ on each resulting boundary component. By drawing $|P|$ additional paths $T$ from zeros to $b_p$ that are disjoint from $S$, we have that $X - (S \cup T) - \bigcup_{p \in P} D_p$ is a flat metric on a topological disk with boundary. Such a metric necessarily comes from a holomorphic $1$-form. Letting $n$ be the number of singularities, the boundary of this disk is made up of $2d := 2(2g+n-1)$ arcs in $S \cup T$ together with the arcs $\bigcup_{p \in P} \del D_p$. Positively orient all of the arcs. Since each arc in $S \cup T$ is formed by cutting an arc on $X$, the $2d$ arcs are naturally paired. Let $v_1, \hdots, v_d$ be periods of one selected representative of each pair and choose $\zeta_1, \hdots, \zeta_d \in S^1$ so that the unselected partners of these representatives have periods $-\zeta_1 v_1, \hdots, -\zeta_d v_d$. Let $w_p$ be the period of $\del D_p$ and set $w := \sum_{p \in P} w_p$. 

The only constraint on the periods is that $w + \sum_{i=1}^d (1-\zeta_i)v_i = 0$. This can be seen directly or by using the period coordinates constructed in \cite[Corollary 2.3]{BCGGM-k-diff}. If $v_1$ and $v_2$ are the periods of $s$ and $t$, then they have generically constant lengths if and only if $w=(1-\zeta_3) = \hdots = (1-\zeta_d) = 0$ and $\zeta_1, \zeta_2 \ne 1$. The last statement implies that $k \ne 1$ and that $X - s \cup t$ is a translation surface. Since $k \ne 1$, by Apisa-Bainbridge-Wang \cite[Theorem 1.6]{ABW2}, $w$ is generically $0$ if and only if the residue around each element of $P$ is generically $0$. But this is a positive codimension condition\footnote{This follows from \cite{ABW2}, but a direct proof is that, when $k \ne 1$, the residues around points in $P$ can be extended to form period coordinates by \cite[Corollary 2.3]{BCGGM-k-diff}.} unless $|P|=0$, as desired. 
\end{proof}

\begin{proof}[Proof of Theorem \ref{T:HatHomology}:]
If $X - (s \cup t)$ is connected, then apply Proposition \ref{P:HatHomology}. If $s$ is separating, then let $Y$ be the component of $X - s$ that contains $t$. $Y$ has a boundary component made up of a single saddle connection $s$, so we mark the midpoint of $s$ and glue both halves to one other with a rotation by $\pi$ to form a closed surface $X'$ on which $s$ and $t$ still have generically constant ratio of lengths. This construction allows us to assume that neither $s$ nor $t$ is separating.

It remains to consider when $X - (s \cup t)$ has two components. Both $s$ and $t$ appear once in the boundary of each component. So glue $s$ and $t$ to themselves as in the previous construction to form $k$-differentials $Y_1, Y_2$. The resulting pair of saddle connections must be hat-homologous on either $Y_1$ or $Y_2$ to be hat-homologous on $X$. By considering the two saddle connections on this component we apply Proposition \ref{P:HatHomology} to conclude.
\end{proof}

\begin{rem}\label{R:HolonomyRatio}
Let $s$ and $t$ be hat-homologous disjoint simple saddle connections on a $k$-differential $X$ of lengths $|s|$ and $|t|$ respectively. If $s$ is nonseparating, let $\zeta_s$ (resp. $\zeta_t$) be the holonomy of a simple closed curve that passes through $s$ (resp. $t$) and that does not intersect $t$ (resp. $s$). Otherwise set $\zeta_s = -1$ (resp. $\zeta_t = -1$). Let $a_s$ (resp. $a_t$) be $1$ if $s$ (resp. $t$) is nonseparating and $2$ otherwise. A consequence of the proof is that $\frac{|1-\zeta_s||s|}{a_s} = \frac{|1-\zeta_t||t|}{a_t}.$
\end{rem}

A cylinder on a $k$-differential is called \emph{generic} if, on each boundary component of the cylinder, all saddle connection are generically parallel. Any cylinder can be perturbed in the stratum to be generic. Masur-Zorich \cite{MZ} found that generic cylinders on quadratic differentials have one of five types (see Apisa-Wright \cite[Section 4.1]{ApisaWrightDiamonds} for a discussion). We will see that the same five possibilities hold for Euclidean cylinders on $k$-differentials. 

Recall that the \emph{multiplicity of a saddle connection} in the boundary of a cylinder is, roughly, the number of times it appears in that boundary component (see \cite{ApisaWrightDiamonds} for a formal definition). 

\begin{cor} \label{C:TypesofCylinders}
A generic Euclidean cylinder in a stratum of $k$-differentials has at most two saddle connections, counted with multiplicity, on each boundary. When there are two distinct saddle connections, cutting them separates the surface into two components; the component not containing the cylinder has trivial holonomy. 
\end{cor}
\begin{proof}
Suppose that a boundary of the cylinder $C$ on the surface $X$ has two distinct hat-homologous saddle connections $s$ and $t$. 

We first show that the component $Y$ of $X - s \cup t$ containing $C$ cannot have trivial holonomy. If it did, then there is a holomorphic $1$-form $\omega$ on $Y$ that specifies the flat structure. Rotate and rescale it so that $C$ is horizontal and so that $s$ and $t$ appear on its upper boundary and have periods $1$ and $\lambda \in \bR_{>0}$. If they appear again in the boundary of $Y$ then their periods are $\zeta_s$ and $\zeta_t \lambda$ for some $\zeta_s, \zeta_t \in S^1$. In order for $\int_{\del Y} \omega = 0$ we must have $\zeta_s = \zeta_t = -1$. But then $X$ is a translation surface and $s$ and $t$ cannot be hat-homologous since $X - s \cup t $ is connected, a contradiction. 

Let $Y'$ be the component of $X - s \cup t$ with trivial holonomy. This exists by Theorem \ref{T:HatHomology}. Note that both $Y$ and $Y'$ have exactly two boundary saddle connections, one corresponding to $s$ and one to $t$. In particular, $s$ and $t$ appear with multiplicity one in the boundary of $C$.

It follows that the boundary of $C$ containing $s$ and $t$ cannot contain a third hat-homologous saddle connection $u$. If it did, then let $Z'$ be the component of $X - t \cup u$ that has trivial holonomy. Then $X - (s \cup t \cup u)$ has three components: $Y'$, $Z'$, and the component containing $C$. Each has $t$ on its boundary, which is a contradiction. 
\end{proof}

Recall that a cylinder is called \emph{simple} if each of its boundary components is comprised of a single saddle connection of multiplicity one.

\begin{cor}\label{C:TypesofCylinders2}
    Any nonseparating generic Euclidean cylinder $C$ in a stratum of $k$-differentials is a simple cylinder.
\end{cor}
\begin{proof}
Corollary \ref{C:TypesofCylinders} implies there is a unique saddle connection on each boundary component of $C$. They have multiplicity one since otherwise cutting the cylinder core curve disconnects the surface. 
\end{proof}

\begin{rem}
All statements and proofs in this section hold for isoholonomic loci of moduli spaces of flat cone metrics since we only used that holonomy was $S^1$-valued and not that it was discretely-valued.
\end{rem}

\section{Cylinders and the proof of Theorem \ref{T:Cyl}}\label{S:Cylinder}

The main result of this section is the following.

\begin{thm}\label{T:Cylinder}
Any component of a stratum of positive genus surfaces contains a surface with a simple nonseparating Euclidean cylinder. 
\end{thm}

%
%




%
%

We will detect the cylinders in Theorem \ref{T:Cylinder} using the boundary. Say that a multiscale $k$-differential is \emph{cylindrical} if it has exactly one irreducible component and one (necessarily horizontal) node. 

\begin{lem}\label{L:CylindricalBoundary}
Theorem \ref{T:Cylinder} holds if and only if the component contains a cylindrical point in its boundary. 
\end{lem}
\begin{proof}
For the forward direction, given a simple nonseparating Euclidean cylinder, sending its height to $\infty$ while fixing its complement produces a cylindrical boundary point. 

For the reverse direction, any horizontal node corresponds to an infinite area Euclidean cylinder. Since the node is assumed to be nonseparating, so is the core curve of the cylinder. We therefore have a surface in the interior of the component with a nonseparating Euclidean cylinder. This cylinder is generically simple by Corollary \ref{C:TypesofCylinders2}.
\end{proof}

\begin{cor}\label{C:GenusOneCyl}
Theorem \ref{T:Cylinder} holds for genus one strata.
\end{cor}
\begin{proof}
We induct on dimension. Recall that minimal dimensional genus one strata project into $\cM_{1,2}$ as $\{(E,p,0) \in \cM_{1,2} : E \in \cM_1, |p| = d \}$ for some integer $d$. These contain cylindrical boundary surfaces as can be seen by placing $p$ at $(\frac{1}{d}, 0)$ on the unit square torus and sending the height of the torus to $\infty$. This establishes the base case by Lemma \ref{L:CylindricalBoundary}. 

For the inductive step, we merge two zeros or two poles (by Corollary \ref{C:GenusOneMerging}), use the induction hypothesis to develop a horizontal node, and then smooth the node formed by merging. 
\end{proof}

A \emph{minimal} degeneration of a stratum is a multiscale $k$-differential in the boundary with the property that no nonempty proper subset of its nodes can be smoothed. 

\begin{lem}\label{L:MinDegen}
Minimal degenerations are one of the following: one or two components joined by a horizontal node or components at two levels with only vertical nodes. In the final case, if all top level components are genus zero then there is one component at each level. 
\end{lem}
\begin{proof}
Since we can smooth all nodes below a given level, minimal degeneration have at most two levels. 

Suppose that the top level has horizontal nodes. Any horizontal node can be smoothed independently, so there must be no lower level and one such node.

Suppose that the top level has no horizontal nodes. Since horizontal nodes at the lower level can be smoothed, there must be none. If additionally all top level components are genus zero, then the global $k$-residue condition is trivial. If there were multiple top vertices, then the trivial residue condition implies that we can create an intermediate level in the level graph and move one of the top vertices to that level. Smoothing all edges joining the intermediate and bottom levels shows that our original degeneration was not minimal. 

Thus, there is a unique top vertex. If there were multiple bottom vertices, then again we move one to the intermediate level and smooth all nodes connecting the top and intermediate levels to see that our original degeneration was not minimal. So there are unique top and bottom vertices.
\end{proof}

\begin{proof}[Proof of Theorem \ref{T:Cylinder}:] Proceed by induction on the dimension of the stratum. Since the genus $g = 1$ case was handled in Corollary \ref{C:GenusOneCyl}, this is our base case and allows us to assume that $g \geq 2$. 

The component has a nonempty boundary by noncompactness of strata (Theorem \ref{T:Chen-Aygun}). Consider a minimally degenerated surface $X$ on the boundary of the component. If $X$ has a single irreducible component, then it is cylindrical by Lemma \ref{L:MinDegen}. 

If $X$ has two irreducible components, $X_1$ and $X_2$, connected by a horizontal node, then one must have genus at least one. Suppose that component is $X_1$. Using the induction hypothesis, degenerate $X_1$ to a cylindrical boundary differential\footnote{As a minor technical point, this requires rescaling $X_2$ throughout this degeneration so that the residue of its pole of order $-k$ matches that of $X_1$.}. Then, smooth the horizontal node that connected $X_1$ and $X_2$ to conclude.

Suppose that $X$ has two levels and only vertical nodes. If $X$ has a genus at least one surface at the top level, then we proceed as in the previous paragraph. So we may suppose that $X$ has only two irreducible components by Lemma \ref{L:MinDegen}. 

Let $X_1$ be the top component, which we may assume has genus zero, and let $X_2$ be the bottom component. Since $X_1$ is genus zero and has no horizontal nodes, the global $k$-residue condition is vacuous. If $X_2$ has genus at least one, then we may proceed as in the previous paragraph. So suppose that $X_2$ also has genus $0$. Since $X$ has genus at least two there are at least three nodes. We merge two nodal poles $P_1$ and $P_2$ on $X_2$ and then smooth $P_1$ and $P_2$ (as in the last two steps of Figure \ref{F:L10-3}) in order to add genus to $X_1$ and thereby reduce to a previous case. 
\end{proof}

\begin{lem}\label{L:Genus0Cylinder}
$\Omega^k \cM_0(\kappa)$ contains a surface with a Euclidean cylinder if and only if a subset of $\kappa$ sums to $-k$. The cylinder can be taken to be simple if neither subset is $(-\frac{k}{2}, -\frac{k}{2})$.
\end{lem}
\begin{proof} 
Let $\kappa = S \cup T$ be the partition into subsets that each sum to $-k$. Let $\gamma$ be a curve separating $S$ from $T$. Pinching this curve causes a node to develop with order $-k$ on both irreducible components. This is a horizontal node, which corresponds to a Euclidean cylinder, and gives the forward implication. The second statement is by Corollary \ref{C:TypesofCylinders}.

Conversely, given a Euclidean cylinder on a surface in $\Omega^k \cM_0(\kappa)$, we can pinch its core curve and obtain two genus zero irreducible components connected by a horizontal node. Each end of the node has a $-k$ pole, so the singularities of $\kappa$ on either of the two irreducible components must sum to $-k$.
\end{proof}

\begin{proof}[Proof of Theorem \ref{T:Cyl}:] This is just Theorem \ref{T:Cylinder} and Lemma \ref{L:Genus0Cylinder}.
\end{proof}

\section{Splits and merges preserving nonhyperellipticity}\label{S:Simple}

Recall that \emph{splits} and \emph{merges} are defined in Section \ref{S:Intro}, where it is explained how they can be constructed from simple Euclidean cylinders. In the sequel, the $\bP^1$ component that is pinched off in a split will always be at the lower level of the multiscale $k$-differential. A split or merge that preserves primitivity will be called \emph{simple}. A \emph{simple degeneration} will mean a simple split or simple merge. Examples include splits and merges obtained by collapsing simple Euclidean cylinders. The main theorem of this section is the following. 

\begin{thm} \label{T:SimpleDegeneration}
Any primitive nonhyperelliptic component of a stratum of surfaces of genus at least two, except for $\Omega^2 \cM_2(5,-1)$ and $\Omega^3 \cM_2(6)$, has a simple degeneration that preserves nonhyperellipticity. It is either a simple split of a zero of order at least $2$ or a simple merge of two singularities of order strictly greater than $-\frac{k}{2}$.
\end{thm}

We emphasize that hyperelliptic components can be made into nonhyperelliptic components simply by marking a nonsingular point. 

\begin{lem}[Splitting zeros in genus one]\label{L:LosingGenusInGenus1}
The boundary of $\Omega^k_r \cM_1(a, \kappa)$ contains $\Omega^k \cM_0(a_1, a_2, \kappa)$ for any $a_1, a_2 > -k$ where $a_1 + a_2 = a - 2k$ and $r \mid \gcd(k+a_1, k+a_2, \kappa)$. If $r \ne a$, such a partition exists. 
\end{lem}
\begin{proof}
Glue together a sphere in $\Omega^k \cM_0(a_1, a_2, \kappa)$ and a sphere in $\Omega^k \cM_0(-a_1-2k, -a_2-2k, a)$ along the first two singularities (see Figure \ref{F:SplitMerge}). Smooth these nodes to produce a surface $X$ in $\Omega^k\cM_1(a, \kappa)$ with $k$-framing $w$. On $X$, the nodes become smooth curves $\gamma_1, \gamma_2$ where $w(\gamma_i) = k+a_i$. The smoothing requires a choice of prong matching. The rotation number of $X$ is $\gcd(w(\eta), k+a_1, k+a_2, \kappa)$ where $\eta$ is a simple closed curve that intersects $\gamma_1$ and $\gamma_2$ exactly once. Suppose that $r \mid \gcd(k+a_1, k+a_2, \kappa)$. By changing the choice of prong matching, we can change $w(\eta)$ mod $\gcd(|k+a_1|, |k+a_2|)$ arbitrarily; in particular, we can change it to be congruent to $r$. The surface we obtain by smoothing according to this choice of prong matching thus has rotation number $r$. For the second claim, set $(a_1, a_2) = (r-k, a-r-k)$. 
\end{proof}

In the sequel, we will frequently use the following two constructions. 

\begin{lem}[Swapping splits and merges]\label{L:CommonConstruction}
Fix $X \in \Omega^k \cM_g(z, z',\kappa)$ with $z \geq 2$. If $k=1$, suppose that $z' \geq 1$.  The following hold:
\begin{enumerate}
    \item\label{I:SplitToMerge} Suppose that $z$ can be simply split to form $Y \in \Omega^k \cM_{g-1}(a,b,z',\kappa)$. Suppose that, on $Y$, $z'$ and $a$ can be simply merged to form $W \in \Omega^k \cM_{g-1}(z'+a,b,\kappa)$ where $z'+a > -k$. Then, on $X$, $z$ and $z'$ can be simply merged to form a $k$-differential $Z$.
    \item\label{I:SplitToMergeHyp} Under the assumptions of (\ref{I:SplitToMerge}), if, after forgetting marked points, either $W$ belongs to a nonhyperelliptic component or $W$ belongs to a hyperelliptic component but $b$ is not exchanged with $a+z'$ by a hyperelliptic involution, then $Z$ does not belong to a hyperelliptic component after forgetting marked points.
    \item\label{I:MergeToSplit} Suppose that $z$ and $z'$ can be simply merged to form $Z \in \Omega^k \cM_g(z+z', \kappa)$. Suppose that, on $Z$, $z+z'$ can be simply split to form $W \in \Omega^k \cM_{g-1}(a+z',b,\kappa)$ where $a,a+z',b > -k$. Then, on $X$, $z$ can be simply split to form a surface $Y \in \Omega^k \cM_{g-1}(a,b,z',\kappa)$.
    \item \label{I:MergeToSplitHyp} Under the assumptions of (\ref{I:MergeToSplit}), $Y$ belongs to a hyperelliptic component after forgetting marked points only if $W$ does too and either $a = 0$ or $a = z'$ and $a+z'$ is fixed by a hyperelliptic involution on $W$. 
\end{enumerate}
\end{lem}
\begin{figure}[H]
    \centering
    \includegraphics[width=.75\linewidth]{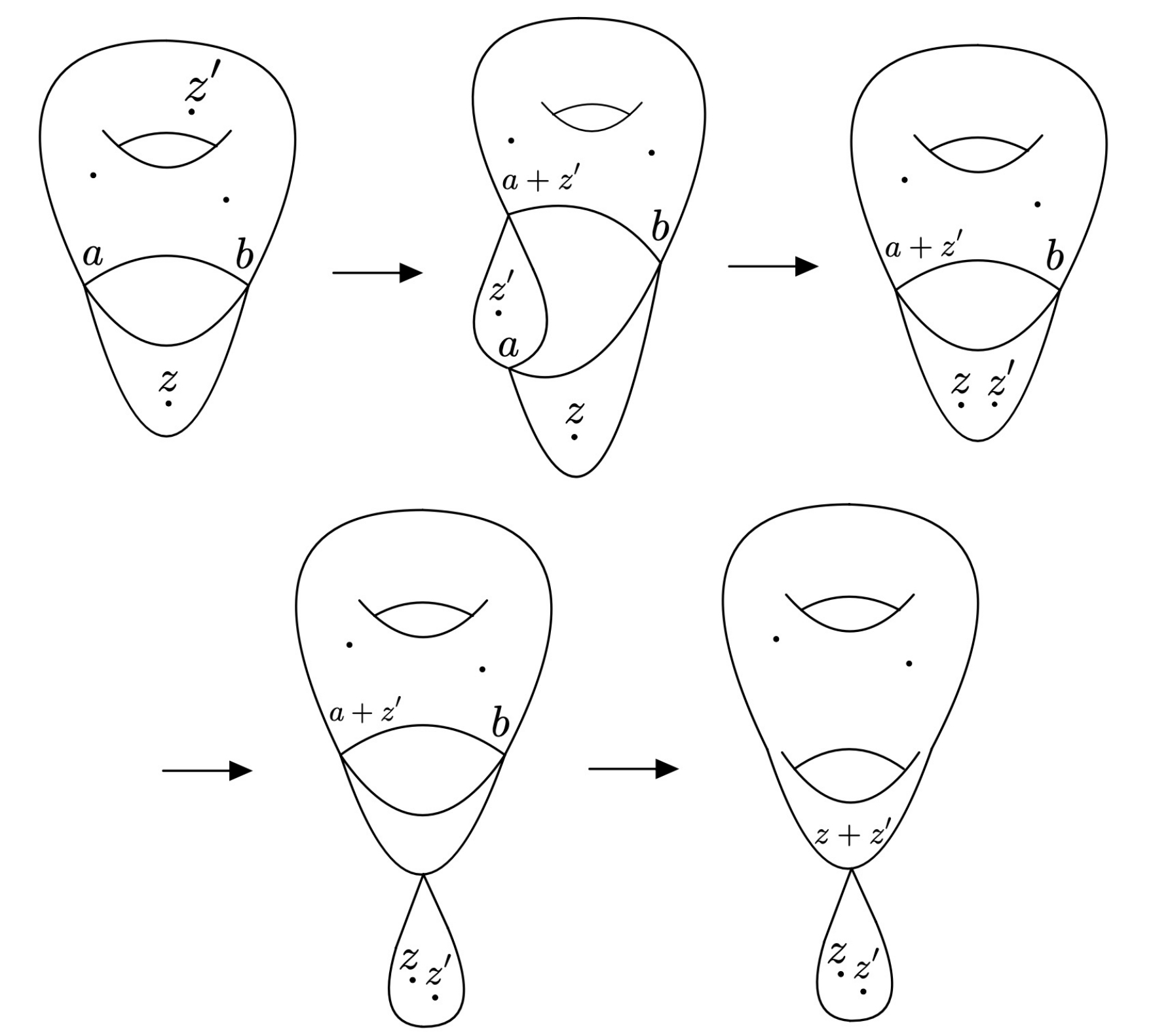}
    \caption{Proof of Lemma \ref{L:CommonConstruction}}
    \label{F:L10-2}
\end{figure} 
\begin{proof} The proof of (\ref{I:SplitToMerge}) is shown in Figure \ref{F:L10-2}. 

The first arrow is by assumption. The second is by smoothing a lower level node. The third is by merging two singularities on the lower level, which uses that genus zero strata with $n$ singularities can be identified with $\cM_{0,n}$. The fourth is by smoothing the top level nodes. The second and third operations require the vacuity of the global $k$-residue condition, which comes from the assumptions on primitivity of the degenerations and that $a, z' \geq 0$ if $k=1$. 

The proof of (\ref{I:SplitToMergeHyp}) follows since a degeneration of surfaces in a hyperelliptic locus (e.g. the one from $Z$ to $W$ when $Z$ belongs to a hyperelliptic component after forgetting marked points) that pinches two nonseparating curves necessarily pinches two curves exchanged by the hyperelliptic involution. 

The proof of (\ref{I:MergeToSplit}) is identical to that of (\ref{I:SplitToMerge}) with all degenerations in Figure \ref{F:L10-2} done in reverse.

The proof of (\ref{I:MergeToSplitHyp}) follows since a degeneration of surfaces in a hyperelliptic locus that pinches a single separating curve necessarily pinches one fixed by the hyperelliptic involution.
\end{proof}

\begin{lem}[Re-splitting]\label{L:NewSplit}
Fix $g \geq 2$. Let $S \in \Omega^k\mathcal{M}_g(z,\kappa)$ have a zero $z \geq 2$ that can be simply split to form $X \in \Omega^k \cM_{g-1}(a,b,\kappa)$ where $a \geq 2$. Suppose that $a$ can be simply split to form $Z \in \Omega^k \cM_{g-2}(c,d,b,\kappa)$. Then, $z$ can be simply split to form $Y \in \Omega^k \cM_{g-1}(c, b+d+2k, \kappa)$.

Moreover, $Y$ belongs to a hyperelliptic component (after forgetting marked points), only if $Z$ does and the involution on $Z$ exchanges $b$ and $d$.
    
If $g = 2$, then $Y$ has rotation number dividing $\gcd(k+d,k+b)$. If $X$ is not in a hyperelliptic component with $a$ and $b$ exchanged by the involution, then the rotation number of $Y$ can be set to $1$ unless the singularities in $\kappa \cup \{c\}$ are all even and $k,a,b$ are all odd, in which case the rotation number can be set to $1$ or $2$.
\end{lem}
\begin{proof} 
The proof of the first statement is Figure \ref{F:L10-3}. The second claim holds since, if $Y$ is hyperelliptic up to forgetting marked points, the two nonseparating pinched curves must be exchanged by the involution. 

For the third statement, let $\beta$ and $\delta$ be the simple closed curves on $Y$ formed by smoothing the nodes labeled $b$ and $d$ respectively. Let $\gamma$ be a simple closed curve on $Y$ intersecting these curves exactly once. Let $\gamma_X$ be the corresponding curve on the simple split that formed $X$. The curve is a union of two simple arcs, one on each component, between the two nodes. Let $w$ be the $k$-framing on $X$. By Corollary \ref{C:GenusOneProngMatchings}, we can change $w(\gamma_X)$ \footnote{The key point is that the winding number of a curve passing through the nodes is defined using the prong matching.} arbitrarily unless $k,a,b$ are odd and all singularities in $\kappa$ are even, in which case we can only change it by even numbers. So we can guarantee that $w(\gamma_X)$ is $1$ or $2$. Now redo the sequence of operations that took us from $X$ to $Y$. Since we changed the prong matchings at $X$, we arrive at a $k$-differential $Y'$ that may be different from $Y$. Let $\gamma'$ be the curve on $Y'$ corresponding to $\gamma_X$ and $\beta'$ the curve formed by smoothing the node labeled $b$. These two curves form a symplectic basis and now have winding numbers which yield the desired rotation number. 
\end{proof}

\begin{figure}[H]
    \centering
    \includegraphics[width=.8\linewidth]{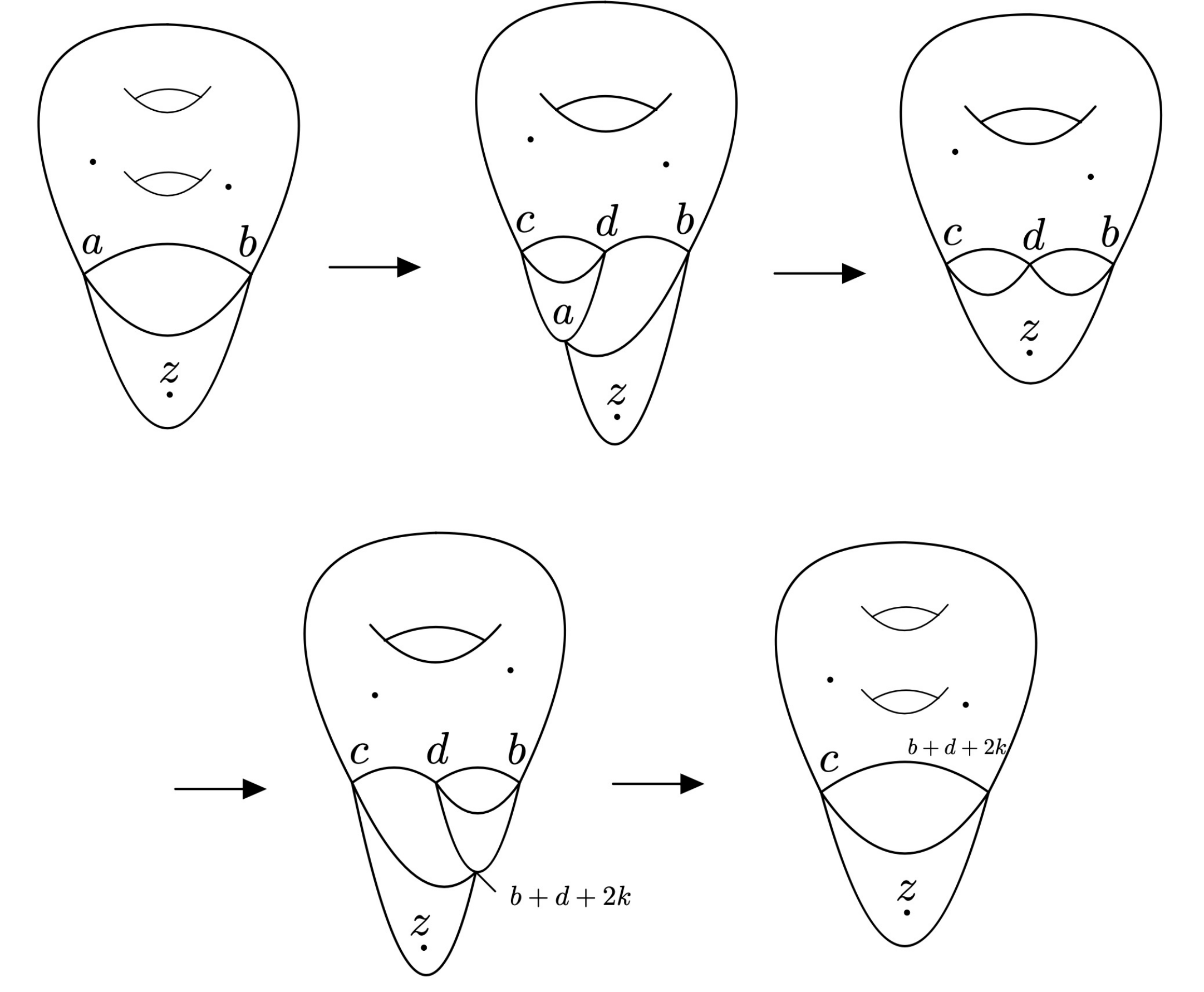}
    \caption{The proof of Lemma \ref{L:NewSplit}.}
    \label{F:L10-3}
\end{figure}

\begin{lem} \label{L:MergeHyp}
Two points exchanged by the involution can be simply merged in hyperelliptic components of strata of positive genus surfaces. A fixed point can be simply split into singularities of equal order in a hyperelliptic component.
\end{lem}
\begin{proof}
Hyperelliptic components can be identified with $\cM_{0,n+\epsilon}$ where $\epsilon \in \{1,2\}$ is the number of orbits of singularities under the hyperelliptic involution and $n$ is the number of nonsingular fixed points. Moving the image of the point(s) to such a fixed point produces a merge or split. We now consider primitivity of the degeneration. 

When $k \neq 1$, a hyperelliptic component is primitive if and only if its multiset of singularity orders is $(2m,2\ell), (2m,\ell,\ell)$, or $(m,m,\ell,\ell)$ where $\gcd(m,\ell,k)=1$. Merging two points of order $m$ results in $(2m, \ell, \ell)$ or $(2m, 2\ell)$. Likewise, splitting a point of order $2m$ results in $(m-k, m-k, \ell, \ell)$ or $(m-k, m-k, 2\ell)$. So we retain primitivity. 
\end{proof}

\begin{lem}\label{L:HyperellipticNonUniqueness}
The only primitive hyperelliptic components of strata with more than one hyperelliptic involution preserving the divisor of the $k$-differential are $\Omega^2 \cM_0(-1^4)$ and $\Omega^3 \cM_0(-2^3)$. 
\end{lem}
\begin{proof}
A genus $g$ hyperelliptic surface has a unique hyperelliptic involution unless possibly $g \leq 1$. Genus zero strata can be identified with $\cM_{0,n}$. Generic surfaces in $\cM_{0,n}$ admit a hyperelliptic involution if and only if $n \leq 4$, in which case the group generated by any two hyperelliptic involutions acts transitively on the $n$ points. Consequently, the orders of the singularities at the $n$ distinguished points must all agree. Since primitivity requires the gcd of $k$ and the orders of the singularities to be $1$, the result holds for genus zero strata.

It remains to consider genus one strata. Hyperelliptic components correspond to codimension $0$ loci of $\cM_{1,1}$ and $\cM_{1,2}$ (when $k=1$) and codimension $1$ loci of $\cM_{1,n}$ for $2 \leq n \leq 4$ (when $k \ne 1$). In codimension $0$ loci, there is clearly a unique hyperelliptic involution. For codimension $1$ loci, if there are two then they act transitively on the $n$ points, which implies that each singularity has order $0$, a contradiction. 
\end{proof}

\begin{lem}\label{L:TwoNodesNotSwapped}
Suppose that $Y \in \Omega^k \cM_{g-1}(a_1, a_2, \kappa)$ is a simple split of a primitive nonhyperelliptic component $\cC$ of $\Omega^k \cM_g(a,\kappa) \ne \Omega^2\cM_2(5,-1), \Omega^3 \cM_2(6)$ where $a, g \geq 2$. If $Y$ belongs to a hyperelliptic component and $a_1$ and $a_2$ are not exchanged by the involution, then $\cC$ admits a simple split of $a$ or a simple merge of two zeros to a nonhyperelliptic component. 
\end{lem}
\begin{proof}
Let $J$ be the involution on $Y$ and suppose that $a_1 \geq a_2$. Using the hyperellipticity of $Y$ and the fact $a_1$ and $a_2$ are not swapped by $J$, $a_1 \geq 0$. If $J(a_1) \ne a_1$, then simply merge $a_1$ and $J(a_1)$ (by Lemma \ref{L:MergeHyp}) and conclude by converting the original split into a merge (Lemma \ref{L:CommonConstruction} (\ref{I:SplitToMerge}) and (\ref{I:SplitToMergeHyp})).

So suppose that $J(a_1) = a_1$. Simply split $a_1$ (by Lemma \ref{L:MergeHyp}) to form a surface $Z$ in a hyperelliptic component of $\Omega^k\cM_{g-2}(\frac{a_1}{2}-k, \frac{a_1}{2}-k, a_2, \kappa')$ where the two singularities of order $\frac{a_1}{2}-k$ are exchanged by the involution. Resplit $a$ (by Lemma \ref{L:NewSplit}) to form a component of $\Omega^k \cM_{g-1}(\frac{a_1}{2}-k, a_2+\frac{a_1}{2}+k, \kappa')$, which is nonhyperelliptic unless the hyperelliptic involution on $Z$ is non-unique. By Lemma \ref{L:HyperellipticNonUniqueness}, this only occurs when the original stratum is $\Omega^2\cM_2(5,-1)$ or $\Omega^3 \cM_2(6)$.
\end{proof}

\begin{proof}[Proof of Theorem \ref{T:SimpleDegeneration}:] Let $\Omega^k \cM_g(\kappa)$ be the stratum. By Theorem \ref{T:Cylinder}, the component has a simple split of a zero or a simple merge of two singularities of order greater than $-\frac{k}{2}$ by collapsing a Euclidean cylinder. This forms a surface $Y$ in a component $\cD$. If $\cD$ is nonhyperelliptic, then we are done, so suppose that $\cD$ is hyperelliptic. Let $J$ be the involution on $Y$ and let $N$ be the set of nodes on $Y$ formed by the simple degeneration. By Lemma \ref{L:TwoNodesNotSwapped}, we may assume that $N$ contains either one point or two points that are exchanged by $J$. 

\noindent \textbf{Case 1: The complement of $N$ contains a $J$-invariant set of zeros.} Simply split it or simply merge them (by Lemma \ref{L:MergeHyp}) and smooth the nodes in $N$ to conclude. (Note that the $k$-differential on the lower level $\bP^1$-component is either nonhyperelliptic or it is hyperelliptic, but the prong matchings at the nodes in $N$ are not $J$-invariant. This certifies that smoothing the nodes in $N$ destroys hyperellipticity.) We may now suppose that the nodes $N$ are zeros and that only zeros occur at the nodes or their image(s) under $J$.

\noindent \textbf{Case 2: $N = \{a\}$ is a zero not fixed by $J$.} The original simple degeneration merged two singularities $z$ and $z'$ where $z \geq z'$ and $z \geq 1$. On $Y$, simply merge $a$ and $J(a)$ (by Lemma \ref{L:MergeHyp}) and smooth the node at $a$ to form a surface $Z \in \Omega^k \cM_g(2a, \kappa') \times \Omega^k \cM_0(-2k-2a, a, z, z')$ where $\kappa' = \kappa \setminus (a,z,z')$ and where the top level $k$-differential is primitive. If $k = 1$, then $z' > 0$ and so the global $k$-residue condition is vacuous. Merge $a$ and $z$ and then smooth all other nodes to produce the desired simple merge $W \in \Omega^k \cM_g(a+z, z', \kappa')$. Passing from $W$ to $Z$ involves merging singularities of nonzero orders $a+z > z'$, which certifies that $W$ is nonhyperelliptic. 

\noindent \textbf{Case 3: $N$ consists of one zero fixed by $J$.} The original simple degeneration merged two singularities $z \geq z'$. If $z = z'$, then the multiscale $k$-differential formed by the merge has $\Omega^k \cM_0(-2k-z-z', z, z')$ at the lower level. This is a hyperelliptic genus zero stratum with a fixed point at the node, and hence the smoothing the node yields a hyperelliptic component. Therefore, we assume $z \ne z'$. We may suppose that $z > z'$ and hence $z \geq  2$. On $Y$, simply split $(z+z')=:2c$ (by Lemma \ref{L:MergeHyp}) and then apply Lemma \ref{L:CommonConstruction} (\ref{I:MergeToSplit}) to find a simple split of $z$ on the boundary of the original stratum into singularities of order $c-k$ and $z-c-k$. (Note that $c < z$, which is needed to apply Lemma \ref{L:CommonConstruction} (\ref{I:MergeToSplit})). If the resulting component is hyperelliptic, then on it we have $z-c-k$ and $z'$ exchanged by the involution (by Lemma \ref{L:CommonConstruction} (\ref{I:MergeToSplitHyp})). We then apply Lemma \ref{L:TwoNodesNotSwapped}.

\noindent \textbf{Case 4: The original degeneration split a zero $a \geq 2$ into two singularities of order $b \geq 0$ exchanged by $J$.}  We may suppose that $b > 0$ unless $k = 1$, in which case it is straightforward to check that the original component of $\Omega^1 \cM_2(2)$ is hyperelliptic. On $Y$, simply merge these two singularities (by Lemma \ref{L:MergeHyp}) and smooth the nodes in $N$ to pass to a surface $Z\in \Omega^k \cM_{g-1}(a-2k, \kappa')^{hyp} \times \Omega^k_r \cM_1(a,-a)$ where $\kappa' = \kappa \setminus (a)$ and the superscript ``hyp'' indicates a primitive hyperelliptic component. Let $n$ be the node joining the two components. Since smoothing $n$ produces a nonhyperelliptic component, $r \ne \frac{a}{2}$. (This follows since $\Omega^k \cM_{g-1}(a-2k, \kappa')^{hyp} \times \Omega^k_{\frac{a}{2}} \cM_1(a,-a)$ lies on the boundary of the hyperelliptic component of $\Omega^k\cM_g(\kappa)^{hyp}$. So smoothing any nodal surface in it results in a surface in a hyperelliptic component by Lemma \ref{L:ProngsAndComps}.) Split $a$ into two singularities of unequal order (by Lemma \ref{L:LosingGenusInGenus1}) and smooth $n$ to find the desired split of $a$.
\end{proof}

\section{Merging}\label{S:Merge}

The main result of this section is the following.

\begin{thm}\label{T:Merge}
In a primitive nonhyperelliptic component, any two zeros and any two poles can be simply merged while retaining nonhyperellipticity. 
\end{thm}
\begin{proof}
We will induct on genus $g$ using Corollary \ref{C:GenusOneMerging} as the base case. We will now suppose that $g \geq 2$ and induct on the dimension of the stratum. By Theorem \ref{T:SimpleDegeneration}, there is a simple degeneration into a nonhyperelliptic component that either splits some singularity $a$ or merges some two singularities $a$ and $b$. Let $X \in \Omega^k \cM_g(a, \kappa)$ be a surface in the original component, and let $Y$ be this simple split or simple merge of $X$ from Theorem \ref{T:SimpleDegeneration}.

We only need to show that some two zeros or some two poles can be merged on $X$ while preserving primitivity and nonhyperellipticity. The induction hypothesis then allows us to merge the rest until we arrive at $X' \in \Omega^k \cM_g(Z, P)$ where $Z$ and $P$ are the sums of the orders of zeros and poles of $(a, \kappa)$ respectively. In the multiscale compactification, this moves all of the zeros and all of the poles to spheres joined to $Z$ and $P$ respectively. We may then merge any two zeros or two poles on these spheres and open the remaining nodes. 

\begin{sublem}\label{L:MergeNon-hyp}
If $Y \in \Omega^k \cM_{g-1}(a_1, a_2, \kappa)$ is a simple split of $a$, then $X$ admits a simple nonhyperelliptic merge of two zeros or two poles.
\end{sublem}
\begin{proof}
Note that $a_1, a_2 > -k$. Suppose first that $a_1 = 0$ and, after forgetting marked points, $Y$ belongs to a hyperelliptic component. Let $J$ be the hyperelliptic involution. Proceed as follows:
\begin{itemize}
    \item If $\kappa$ contains a zero that is distinct from $J(a_2)$, merge it with $a_1$ and conclude by Lemma \ref{L:CommonConstruction} (\ref{I:SplitToMerge}) and (\ref{I:SplitToMergeHyp}).
    \item If $\kappa$ contains a zero that is equal to $J(a_2)$, then merge it with $a_2$ using Lemma \ref{L:MergeHyp} and conclude by Lemma \ref{L:CommonConstruction} (\ref{I:SplitToMerge}) and (\ref{I:SplitToMergeHyp}).
    \item If $\kappa$ contains no zeros and two poles, then they are swapped by the hyperelliptic involution. We merge them using Lemma \ref{L:MergeHyp} and undo the split of $a$.
\end{itemize}
Suppose next that $Y$ belongs to a nonhyperelliptic component even after forgetting marked points. 

If $\kappa$ contains two poles or two zeros, then, by the induction hypothesis, we can merge them while preserving primitivity and nonhyperellipticity. Smoothing the two nodes formed by splitting $a$ gives the desired merge. If $\kappa$ is empty or only contains a pole, then the claim is vacuous. Thus we assume that $\kappa$ contains a unique zero $z'$ and possibly, although not necessarily, a pole $p$.

If $a_1$ or $a_2$ is a zero or marked point, then we merge $z'$ with it and conclude by Lemma \ref{L:CommonConstruction}. So we may suppose that $a_1,a_2 < 0$.

\textbf{Case 1: $Y$ has genus one.} If $Y \in \Omega_r^k \cM_{1}(a_1, a_2, z')$  then, by Proposition \ref{P:GenusOneMerging}, we can merge $z'$ with $a_1$ or $a_2$ (otherwise $a_1 = a_2 = -r$, implying that $Y$ belongs to a hyperelliptic component, a contradiction). We conclude by Lemma \ref{L:CommonConstruction}. If $Y \in \Omega_r^k \cM_1(a_1, a_2, z', p)$ then we can also merge $z'$ and $a_1$ or $a_2$ (otherwise $a_1 = a_2 = -r = -z'$, implying that $Y$ belongs to a hyperelliptic component, a contradiction). We conclude by Lemma \ref{L:CommonConstruction}.

\textbf{Case 2: $Y$ has genus at least two.} By the induction hypothesis, we can merge all poles to create a single pole of order $c \leq -2$ on a new surface $Z$ in a primitive nonhyperelliptic component of $\Omega^k \cM_{g-1}(z',c)$. By Theorem \ref{T:SimpleDegeneration}, there is a simple degeneration $W$ of $Z$.
\begin{itemize}
	\item If $W \in \Omega^k \cM_{g-2}(b_1, b_2, c)$ is a simple split, then some $b_i > 0$ since otherwise $W$ is a sphere, contradicting that $Y$ has genus at least two. Smooth all nodes formed by splitting $z'$ to reduce to a previously handled cases.
	\item If $W \in \Omega^k \cM_{g-1}(k(2g-4))$ is a simple merge,  then, after smoothing nodes at lower levels, the lower level is an element of $\Omega^k \cM_1(-2k(g-1), a, \kappa)$. Merge $a$ and $z'$ on the lower level and smooth the other nodes to conclude. When $k=1$ this requires verifying the vacuity of the global $k$-residue condition. But, when $k=1$, Theorem \ref{T:SimpleDegeneration} does not produce a degeneration where $z'$ and $c$ merge; so we're in the previous subcase.
\end{itemize}
\end{proof}

So it remains to suppose that $Y$ is a simple merge, say of $a$ and some other singularity $b$. Write $Y \in \Omega^k \cM_g(a+b, \kappa')$ where $\kappa' := \kappa \setminus (b)$. If $a$ and $b$ are both zeros or both poles, then we are done. We conclude with the following.

\begin{sublem}
If $k>1$ and $Y \in \Omega^k \cM_g(a+b, \kappa')$ where $a > 0 > b$, then $X$ admits a simple nonhyperelliptic merge of two zeros or two poles.
\end{sublem}
\begin{proof}
First assume that there is a pair of zeros (including marked points) or a pair of poles (including marked points) in $Y$. Simply merge them by Lemma \ref{L:MergeHyp} if $Y$ is hyperelliptic after forgetting marked points or by the induction hypothesis otherwise. Next, resolve the $a+b$ node to form a surface $Z$. If the points we just merged belonged to $\kappa'$, then we are done. Otherwise, we merged a point $c$ with $a+b$ and $Z \in \Omega^k \cM_g(a+b+c, \kappa' \setminus(c)) \times \Omega^k \cM_0(a,b,c,-2k-a-b-c)$. Merge $c$ with whichever of $a$ and $b$ has the same sign and smooth all other nodes to conclude. 

The remaining case is $\kappa = (a,b,c)$ where $a+b \ne 0$ and $a+b$ and $c$ have opposite signs. By Theorem \ref{T:SimpleDegeneration}, let $W$ be a simple degeneration of $Y$. 
\begin{itemize}
    \item If $W$ is a simple merge, then once again merge $c$ with whichever of $a$ and $b$ has the same sign and smooth all other nodes. 
    \item If $W$ is a simple split of $c \geq 2$, then smooth the $a+b$ node and apply Sublemma \ref{L:MergeNon-hyp}.
    \item If $W$ is a simple split of $a+b$ then apply Lemma \ref{L:CommonConstruction} (using that then $a \geq 2$ and $b<0$) to split $a$ and apply Sublemma \ref{L:MergeNon-hyp}.
\end{itemize}
\end{proof}
\end{proof}

\section{Proof of Theorem \ref{T1}}\label{S:ProofMainTheorem}

In this section we prove Theorem \ref{T1}. 

\begin{lem}\label{L:KeyChangeOfRN}
Suppose that we can split a zero $z\geq 2$ on $X \in \Omega^k \cM_2(z, \kappa)$ to form $Y \in \Omega^k_r \cM_1(a,b,\kappa)$ where $r \ne a \geq 2$ and the stratum is not hyperelliptic with $a$ and $b$ exchanged by the involution. Then we can split $z$ to form a surface in $\Omega^k_{r'} \cM_1(1-k,z-k-1,\kappa)$ where $r' = 1$ unless $k$ is odd and the entries of $\kappa$ are even, in which case $r' \in \{1,2\}$.
\end{lem} 
\begin{proof}
Split $a$ using Lemma \ref{L:LosingGenusInGenus1} to form a surface in $\Omega^k \cM_0(c,d,b,\kappa)$ where $-k < c \leq d$. Resplit $z$ by Lemma \ref{L:NewSplit}, to form $Z \in \Omega_{r''}^k\mathcal{M}_1(c_1,c_2,\kappa)$ where $c_1 = b+d+2k\geq 2$, $c_2 = c$, and where $r'' =1$ unless $k, a,b$ are odd and $c$ and the entries of $\kappa$ are even, in which case $r'' \in \{1,2\}$. If $c_1=2$, then $b = d= 1-k$ and so $c = 1-k$. Therefore, $\Omega_{r''}^k\mathcal{M}_1(c_1,c_2,\kappa)$ is our desired component. 

Suppose now that $c_1>2$, which implies that $\Omega_{r''}^k\mathcal{M}_1(c_1,c_2,\kappa)$ is not hyperelliptic with $c_1$ and $c_2$ exchanged by the hyperelliptic involution. In particular, we may apply Lemma \ref{L:NewSplit} again. We also note that 
\[ z \geq c_1 + (1-k) + 2k \geq k+4. \]

If $r'' = 1$, then by Lemma \ref{L:LosingGenusInGenus1}, we can split $c_1$ to form $\Omega^k \cM_0(1-k,c_1-1-k,c_2,\kappa)$. We then conclude by resplitting $z$ with Lemma \ref{L:NewSplit}.

If $r'' = 2$, then $k$ is odd and the singularities in $z, c_1, c_2,$ and $\kappa$ are even. As above, we could have split $z$ to $\Omega_1^k\mathcal{M}_1(2-k,z-k-2,\kappa)$. This stratum is only hyperelliptic with $2-k$ and $z-k-2$ interchanged by the hyperelliptic involution when $2-k=1 =z-k-2$ or $2-k=-1 =z-k-2$. This only happens when $(k,z) = (1,4)$ or $(k,z) = (3,4)$, contradicting that $z \geq k+4$. Using that $z \geq k+4$ and Lemma \ref{L:LosingGenusInGenus1}, we can split $z-k-2$ to form $\Omega^k \cM_0(1-k,z-2k-3, 2-k, \kappa)$. We then conclude by resplitting $z$ with Lemma \ref{L:NewSplit}. 
\end{proof}

\begin{prop}[Classification of minimal genus two strata] \label{P:MinimalGenusTwo}
The number of primitive nonhyperelliptic components of $\Omega^k \cM_2(2k)$ is zero if $k = 1,2$, one if $k = 3$ or if $k$ is even, and two otherwise. Moreover, each such component has $\Omega_r^k \cM_1(1-k, k-1)$ on its boundary where $r = 1$ if $k$ is even or $k =3$ and $r \in \{1,2\}$ otherwise.
\end{prop}
\begin{proof}
By Theorem \ref{T:SimpleDegeneration}, simply split the zero to form a surface in $\Omega_r^k \cM_1(a,-a)$ which is nonhyperelliptic if $k \ne 3$ and where $a = 0$ if $k = 1$ and $a \in \{2, \hdots, k-1\}$ otherwise. If $k = 1,2$, then these boundary strata are hyperelliptic or empty, showing that there are no primitive nonhyperelliptic components of these minimal strata. When $k=3$, $\Omega^3_1 \cM_1(2,-2)$ has one (local) equivalence class of prong matching between the singularities and so $\Omega^3 \cM_2(6)$ has one primitive nonhyperelliptic component (see Lemma \ref{L:ProngsAndComps}). 

Now assume that $k > 3$. Note that $r \ne a \geq 2$. By Lemma \ref{L:KeyChangeOfRN}, $\Omega_{r'}^k \cM_1(1-k, k-1)$ belongs to the boundary where $r' = 1$ if $k$ is even and $r' \in \{1,2\}$ if $k$ is odd. Each of these boundary strata has a unique (local) equivalence class of prong matchings, so the result follows by Lemma \ref{L:ProngsAndComps} when $k$ is even. When $k$ is odd, we achieve an upper bound of two primitive nonhyperelliptic components.

When $k > 3$ is odd, we will show that this bound is sharp by smoothing the nodes of a multiscale $k$-differential in $\Omega^k_{r'} \cM_1(1-k, k-1) \times \Omega^k \cM_0(-k-1, 1-3k, 2k)$ with the two surfaces glued together along the first two nodes and where $r' \in \{1,2\}$. Let $w$ be the $k$-framing after smoothing the nodes to form a surface $X$. If $\gamma$ is a loop that is pinched to form a node, then $w(\gamma)$ is odd. It follows that $r'+1$ mod $2$ is the Arf invariant of $X$. In particular, there are two components distinguished by the Arf invariant when $k>3$ is odd.
\end{proof}

\begin{cor}[Splitting in genus two] \label{C:SplitinGenusTwo}
Any primitive nonhyperelliptic component of $\Omega^k \cM_2(a, \kappa)$, where $a \geq 2$ can be split to form $\Omega^k_r \cM_1(1-k, a-k-1, \kappa)$ where $r = 1$ unless $k$ is odd and all elements of $\kappa$ are even in which case $r \in \{1,2\}$.
\end{cor}
\begin{proof}
Induct on $|\kappa|$. Proposition \ref{P:MinimalGenusTwo} is the base case, so assume now that $|\kappa| > 0$. First suppose $\kappa$ contains a singularity $b$ so that $a+b \geq 2$ and $a$ and $b$ can be merged while preserving primitivity and nonhyperellipticity. Note that by Theorem \ref{T:Merge}, this occurs if $\kappa$ contains a zero aside from $a$. Then merge $a$ and $b$, split the resulting zero (by the induction hypothesis) and apply Lemma \ref{L:CommonConstruction} (\ref{I:MergeToSplit}) to find the desired split of $a$.

Therefore, we may assume that $a$ is the only zero in $\kappa$. If $\kappa$ contains two poles, then merge them by Theorem \ref{T:Merge}, split $a$ by the induction hypothesis, and then unmerge the poles to conclude.

It remains to assume that $\kappa = (a,b)$ where $a \geq 2 > 0 > b$ and that $a$ and $b$ cannot be merged while preserving primitivity and nonhyperellipticity. The claim is clear for $\Omega^2 \cM_2(5,-1)$ because collapsing a simple cylinder sends us to $\Omega^2 \cM_1(2,-1,-1)$, so we ignore this case. By our hypothesis and Theorem \ref{T:SimpleDegeneration}, we can split $a$ while retaining nonhyperellipticity and primitivity to form $\Omega^k_r \cM_1(a_1, a_2, b)$ where $a_1 \geq a_2$. Since $\Omega^k \cM_1(1,0,-1)$ is empty and $\Omega^k \cM_1(1,1,-2)$ is hyperelliptic, we may suppose that $a_1 \geq 2$. If $r = a_1$, then we are in $\Omega^k_r \cM_1(r,r,-2r)$ which is hyperelliptic and hence a contradiction. So $r \ne a_1 \geq 2$ and we conclude with Lemma \ref{L:KeyChangeOfRN}. 
\end{proof}

\begin{prop}[Splitting in general]\label{P:Splitting}
In primitive nonhyperelliptic components of $\Omega^k \cM_g(\kappa)$ with $g \geq 2$, it is possible to split any zero of order $a \geq 2$ into $(1-k, a-k-1)$ while retaining primitivity. If $g \geq 3$, the split can be taken to be nonhyperelliptic after forgetting marked points with the following exceptions:
\begin{itemize}
    \item $\Omega^1 \cM_3(\kappa)$ where $\kappa$ is a positive partition of $4$ or $\Omega^1 \cM_3(\kappa', -2)$ or $\Omega^1 \cM_3(\kappa', -1, -1)$ where $\kappa' = (6), (4,2), (2,2,2)$. 
    \item $\Omega^2 \cM_3(\kappa', -1)$ where $\kappa' = (9), (6,3), (3,3,3)$.
    \item $\Omega^3 \cM_3(\kappa)$ where $\kappa = (12), (8,4), (4,4,4)$. 
    \item $\Omega^1 \cM_4(\kappa)$ with $\kappa = (6), (4,2), (2,2,2)$, 
\end{itemize}
\end{prop}
The exceptions in the statement will be called \emph{exceptional strata}.
\begin{proof}
Induct on genus and then on $|\kappa|$ with Corollary \ref{C:SplitinGenusTwo} forming the base case for the induction on $g$. Suppose now that $g \geq 3$.

\begin{sublem}
The claim holds when $|\kappa|=1$.
\end{sublem}
\begin{proof}
Split the zero by Theorem \ref{T:SimpleDegeneration} to form a primitive nonhyperelliptic surface in a component $\cC$ of $\Omega^k \cM_{g-1}(b,c)$ where $b \geq c > -k$. 

\noindent \textbf{Case 1: $\cC$ is nonhyperelliptic after forgetting marked points.} By the induction hypothesis split $b$ to form a component $\cD$ of $\Omega^k \cM_{g-2}(1-k, b-k-1, c)$. By Lemma \ref{L:NewSplit}, resplit the original stratum to form a component $\cC'$ of $\Omega^k \cM_{g-1}(1-k, k(2g-3)-1)$. Moreover, $\cC'$ is hyperelliptic after forgetting marked points only if $\cD$ is and $(b-k-1)$ and $c$ are exchanged by the involution. So
\[ b = \frac{k(2g - 3) + 1}{2} \quad \text{and} \quad c = \frac{k(2g - 5) - 1}{2}. \]
If $c \geq 2$, then we split $c$ instead of $b$ by the induction hypothesis and run the previous argument to conclude. If $c < 2$, then $(k,g) = (1,3), (1,4), (3,3)$, which yield exceptional strata.

\noindent \textbf{Case 2: $\cC$ is hyperelliptic after forgetting marked points.} In other words, $b = k(2g-4)$ and $c = 0$. After forgetting the marked point, $\cC$ becomes a primitive hyperelliptic component of $\Omega^k \cM_{g-1}(k(2g-4))$, so $k=1$. Split $b$ to form $\Omega^1 \cM_{g-2}(g-3, g-3, 0)$. By Lemma \ref{L:NewSplit} we split the original stratum to $\Omega^1 \cM_{g-1}(g-3, g+1)$, which is nonhyperelliptic after forgetting marked points for $g \ne 3$, which is exceptional. This reduces us to the previous case.
\end{proof}

We may suppose that $\kappa$ contains at most one pole. To see this, use Theorem \ref{T:Merge} to merge the poles while retaining nonhyperellipticity and primitivity. Merging poles never causes a stratum to become exceptional unless it was already exceptional. Now use the induction hypothesis to split the desired zero, while retaining primitivity and, in the non-exceptional cases, nonhyperellipticity, after forgetting marked points, and undo the merge of the poles.

\begin{sublem}\label{L:SplitToSplit}
Suppose that $a$ can be merged with a singularity $b$ while preserving nonhyperellipticity and primitivity. Suppose too that $a+b \geq 2$ and, if $k=1$, that $b > 0$. Then the claim holds.
\end{sublem}
\begin{proof}
By the induction hypothesis, split $a+b$ to produce a component $\cC$ of $\Omega^k \cM_{g-1}(1-k, a+b-k-1, \kappa')$ where $\kappa' = \kappa \setminus (a,b)$ and where $\cC$ is nonhyperelliptic after forgetting marked points unless $\cS := \Omega^k \cM_g(a+b, \kappa')$ is exceptional. By Lemma \ref{L:CommonConstruction}, split $a$ to form a component $\cD$ of $\Omega^k \cM_{g-1}(1-k,a-k-1,b,\kappa')$. This split is nonhyperelliptic after forgetting marked points unless $\cS$ is exceptional and either (1) $a-k-1=0$ or (2) $a-k-1=b$. Let $c = a+b$. Then either (1) $a = k+1$ and $b = c - (k+1)$ or (2) $a = \frac{c+k+1}{2}$ and $b = \frac{c-k-1}{2}$.

\noindent \textbf{Case 1: $k = 3$.} In this case, $(a+b, \kappa') = (12), (8,4), (4,4,4)$ so $c \in \{4,8,12\}$ and so $(a,b) = (4,0), (4,4), (4,8)$ or $(a,b) = (4,0), (6,2), (8,4)$. The only case that is not an exceptional stratum is $\Omega^3 \cM_3(6,2,4)$. However, in this case we simply merge $a = 6$ with the singularity of order $4$ instead of $2$.

\noindent \textbf{Case 2: $k = 2$.} In this case, $(a+b, \kappa') = (9, -1), (6,3,-1), (3,3,3,-1)$ so $c \in \{3,6,9\}$ and so $(a,b) = (3,0), (3,3), (3,6)$ or $(a,b) = (3,0), (6,3)$. All cases give exceptional strata.

\noindent \textbf{Case 3: $k = 1$ and $g=3$.} The case of $\kappa$ a positive partition of $4$ is an exception, so we ignore this case and therefore assume that $\kappa$ contains a pole of order $-2$. So given that $a \geq 2$ and $b >0$ by assumption, $(a+b, \kappa') = (6,-2), (4,2,-2)$ and so $\kappa = (2,4,-2), (2,2,2,-2), (4,2,-2), (3,1,2,-2)$. All cases give exceptional strata except the final one, where here we can merge $a = 3$ with the double zero instead of the simple zero. 

\noindent \textbf{Case 4: $k = 1$ and $g=4$.} Given that $a \geq 2$ and $b >0$ by assumption, $(a+b, \kappa') = (6), (4,2)$ so $\kappa = (2,4), (2,2,2), (4,2), (3,1,2)$. All cases give exceptional strata except the final one, where we then merge $a = 3$ with the double zero instead of the simple zero. 
\end{proof}

By Sublemma \ref{L:SplitToSplit} and Theorem \ref{T:Merge}, we conclude with the following.

\begin{sublem}
The claim holds for components of $\Omega^k \cM_g(a,b)$ where $a > 0 > b$. 
\end{sublem}
\begin{proof}
By Sublemma \ref{L:SplitToSplit} and Theorem \ref{T:SimpleDegeneration}, we may suppose that $a$ can be split to form a nonhyperelliptic component $\cC$ of $\Omega^k \cM_{g-1}(a_1, a_2, b)$ where $a_1 \geq a_2 > -k$. 

\noindent \textbf{Case 1: $\cC$ is nonhyperelliptic after forgetting marked points.} By the induction hypothesis, split $a_1$ to form a component $\cD$ of $\Omega^k \cM_{g-2}(1-k, a_1-k-1, a_2, b)$. If $g > 3$ and $\cC$ is not an exception listed in the statement, then we may also suppose that $\cD$ is nonhyperelliptic after forgetting marked points. By Lemma \ref{L:NewSplit}, we may resplit $a$ to form a component $\cC'$ of $\Omega^k \cM_{g-1}(1-k, a-k-1, b)$. This is nonhyperelliptic unless either (1) $g=3$ or (2) $\cC$ is exceptional and, in both cases, $a_1-k-1 = a_2$. Additionally, $\cD$ must be hyperelliptic after forgetting marked points. 

\noindent \textbf{Case 1a: $g=3$.} In this case, $\cD = \Omega^k_r \cM_{g-2}(1-k, a_1-k-1, a_2, b)$ where $r = 1$ unless $k$ is odd and $\{a_1, a_2, b\}$ are all even in which case $r \in \{1,2\}$. Since $\cD$ is hyperelliptic, $a_1-k-1=a_2=r$ and $b = -2r$ if $k=1$ and $b = -r = 1-k$ otherwise. The solutions are 
\[(a_1,a_2,b,k,r)\in \{(3,1,-2,1,1),\ (4,2,-4,1,2),\ (4,1,-1,2,1),\ (6,2,-2,3,2)\}.\]
The first and third correspond to the exceptional strata $\Omega^1 \cM_3(6,-2)$ and $\Omega^2 \cM_3(9,-1)$. For the other two we proceed by splitting $a_2$ and running the same argument.

\noindent \textbf{Case 1b: $g>3$.} Since $\cC$ is exceptional and contains a pole, $k \in \{1,2\}$. Either $k=1$ and $(a_1, a_2, b) = (4,2,-2), (6,0,-2)$ or $k=2$ and $(6,3,-1), (9,0,-1)$. The second cases are ruled out by the condition that $a_1-k-1 = a_2$. For the first cases, split $a_2$ instead of $a_1$ and run the same argument.

\noindent \textbf{Case 2: $\cC$ is hyperelliptic after forgetting marked points.} So $a_2 = 0$. Split $a_1$ (by Lemma \ref{L:MergeHyp}) to form $\Omega^k\cM_{g-2}( \frac{a_1}{2}-k, \frac{a_1}{2}-k, a_2, b)$. Resplitting $a$ (by Lemma \ref{L:NewSplit}) produces a component of $\Omega^k \cM_{g-1}( \frac{a_1}{2}-k, \frac{a_1}{2}+k, b)$, which is not hyperelliptic (even after forgetting marked points) since there is a degeneration of this surface which produces a marked point. This reduces us to a previously handled case.
\end{proof}
\end{proof}

\begin{proof}[Proof of Theorem \ref{T1}:] Induct first on $g$ and then on $|\kappa|$. Proposition \ref{P:MinimalGenusTwo} is the base case. Since the sum of residues of poles for an abelian differential is zero, if $k = 1$, then the multiset of poles in $\kappa$ is not $(-1)$. Similarly, if $k=1$, then merging (Theorem \ref{T:Merge}) implies that there are no nonhyperelliptic components of $\Omega^1 \cM_2(1,1)$ since then there would be one for $\Omega^1 \cM_2(2)$, contradicting Proposition \ref{P:MinimalGenusTwo}. We now have two cases for strata remaining.

\noindent \textbf{Case 1: $\kappa$ contains three simple zeros.} Set $\kappa' := \kappa \setminus (1,1,1)$. Merge two simple zeros (by Theorem \ref{T:Merge}). The resulting stratum is a primitive and nonhyperelliptic component of $\Omega^k \cM_g(2,1,\kappa')$. By the induction hypothesis, this is nonempty and connected except when $k=2$, $g=2$, and $\kappa' = (1)$; in which case we see that $\Omega^2 \cM_2(1,1,1,1)$ has no primitive hyperelliptic component since the same holds for $\Omega^2 \cM_2(2,1,1)$ by the induction hypothesis. The remaining cases are connected by Lemma \ref{L:ProngsAndComps} and since they merge to a unique component.

\noindent \textbf{Case 2: $\kappa$ contains a zero $a \geq 2$.} Set $\kappa' := \kappa \setminus (a)$. Split $a$ by Corollary \ref{C:SplitinGenusTwo} (when $g = 2$) and Proposition \ref{P:Splitting} (when $g > 2$). The resulting component is, 
\begin{itemize}
    \item when $g = 2$, $\Omega^k_r \cM_1(1-k,a-k-1, \kappa')$ where $r = 1$ unless $k$ is odd and all the singularities in $\kappa$ are even, in which case $r \in \{1,2\}$, and
    \item when $g > 2$, a nonhyperelliptic component of $\Omega^k \cM_{g-1}(1-k, a-k-1, \kappa')$ that remains nonhyperelliptic after forgetting marked points outside of the exceptions listed in Proposition \ref{P:Splitting}.
\end{itemize}
Because all prong matchings are (locally) equivalent, Lemma \ref{L:ProngsAndComps} implies that the original stratum has the number of primitive nonhyperelliptic components bounded above by the number of boundary strata described above. 

\begin{sublem}\label{L:ClassificationForGenus2}
Theorem \ref{T1} holds for $g=2$.
\end{sublem}
\begin{proof}
We have seen that there is at most one component unless $k$ is odd and all elements of $\kappa$ are even, in which case there are at most two. By undoing the split, these upper bounds are equalities unless $\Omega^k_r\cM_1(1-k,a-k-1, \kappa')$ is empty or hyperelliptic with $1-k$ and $a-k-1$ exchanged by the involution. Genus one strata are only ever empty if, ignoring marked points, there are two singularities. If $|\kappa'| \geq 2$, this can only occur if $k=1, a=2, r=2, \kappa' = (2,-2)$; the original stratum is then $\Omega^1 \cM_2(2,2,-2)$. If $\kappa' = (b)$, then the exceptions are $r=k-1 = b$ when $k>1$ and $r = a-2 = -b$ when $k=1$. These exceptions are
\[ \Omega^1 \cM_2(4,-2), \text{ } \Omega^1 \cM_2(2,2,-2), \text{ } \Omega^2 \cM_2(3,1), \text{ } \text{and } \Omega^3 \cM_2(4,2).\] 
Each empty boundary stratum reduces the number of expected components by one, which now produces a sharp upper bound on the number of components. In the exceptions due to hyperellipticity, $a=2$ and $r = k-1$ so the exceptions are 
\[ \Omega^2 \cM_2(2,1,1), \text{ } \Omega^2 \cM_2(2,2), \text{ } \Omega^3 \cM_2(4,2), \text{and } \Omega^3 \cM_2(2,2,2).\] 
Each of these strata has its expected number of primitive nonhyperelliptic components reduced by one, which now produces a sharp upper bound on the number of components.
\end{proof}

When $g > 2$ the boundary strata given by Proposition \ref{P:Splitting} are always nonempty. They are nonhyperelliptic after forgetting marked points in the unexceptional cases. When $k>3$, there is one primitive nonhyperelliptic component of these boundary strata unless $k$ is odd and all elements of $\kappa$ are even, in which case there are two (by the induction hypothesis). Because all prong matchings are (locally) equivalent, Lemma \ref{L:ProngsAndComps} implies that there is one primitive nonhyperelliptic component of $\Omega^k \cM_g(\kappa)$ except when $k$ is odd and $\kappa$ is an even partition, in which case we have an upper bound of two components. As before, the $k$-framing of small loops around the nodes are odd for a multiscale $k$-differential in $\Omega^k \cM_{g-1}(1-k, a-k-1, \kappa') \times \Omega^k \cM_0(-k-1, -k-a+1, a)$. This implies that the Arf invariant of the top level $k$-differental and the Arf invariant of the $k$-differential formed by smoothing the nodes coincide. By the induction hypothesis, the Arf invariant distinguishes the two primitive nonhyperelliptic components of $\Omega^k\cM_{g-1}(1-k,a-k-1,\kappa')$, and hence we realize our upper bound of two components. 

When $k=3$, the same argument goes through with the caveat that $\Omega^3 \cM_3(\kappa)$ for $\kappa = (12), (4,8), (4,4,4)$ have splits that may be hyperelliptic. So there are three possible components of boundary strata and hence at most three primitive nonhyperelliptic components of these strata. This upper bound is sharp by Theorem \ref{T:Genus3MinimalK=3}, whose proof we defer. 

When $k=2$, the same argument as with $k>3$ mostly goes through. The cases of $\Omega^2 \cM_3(\kappa, -1)$ where $\kappa = (9), (6,3), (3,3,3)$ are handled as with the exceptions in $k=3$. The only new phenomena is that genus four strata might split to genus three strata that have ``extra'' components. These strata are precisely $\Omega^2 \cM_4(\kappa)$ where $\kappa = (12), (9,3), (6,6), (6,3,3), (3,3,3,3)$. Since these genus four strata have no poles, splits in higher genera don't lead to them and hence genus five strata don't accrue ``extra'' components. ``Extra'' components in genus three and four actually arise by Chen-M\"oller \cite{ChenMoller-Exceptional}. The analysis when $k=1$ is completely analogous.
\end{proof}

\section{Sporadic components of genus $3$ cubic differentials}\label{S:SporadicCubics}

Let $X$ be a compact genus $3$ Riemann surface with sheaf of regular functions $\cO_X$ and canonical sheaf $K_X$. The main result is the following.

\begin{thm}\label{T:Genus3MinimalK=3}
There are at least three primitive components of $\Omega^3 \cM_3(\kappa)$ where $\kappa = (12), (8,4), (4,4,4)$. They are distinguished by the Arf invariant and, given a cubic differential $X$ with divisor $4P$, the generic value of $h^0(\cO_X(P))$.
\end{thm}

Let $X$ be a surface in a primitive component of $\Omega^3 \cM_3(12)$ with cubic differential $w$ having divisor $12p$. Since $3K_X = 12p$ in the Jacobian, $K(-4p)$ is $3$-torsion in the Picard group. Since $w$ is primitive, i.e. since $K_X \ne 4p$, $K(-4p)$ is not trivial. So there is a homomorphism $\chi \in H^1(X, \bZ/3)$ so that if $\bC_\chi$ is the associated local system, $K_X(-4p) = \cO_X \otimes \bC_\chi$. The kernel of $\chi$ determines a holomorphic cover $f: Y \ra X$ so that $f^*(K_X(-4p))$ has a global nonzero section $\omega$. Write $Q = q_1 + q_2 + q_3 = f^{-1}(p)$. Then $f^*(K_X(-4p)) = K_Y(-4Q)$. So $\omega$ is a holomorphic $1$-form whose divisor is $4Q$. Since $\omega^3$ and $f^* w$ are both global sections of $K_Y^{\otimes 3}(-12Q)$, they must be equal (up to rescaling $\omega)$, which shows that $f: Y \ra X$ is the holonomy cover of $(X, w)$. By Johnson \cite{Johnson-Spin}, the Arf invariant of $(Y, \omega)$ is $h^0(\cO_Y(2Q))$ mod $2$. 

\begin{lem}\label{L:Arf}
The Arf invariant of $(X,w)$ is $h^0(K_X^{\otimes2}(-6p))$ mod $2$. 
\end{lem}
\begin{proof}
Let $a_3$ be a simple closed curve so that the holonomy homomorphism is given by $a_3^*$ mod $3$, where $a_3^*$ is the Poincare dual of $a_3$. Complete this to a symplectic basis $(a_i, b_i)_{i=1}^3$ of $X$. The winding number $v(\cdot)$ of these curves is valued in $\frac{1}{3}\bZ$ and the Arf invariant of $X$ is
\[ \sum_{i=1}^3 (3v(a_i)+1)(3v(b_i)+1) = (v(a_3)+1)(3v(b_3)+1)+\sum_{i=1}^2 3(v(a_i)+1)(v(b_i)+1) \]
where we have used that $(3x+1) = 3(x+1) = (x+1)$ mod $2$ for any integer $x$. But the expression on the right is the Arf invariant for $(Y, \omega)$ which has a symplectic basis given by the three lifts of $(a_i, b_i)_{i=1}^2$ together with the unique lift of $b_3$ and any of the three lifts of $a_3$.

Recall that $\bC_\chi$ is the local system associated to $\chi$. Then,
\[ f_* \cO_Y = \cO_X \otimes (\bC \oplus \bC_\chi \oplus \bC_\chi^{\otimes 2}) = \cO_X \oplus K_X(-4p) \oplus K_X^{\otimes 2}(-8p), \] 
which is just a way of saying that holomorphic functions defined on open $\bZ/3$-invariant subsets of $Y$ can be decomposed into eigenspaces for the deck group action. The Arf invariant of $(Y, \omega)$ is the parity of
\[ h^0(\cO_Y(2Q)) = h^0(f_* \cO_Y(2Q)) = h^0(\cO_X(2p)) + h^0(K_X(-2p)) + h^0(K_X^{\otimes 2}(-6p)).\]
By Riemann-Roch, the sum of the first two summands is even.
\end{proof}

\begin{lem}\label{L:QDsAtp}
If $(X, w)$ has odd Arf invariant, then $h^0(K_X^{\otimes 2}(-6p)) = 1$ and $h^0(K_X^{\otimes 2}(-7p)) = 0$. If $(X, w)$ has even Arf invariant, then $h^0(K_X^{\otimes 2}(-6p)) = 0$.
\end{lem}
\begin{proof}
Since, by primitivity, $K_X \ne 4p$, it follows that $X$ has no quadratic differential $w'$ that vanishes to order at least (and hence exactly) $8$ at $p$ since then $\frac{w}{w'}$ would be a holomorphic $1$-form with divisor $4p$. So $h^0(K_X^{\otimes 2}(-8p)) = 0$.

If $X$ had a quadratic differential $w'$ that vanished to order exactly $7$ at $p$ and to order $1$ at $q \ne p$, then $\frac{w}{w'}$ would be a meromorphic $1$-form with divisor $5p-q$, which is a violation of the residue theorem. So $h^0(K_X^{\otimes 2}(-6p)) \leq 1$. The claim now follows from Lemma \ref{L:Arf}.
\end{proof}

By Lemma \ref{L:QDsAtp}, if $(X, w)$ has even Arf invariant, then $h^0(K_X(-3p)) = 0$. We suppose now that $(X, w)$ has odd Arf invariant, which implies (by Lemma \ref{L:QDsAtp}) that $X$ has a quadratic differential $Q$ that vanishes to order exactly $6$ at $p$. We will see that the components of $\Omega^3 \cM_3(12)$ with odd Arf invariant are distinguished by whether or not this quadratic differential is the square of a holomorphic $1$-form, i.e. whether $h^0(K_X(-3p))$ is $0$ or $1$.

\begin{prop}\label{P:Component1}
There is a primitive component of $\Omega^3 \cM_3(12)$ consisting of generically nonhyperelliptic surfaces, with odd Arf invariant, and so that $h^0(K_X(-3p)) = 0$ generically. 
\end{prop}
\begin{proof}
A nonhyperelliptic genus $3$ surface $X$ can be given by a plane quartic, i.e. by a homogeneous degree four polynomial $F \in \bC[x,y,z]$. A $k$-differential $G$ on the plane quartic is determined by a homogeneous polynomial $G$ of degree $k$. Primitivity simply means that $G$ is not the power of a lower degree polynomial. Without loss of generality fix a point $p = [0:0:1]$ on a plane quartic. We can dehomogenize $F$ to write it as $f(x,y) \in \bC[x,y]$, a degree $4$ polynomial. Now $p$ corresponds to $(0,0)$. We can apply projective transformations to ensure that $T_p X$ is the $x$-axis, which simply says that $f(x,y) = y + h.o.t.$ where \emph{h.o.t.} will mean ``higher order terms''.

The point $p$ has the property\footnote{In classical plane algebraic geometry, this called being a \emph{flex} point.} that $h^0(K_X(-3p)) = 1$ if the tangent plane is tangent to $X$ at $p$ to order $2$. The tangent line is given by $t \ra (t,0)$, so the requirement is that $f(x,y) = y + axy + by^2 + h.o.t.$, i.e. the coefficient of $x^2$ vanishes. Conversely, $h^0(K(-3p)) = 0$ if the coefficient of $x^2$ is nonzero. This shows that if this cohomology vanishes at any point in a component of $\Omega^3 \cM_3(12)$, it vanishes generically in that component (which also follows by upper semicontinuity). 

When $f(x,y) = y  + h.o.t.$ we can parameterize the locus where $f = 0$ near $(0,0)$ by $(x, \phi(x))$ where $\phi(x) = \sum_{n \geq 0} a_n x^n$ can be computed by choosing the coefficients recursively to ensure that $f(x, \phi(x)) = 0$. Given a plane curve defined by a polynomial $g \in \bC[x,y]$ that passes through $p = (0,0)$, the intersection multiplicity with $f$ is defined as the order of vanishing of $g(x, \phi(x))$ at $x = 0$. 

To summarize:
\begin{itemize}
    \item Producing a nonhyperelliptic curve $X \in \cM_3$ with a point $p$ so that $h^0(K_X(-3p)) = 0$ requires specifying a smooth plane quartic defined by $F \in \bC[x,y,z]$ so that the dehomogenization $f(x,y) = F(x,y,1) = y+axy+by^2 +cx^2 + h.o.t.$ for some $a,b,c \in \bC$ with $c \ne 0$.
    \item Producing a cubic differential $w$ on $X$ with divisor $12p$ requires producing a cubic $g \in \bC[x,y]$ a cubic so that $g(x, \phi(x))$ vanishes to order $12$ at $(0,0)$. Primitivity is automatic by $h^0(K_X(-3p)) = 0$.
    \item Guaranteeing that $(X, w)$ has odd Arf invariant amounts to producing a quadratic differential that vanishes to order at least $6$ at $p$, i.e. a quadratic $h \in \bC[x,y]$ so that $h(x, \phi(x))$ vanishes to order $6$ at $(0,0)$.
\end{itemize}
The following are the desired polynomials:
\[
F(X,Y,Z)
= X^{4} - X Y^{3} + X^{3} Z - Y^{3} Z - X^{2} Z^{2} - X Y Z^{2} - Y^{2} Z^{2} + Y Z^{3},
\]
and
\[
f(x,y)
= F(x,y,1)
= x^{4} - x y^{3} + x^{3} - y^{3} - x^{2} - x y - y^{2} + y,
\]
so
\[
\phi(x)
= x^{2} + x^{6} + 2x^{7} + 4x^{8} + 8x^{9} + 19x^{10} + 44x^{11} + 101x^{12} + O(x^{13}),
\]
and
\[ g(x,y)  = 2 x^{3} - y^{3} - x^{2} - 2 x y + y, \]
and 
\[h(x,y) = x^{2} - y.\]
%
\end{proof}

\begin{prop}\label{P:ConstructionOfSporadicComponentK3}
There is a primitive component of $\Omega^3 \cM_3(12)$ consisting of generically nonhyperelliptic surfaces, with odd Arf invariant, and so that $h^0(K(-3p)) = 1$. 
\end{prop}
\begin{proof}
By Chen-M\"oller \cite{ChenMoller-Exceptional}, there is a (projectived) component $\cQ \subseteq \bP(\Omega^2 \cM_3(9,-1))$ whose points $(X, Q)$ have divisor $9p-q$ so that $h^0(K_X(-3p)) = 1$. In other words, $X$ has a holomorphic $1$-form $\omega$ with divisor $3p+r$. This form cannot vanish to order $4$ at $p$ since then $\omega^2/Q$ would be a degree $1$ map to $\mathbb{P}^1$. So $p \ne r$ and $\omega Q$ is a cubic differential with divisor $12p + r - q$. As $\cQ$ is a projectivized component, we may write $\cQ \subseteq \cM_{3,2}$. Let $\cC_{\cQ}$ be the restriction of the universal curve to $\cQ$ and consider the two section $s_r, s_q: \cQ \ra \cC_{\cQ}$ given by marking $r$ and $q$ respectively. It follows from the Weierstrass preparation theorem that the locus $\cL \subseteq \cQ$ where the two holomorphic sections coincide is, after deleting the singular locus, either empty or a pure codimension one complex analytic subspace. If $\cL$ is nonempty, then it has dimension $\dim \cQ - 1 = \dim \bP( \Omega^3 \cM_3(12))$ Therefore, the image of $\cL$ in $\Omega^3 \cM_3(12)$ under the map sending $(X, Q)$ to $(X, \omega Q)$ is a union of components where $h^0(K_X(-3p)) = 1$ and where $\omega Q$ is primitive since we have already shown that $X$ has no holomorphic $1$-form that vanishes to order $4$ at $p$. Moreover, by Lemma \ref{L:QDsAtp}, $(X, \omega Q)$ has odd Arf invariant.

So it remains to show that $\cL$ is nonempty. It suffices to construct a smooth plane quartic $X$ with a flex point $p$, i.e. so that $h^0(K_X(-3p)) = 1$, that is not a hyperflex point, i.e. $h^0(K_X(-4p)) = 0$, and a cubic differential $w$ with divisor $12p$. To see this, the flex and hyperflex conditions give a holomorphic $1$-form $\omega$ on $X$ with divisor $3p+q$ where $p \ne q$ and so $w/\omega$ is a quadratic differential with divisor $9p-q$. As in Proposition \ref{P:Component1}, to produce $(X,p)$ it suffices to produce a quartic homogeneous polynomial $F \in \bC[x,y,z]$ that determines a smooth plane curve and whose dehomogenization $f \in \bC[x,y]$ can be written as $f(x,y) = y + h.o.t.$ where the coefficient of $x^2$ is zero and the coefficient of $x^3$ is nonzero. The cubic differential is then specified by a cubic $g \in \bC[x,y]$ so that $g(x, \phi(x))$ vanishes to order $12$ at $(0,0)$. Set
\[
F(X,Y,Z)
= -X^{4} + X^{3}Y - XY^{3} - Y^{4} - X^{3}Z - X^{2}YZ - XY^{2}Z - Y^{3}Z - XYZ^{2} - Y^{2}Z^{2} + YZ^{3},
\]
and so 
\[
f(x,y)
= F(x,y,1)
= -x^{4} + x^{3}y - xy^{3} - y^{4} - x^{3} - x^{2}y - xy^{2} - y^{3} - xy - y^{2} + y,
\]
and so 
\[
\phi(x)
= x^{3} + 2x^{4} + 3x^{5} + 5x^{6} + 11x^{7} + 27x^{8} + 66x^{9}
+ 162x^{10} + 407x^{11} + 1043x^{12} + O(x^{13}),
\]
and set
\[
g(x,y)
= -x^{3} + x^{2}y - y^{3} - 2xy - y^{2} + y.
\]
\end{proof}

\begin{proof}[Proof of Theorem \ref{T:Genus3MinimalK=3}:]
A primitive component of $\Omega^3 \cM_3(12)$ with even Arf invariant exists by attaching an element of $\Omega^3 \cM_0(-4, -14, 12)$ to the appropriate singularities of a surface with even Arf invariant in $\Omega^3 \cM_2(8,-2)$ (which exists by Sublemma \ref{L:ClassificationForGenus2}) and then smoothing the nodes. A similar construction proves that primitive components with even Arf invariants exist for all strata under consideration. It now suffices to show that there are at least two primitive components in these strata with odd Arf invariant. Let $X$ be a cubic differential in one of these strata. We will show that the generic value of $h^0(K_X(-P))$ distinguishes components with odd Arf invariant. 

To produce components where $h^0(K_X(-P)) = 0$ take a surface in the primitive component of $\Omega^3 \cM_3(12)$ where $h^0(K_X(-3p)) = 0$ generically (which exists by Proposition \ref{P:Component1}) and glue to the zero a sphere in $\Omega^3 \cM_3(-18, \kappa)$. Smoothing the singularities creates the desired component since the Arf invariant and primitivity are unchanged under this operation. Upper semicontinuity shows that $h^0(K_X(-P)) = 0$ after smoothing.

Now we will produce a component where $h^0(K_X(-P)) = 1$ for all surfaces $(X, P)$ in the (projectivized) component. By Chen-M\"oller \cite{ChenMoller-Exceptional}, there is a component $\cQ$ of the partial compactification of $\Omega^2 \cM_2(3,3,3,-1)$ in the multiscale compactification consisting of multiscale quadratic differentials with exactly one top vertex, which, additionally, has genus three and where $h^0(K_X(-P)) = 1$ where $3P-D$ is the divisor of the differential on the top vertex. Let $(X, Q) \in \cQ$ be the pair of the genus three top vertex and the quadratic differential $Q$ on it. Let $\omega$ be the holomorphic $1$-form on $X$ that vanishes at $P$. As in Proposition \ref{P:ConstructionOfSporadicComponentK3}, $(X, \omega Q)$ is a cubic differential that vanishes to order at least $4$ at points in $P$ and where $h^0(K_X(-P)) = 1$. We saw in Proposition \ref{P:ConstructionOfSporadicComponentK3} that the locus where $\mathrm{div}(\omega Q) = 4P$ is nonempty and hence it is codimension one, which constructs our desired components. 
\end{proof}

\section{Proof of Theorem \ref{T2}}

In this section, we only consider strata parameterizing finite area surfaces i.e. all orders of singularities are strictly greater than $-k$. Given a primitive component $\cC$ of a stratum $\Omega^k \cM_g(\kappa)$ let $\cM_\cC$ be the smallest $\mathrm{GL}(2, \bR)$-orbit closure containing the locus $\cN_\cC$ of all holonomy covers of $k$-differentials in $\cC$. 

Consider a surface $X$ with singularities $\Sigma$ in an invariant subvariety $\cM$ and with a collection of parallel cylinders $\bfC$. The \emph{standard shear} of $\bfC$ is $\sum_{C \in \bfC} h_C \gamma_C^*$ where $h_C$ is the height of $C$ and $\gamma_C^*$ is the Poincare dual of its core curve. A cylinder $C$ is \emph{$\cM$-free} if the standard shear in $\bfC := \{C\}$ belongs to the tangent space of $\cM$. Two cylinders $C$ and $D$ are called \emph{twins} if the equation $h_C = h_D$ holds in $\cM$ and the standard shear in $\bfC := \{C, D\}$ belongs to the tangent space of $\cM$. Finally, given a cylinder $C$ on $X$, its \emph{$\cM$-equivalence class} is the maximal collection of cylinders on $X$ that are \emph{$\cM$-parallel} to $C$, i.e. parallel to $C$ on $X$ and on all nearby surfaces in $\cM$. 

\begin{lem} \label{L:Mfreecyl}
The lift of a Euclidean cylinder to a surface in $\cN_\cC$ is $\mathcal{M_\cC}$-free (resp. has a twin) when $k$ is odd (resp. even). 
\end{lem}
\begin{proof}
By the cylinder finiteness theorem \cite[Theorem 5.1]{MirWri}, there is a constant $c_{\cM_\cC}$ so that any two $\cM_\cC$-parallel cylinders have a ratio of circumferences bounded above by $c_{\cM_\cC}$. Let $X \in \cC$ have a Euclidean cylinder $C$ which we normalize to have circumference $1$ and which, by increasing its height, we may suppose has height greater than $c_{\cM_\cC}$.

Let $Y$ be the holonomy cover of $X$ and let $\bfC$ be the lifts of $C$ to $Y$. Arbitrarily fix some $\wt{C} \in \bfC$. Let $X_t$ be the family of surfaces in $\cC$ formed by sending the height of $C$ to $\infty$ while fixing its complement. Let $Y_t$ be the corresponding holonomy covers and let $\sigma$ be the corresponding tangent direction at $Y$. Let $\bfD$ be the $\cM_\cC$-equivalence class of $\wt{C}$.

By the twist space decomposition theorem \cite[Proposition 4.20]{ApisaWrightHighRank}, since $\sigma$ does not change the period of the core curve of $\wt{C}$, the restriction of $\sigma$ to $\bfD$ belongs to $T_Y \cM_\cC$. However, any cylinder in $\bfD$ not contained in $\bfC \cap \bfD$ cannot cross a cylinder in $\bfC$ (if it did, its circumference would exceed $c_\cM$) and so it is unchanged by $\sigma$. So the restriction of $\sigma$ to $\bfD$ is the standard shear on $\bfC \cap \bfD$. This intersection only contains $\wt{C}$ if $k$ is odd and $\wt{C}$ and a twin when $k$ is even. 
\end{proof}

Given a surface $X\in \cC$, let $f: Y \ra X$ be its holonomy cover. Let $P$ be the collection of singularities on $X$, and let $Q$ be the collection of their preimages on $Y$.  Let $\rho: \pi_1(X-P) \ra \bZ/k$ be the holonomy homomorphism. Let $t: Y \ra Y$ generate the deck group.

\begin{lem}\label{L:HypOrbClosure}
If $\cM_{\cC}$ is contained in a hyperelliptic locus, then either the hyperelliptic involution is $t^{k/2}$ or $\cC$ is a hyperelliptic component. Conversely, if $\cC$ is a hyperelliptic component, then $\cM_{\cC}$ is contained in an abelian or quadratic double. 
\end{lem}
\begin{proof}
For the first claim, the hyperelliptic involution is central in the group of holomorphic self-maps of a genus at least two\footnote{The cases where $Y$ has genus $1$ correspond to the cases where $X$ is either a pillowcase or the double of an equilateral, isosceles-right, or $30-60-90$ triangle. The claim is clear in those cases.} Riemann surface. So if $J$ is the hyperelliptic involution on $Y$, it descends to an involution $j$ on $X$. The involution $j$ is nontrivial if $J$ is not a power of $t$, in which case $X/j$ is a sphere since it is the image of $Y/J$. 

For the second claim, a holomorphic involution $j$ on $X$ lifts to a holomorphic involution on the regular cover $Y$ determined by $\rho$ if and only if $\rho \circ j = \rho$. Since $j$ preserves the winding number of curves on $X$ this holds. 
\end{proof}


\begin{lem}\label{L:CertificateToConclude}
Let $(X_t)_{t \geq 0}$ be a family of surfaces in $\cC$ that converge to a primitive boundary point $Z \in \cD$. Suppose that the space of vanishing cycles on $X := X_0$ is generated by a single arc $s$ between two singularities and that there is a Euclidean cylinder $C$ on $X$ that has algebraic intersection number $1$ with $s$. Then $\cM_\cC$ is as big as possible if $\cM_\cD$ is.
\end{lem}
\begin{proof}
Let $\cM$ be the quadratic double containing $\cM_\cC$ if $k$ is even, the smallest locus of doubles containing $\cM_\cC$ if $k$ is odd and $\cC$ is hyperelliptic, and the ambient stratum otherwise. 

Let $S$ (resp. $\bfC$) be the lift of $s$ (resp. $C$) to $Y$. Let $Y_t$ (resp. $W$) be the lift of $X_t$ (resp. $Z$) to the multiscale compactification of $\cN_{\cC}$. Then $S$ generates the space of vanishing cycles of $(Y_t)_{t \geq 0}$. Let $p: H_1(Y-Q; \bC) \ra \bC^k$ be the map that sends $\gamma$ to $(\sigma(\gamma))_{\sigma \in S}$ where $\sigma(\gamma)$ is the intersection pairing. When $k$ is odd, Lemma \ref{L:Mfreecyl} implies that the restriction of $p$ to $T_Y \cM_{\cC}$ is a surjection. When $k$ is even, Lemma \ref{L:Mfreecyl} implies that the restriction of $p$ to $T_Y \cM$ sends $v-t^{k/2}v$ to $(\sigma(v)-(t^{k/2}\sigma)(v))_{\sigma \in S}$. These are the maps $f: S \ra \bC$ for which $f \circ t^{k/2} = -f$. By Lemma \ref{L:Mfreecyl}, the image of $p$ when restricted to either $T_Y \cM_{\cC}$ or $T_Y \cM$ is the same. So the claim follows if we can show that $\ker(p) \cap T_Y \cM_{\cC} = \ker(p) \cap T_Y \cM$. 

By the Mirzakhani-Wright boundary formula \cite{MirWri}, $\ker(p) \cap T_Y \cM_{\cC}$ contains a subspace that can be identified with $T_W \cM_{\cD}$. 

When $k$ is odd, this implies that $\cM_{\cC}$ is at most codimension one in the ambient stratum if $\cM_{\cD}$ is as big as possible. By Mirzakhani-Wright \cite{MirWri2}, $\cM_{\cC}$ is either a stratum or a codimension one hyperelliptic locus. We are now done by Lemma \ref{L:HypOrbClosure}.

When $k$ is even, $\cM_{\cD}$ is the quadratic double that contains the holonomy cover of $Z$, so 
\[  T_Z \del \cM = T_Z \cM_{\cD} \subseteq T_Z \del \cM_{\cC}  \subseteq T_Z \del \cM. \]
We're done since
\[ \ker(p) \cap T_Y \cM_{\cC} \cong T_Z \del \cM_{\cC} \cong T_Z \del \cM \cong \ker(p) \cap T_Y \cM. \]
\end{proof}

The \emph{cylinder submodule} $C_Y \subseteq H_1(Y-Q)$ is the submodule generated by curves which, up to deformations within the locus of holonomy covers, project to core curves of Euclidean cylinders.

\begin{lem}\label{L:CylinderSubmodule}
If $X$ belongs to a primitive component $\Omega_r^k \cM_1(a,-a)$ where $a \in \{2, \hdots, k-1\}$, then $H_1(Y-Q)$ splits as $C_Y \oplus \bZ$ where a generator of the $\bZ$ summand is invariant under the deck group.
\end{lem}
\begin{proof}
Let $(E, P)$ be the unit square torus with a set $P$ of two marked points consisting of the origin $0$ and the point $p = \left( \frac{1}{n}, 0 \right)$ where $n = \frac{a}{r}$. Let $X$ be the corresponding point in the stratum of $k$-differentials. As observed in the proof of Corollary \ref{C:GenusOneCyl}, the core curves of cylinders in one-cylinder directions on the abelian differential $(E,P)$ correspond to core curves of Euclidean cylinders at some point in the stratum of $k$-differentials.

On $(E, P)$ the horizontal direction and the direction parallel to $(1,n)$ are both one-cylinder directions. Figure \ref{F:DeformationRetract} shows that their core curves, $\gamma_1$ and $\gamma_2$ respectively, can be extended by adding an element $\gamma_3$ to a generating set of the rank $3$ free group $\pi_1(E-P)$.

\begin{figure}[h]
    \centering
    \includegraphics[width=.70\linewidth]{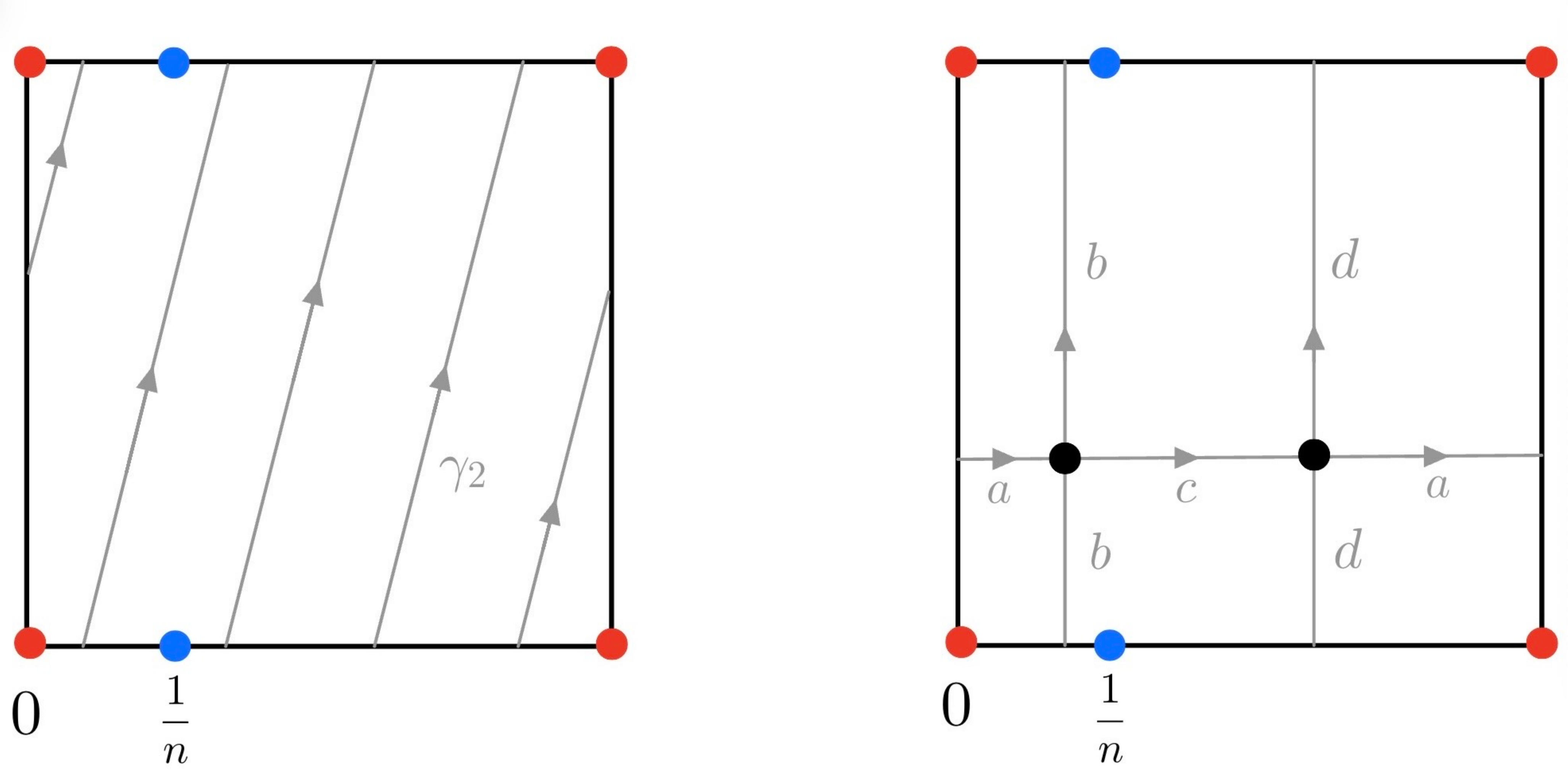}
    \caption{The curve $\gamma_2$ is shown on the left. A skeleton onto which $E-P$ deformation retracts is shown on the right. Contracting $c$ to a point, a basis for $\pi_1(E - P)$ is $(a,b,d)$. In this basis, $\gamma_1 = a$ and $\gamma_2 = d^{n-1}ab$. Together with $d$ this generates $\pi_1(E-P)$.}
    \label{F:DeformationRetract}
\end{figure}

Let $B$ be the rose with $3$ petals, each labeled by $\gamma_i$ for some $i \in \{1,2,3\}$. We identify its fundamental group with $\pi_1(E - P)$ by sending the $i$th petal to $\gamma_i$. Let $B'$ be the cover of $B$ determined by the holonomy cover $f:Y \ra X$. We therefore can identify the fundamental group of $B'$ with $\pi_1(Y - Q)$. Under this identification, the cylinder subspace contains the lift of the first two petals. Note that $B'$ is a graph with $k$ vertices. Since $\rho(\gamma_1) = 0 =\rho(\gamma_2)$, for each vertex, there is a unique lift of $\gamma_1$ and $\gamma_2$ that begins and ends at it. There is one lift of $\gamma_3$ and it is, consequently, $t$-invariant. These $2k+1$ lifts generate $H_1(Y-Q)$ and $C_Y$ contains the rank $2k$ submodule generated by lifts of $\gamma_1$ and $\gamma_2$. It remains to see that $C_Y$ coincides with this submodule.

Let $C_X$ be the cylinder submodule of $X$. Since $f_*(C_Y) = C_X$ it suffices to show that $C_X$ has rank two. Using the horizontal vector field on $\Sigma_{1,2} := E-P$ we have a section $s: \Sigma_{1,2} \ra UT\Sigma_{1,2}$ which provides a $\Gamma_1(d)$-equivariant section $s_*: H_1(\Sigma_{1,2}) \ra H_1(UT\Sigma_{1,2})$. The cylinder core curves on $X$ are in the kernel of $w \circ s_*$, where $w: H_1(UT\Sigma_{1,2}) \ra \bZ$ is the $k$-framing, which completes the claim.
\end{proof}

\begin{cor}\label{C:OrbitClosureBaseCase}
Given a primitive component $\cC$ of $\Omega_r^k \cM_1(a,-a)$, $\cM_\cC$ is as big as possible.
\end{cor}
\begin{proof}
By Lemma \ref{L:CylinderSubmodule}, $H_1(Y-Q) = C_Y \oplus \bZ$ where the generator of the $\bZ$ summand is $t$-invariant. If $k$ is even, then $Y$ belongs to a quadratic double $\cQ$ and its holonomy involution on $Y$ is $t^{k/2}$. In this case, the projection from $H_1(Y-Q; \bC)$ to $T_Y \cQ$ is given by sending $v$ to $\frac{v-t^{k/2}\cdot v}{2}$. Note that the generator of the $\bZ$-summand is in the kernel of this projection. 

By Lemma \ref{L:Mfreecyl}, $T_Y \cM_{\cC}$ has codimension at most one in $H_1(Y-Q)$ when $k$ is odd and coincides with $T_Y \cQ$ when $k$ is even. If $T_Y \cM_\cC$ has codimension one, then it is full rank since its image in absolute homology is symplectic by \cite{AEM}. Mirzakhani-Wright \cite{MirWri2} implies that $\cM_{\cC}$ is a hyperelliptic locus, which occurs precisely when $\cC$ is hyperelliptic by Lemma \ref{L:HypOrbClosure}. When $\cC$ is hyperelliptic, Lemma \ref{L:HypOrbClosure} implies that $\cM_\cC$ has positive codimension unless the ambient component of the stratum of abelian differentials is hyperelliptic. But this does not occur since the preimages of the two singularities on the $k$-differential have different orders.
\end{proof}

A genus zero stratum containing a surface with a Euclidean cylinder will be called \emph{cylindrical}. These are classified in Lemma \ref{L:Genus0Cylinder}.

\begin{lem}\label{L:CylinderSubmoduleGenus0}
If $X$ belongs to a genus zero cylindrical stratum $\cC$ with four singularities, then $H_1(Y-Q; \bQ) = (C_Y \otimes \bQ) \oplus \bQ^{2}$ where generators of the $\bQ^{2}$ summand may be chosen to be $t$-invariant and hence trivial elements of $H_1(Y; \bQ)$. Moreover, $\cM_\cC$ is as big as possible. 
\end{lem}
\begin{proof}
We begin by deducing the second claim from the first. If $k$ is even, we're done by Lemma \ref{L:Mfreecyl}. If $k$ is odd, then $\cM_{\cC}$ is full rank and hence coincides with the ambient stratum unless $\cC$ is hyperelliptic (by Lemma \ref{L:HypOrbClosure}), in which case $\cC = \Omega^k \cM_0(a,a,-k-a,-k-a)$ where $\gcd(a,k) = 1$. The hyperelliptic involution on $X$ exchanges the two singularities of equal order. Each singularity is fully ramified so the hyperelliptic involution $J$ on $Y$ exchanges the corresponding pairs of points in $Q$. Since $|Q/J|=2$, the hyperelliptic locus $\cM_{\cC}$ is codimension one. 

Let $(a,c,b)$ be loops around three points in $P$. These are generators for the rank three free group $\pi_1(X-P)$. Since the stratum is cylindrical we may suppose that $a$ and $c$ are chosen so that $d := ac$ is a cylinder core curve. The pure braid group preserves the holonomy homomorphism and there is a pure braid that sends $ac$ to $ac[b,a]$, so $ac[b,a]$ is also a cylinder core curve. We will abuse notation and use $a,b,c,d$ for both loops and their holonomies in $\bZ/k$. 


Let $B$ be the rose with $3$ petals, each labeled by $(a,b,d)$. We identify its fundamental group with $\pi_1(X - P)$ by sending the petals to $a$, $b$, and $d$ respectively. Let $B'$ be the cover determined by the holonomy cover $Y \ra X$. We can identify the fundamental group of $B'$ with $\pi_1(Y - Q)$. Under this identification, the cylinder submodule contains the lifts of $d$ and $d[b,a]$ and hence the lifts of $[b,a]$. Note that $B'$ is a graph with vertices labeled by elements of $\bZ/k$. Each vertex has a lift of $d$ that begins and ends at the vertex. It also has a lift of $a$ (resp. $b$) connecting the vertex $v$ to $v+a$ (resp. $v+b$). 

Let $R = \bZ[\bZ/k] = \bZ[t]/(t^k-1)$. Let $R_\bQ := R \otimes \bQ$. Letting $A$ (resp. $B$, $D$) be the lift of the $a$ (resp. $b$, $d$) edge that joins vertex $0$ to vertex $a$ (resp. $b$, $0$) we have that as a $1$-chain, the lift of $[b,a]$ based at $0$ is $B+t^bA - t^{a}B - A = (1-t^a)B - (1-t^b)A$. The subspace $\cL$ of $C_1(Y-Q; \bQ)$ generated by these lifts is $R_\bQ \cdot \left( (1-t^b) \cdot A - (1-t^a) B\right)$. The boundary map from $C_1(Y-Q; \bQ)$ to $R_\bQ \cong C_0(Y-Q)$ sends $A$ to $t^a-1$, $B$ to $t^b-1$, and $D$ to zero. So $H_1(Y-Q; \bQ)$ corresponds to elements $r_a A + r_b B + r_d D$ such that $r_a (1-t^a) + r_b(1-t^b) = 0$ where $r_a, r_b, r_d \in R_\bQ$. Set $e =\gcd(a,b)$. To compute the kernel, we consider the following short exact sequence
\[ 0 \ra \bQ[t] \cdot \left( \frac{1-t^b}{1-t^e}, \frac{t^a-1}{1-t^e} \right) \ra \bQ[t]^2 \xrightarrow{(r_a, r_b) \ra r_a(t^a-1) + r_b(t^b-1)} \bQ[t] \cdot (1-t^e) \ra 0.  \]
Since $\gcd(e,k)=1$, the kernel of the quotient map $\bQ[t] \cdot (1-t^e) \ra \bQ[t]/(t^k-1)$ is $\bQ[t] \cdot \left( \frac{(t^k-1)(t^e-1)}{t-1} \right)$. The kernel of the composition $\bQ[t]^2 \ra \bQ[t]/(t^k-1)$ is then 
\[ \bQ[t] \cdot \left( \frac{1-t^b}{1-t^e}, \frac{t^a-1}{1-t^e} \right) + \bQ[t] \cdot \left( \frac{t^k-1}{1-t},0 \right) + \bQ[t] \cdot \left( 0,\frac{t^k-1}{1-t} \right) \]
Reducing this kernel mod $(t^k-1)$ we have,
\[ H_1(Y-Q; \bQ) = \frac{1}{1-t}\cL \oplus R_\bQ \cdot D \oplus \bQ \cdot \left(\sum_{i=0}^{k-1} t^i\right)A \oplus \bQ \cdot \left(\sum_{i=0}^{k-1} t^i\right)B.\]
Using that an element of $R_\bQ$ can be written as $q(\sum_{i=0}^{k-1} t^i) + r(1-t)$ where $q \in \bQ$ and $r \in R_\bQ$, we have 
\[ H_1(Y-Q; \bQ) = \cL \oplus R_\bQ \cdot D \oplus \bQ \cdot \left(\sum_{i=0}^{k-1} t^i\right)A \oplus \bQ \cdot \left(\sum_{i=0}^{k-1} t^i\right)B.\]
\end{proof}


\begin{proof}[Proof of Theorem \ref{T2} when $g \geq 1$:]
First assume that $g =1$. Let $\cC$ be the component. Induct on the number of singularities $n$ with Corollary \ref{C:OrbitClosureBaseCase} serving as the base case. We now suppose that $n > 2$. 

By Proposition \ref{P:GenusOneMerging}, if $\cC$ is nonhyperelliptic we can simply merge a zero and a pole. This operation has a single vanishing cycle $s$ joining the zero and pole. 

\begin{sublem}\label{L:CylinderCrossing}
If $\cC$ is nonhyperelliptic, we may choose $s$ so that there is a curve on $X$ that crosses $s$ exactly once and can be deformed into a core curve of a simple Euclidean cylinder. 
\end{sublem}
\begin{proof}
Using Proposition \ref{P:GenusOneMerging}, we may merge any two zeros and any two poles while preserving primitivity and nonhyperellipticity. Because cylinders and saddle connections are preserved under small deformations, this observation implies that it suffices to prove the claim when $n = 3$. Let $(a,b,c)$ be the three singular points on a torus $E$ and let $n_a,n_b,$ and $n_c$ be their respective orders. Suppose that $n_a$ has a different sign than $n_b$ and $n_c$ and identify $a$ with $0$. By Proposition \ref{P:GenusOneMerging} it is possible to merge $b$ and $c$ while preserving rotation number $r$. Up to applying an element of $\mathrm{SL}(2, \bR)$, we may suppose that merging $b$ and $c$ gives us a square torus $E$ with marked points $P = \{0, p\}$ where $p := (\frac{1}{e}, 0)$ is the point formed by merging $b$ and $c$ and $e = \frac{|n_a|}{r}$ (see Remark \ref{R:RotationNumberDefinition}). Since this stratum is not hyperelliptic $e > 2$. Since the stratum is primitive, $\gcd(r,k)=1$. Let $d = \gcd(n_a, n_b, n_c) = \gcd(n_a, n_c) = \gcd(er, n_c)$. Let $e'$ be the order of $\frac{n_a}{d}a + \frac{n_b}{d}b + \frac{n_c}{d} c$ in $E$. Then the rotation number satisfies $d = e'r$ by Remark \ref{R:RotationNumberDefinition}. 

The equation governing the location of $b$ and $c$ in $E$ is $n_b b + n_c c = 0$. Suppose that $|n_b| \leq |n_c|$, which implies that $b$ ``moves faster than $c$.'' Starting from the point where $b = c = (\frac{1}{e}, 0)$, move $b$ in the negative horizontal direction to the origin. By the time that $b$ arrives at the origin, $c$ is at $(\frac{n_b+n_c}{n_c e}, 0)$. Then $e''$ is the order of this point where $\frac{n_a'}{e''}$ is $\frac{n_a}{n_c e}$ in lowest terms. By Remark \ref{R:RotationNumberDefinition}, the rotation number $r''$ of the stratum formed by merging $a$ and $b$ is $\frac{|n_c|}{e''}$, i.e. $|n_c| = e''r''$. Since $r \mid |n_c|$ write $|n_c| = rm$. Then $\frac{n_a}{n_c e} = -\frac{er}{erm} = -\frac{1}{m}$ so $m = e''$, i.e. $|n_c| = re''$, so $r = r''$. In other words, this merge does not change the rotation number. The vanishing cycle for this merge is the horizontal line connecting $0$ to $p$, which is crossed exactly once by a Euclidean cylinder (on the $k$-differential) as seen in Figure \ref{F:DeformationRetract}.
\end{proof}

By Sublemma \ref{L:CylinderCrossing} and Lemma \ref{L:CertificateToConclude}, we are done with genus one (using the induction hypothesis) when $\cC$ is nonhyperelliptic. 

If $\cC$ is hyperelliptic, then find a Euclidean cylinder $C$ on $X$ (by Theorem \ref{T:Cyl}). Send its height to zero to collapse a cross curve. The resulting $k$-differential belongs to a stratum $\cD$ where $\cM_{\cD}$ is as big as possible by the induction hypothesis (when the collapse is a simple merge) or Lemma \ref{L:CylinderSubmoduleGenus0} (when the collapse is a simple split, since hyperelliptic genus zero strata with four singularities are necessarily cylindrical). We are done by Lemma \ref{L:CertificateToConclude}. This completes the $g=1$ case.

Now suppose that $g \geq 2$. Induct first on genus and then dimension of the stratum with the previous $g=1$ case serving as the base case. Let $X$ be a surface in the component with simple Euclidean cylinder $C$ (which exists by Theorem \ref{T:Cylinder}). Collapsing it, as in the previous paragraph, produces a surface in a component $\cD$ where $\cM_{\cD}$ is as big as possible by the induction hypothesis. We're done by Lemma \ref{L:CertificateToConclude}. 
\end{proof}

\begin{proof}[Proof of Theorem \ref{T2}:]
It remains to handle cylindrical genus zero strata. Induct on $k$ with the base case of $k=2$ being trivial. Then induct on the number $n$ of singularities with Lemma \ref{L:CylinderSubmoduleGenus0} giving the base case. 

\begin{sublem}\label{SL:Genus0Merge}
Let $n \geq 5$. Two pairs of singularities can be merged to form a cylindrical stratum of finite area surfaces. Both merges may be taken to be nonhyperelliptic after forgetting marked points unless $\kappa = (-a, -a, -a, a-k, 2a-k)$. 
\end{sublem}
\begin{proof}
Let $\kappa = (a_1, \hdots, a_m, a_{m+1}, \hdots, a_n)$ be arranged so that $\sum_{i=1}^m a_i = -k$ and where $m \geq 3$ and $a_1 \geq a_2 \geq \hdots \geq a_m$. 

If $a_i + a_m \leq -k$, then $a_i < 0$. If this held for $a_1$, then $\sum_{i=1}^m a_i = (a_1+a_m) + \sum_{i=2}^{m-1} a_i < -k$, a contradiction. If it held for $a_2$, then the same reasoning implies that $a_1 > 0$ and so $a_1 + a_2 > a_2 > -k$. So we have two merges that remain in cylindrical strata of finite area $k$-differentials. All subsequently considered merges will be assumed to have this property. 

Suppose first that merging $a_i$ and $a_j$ is hyperelliptic after forgetting marked points and that the merged singularity is a marked point, i.e. $a_i = -a_j$. Assume that $a_i > 0$. Then there is some $a \in \{1,\hdots, k-1\}$ so that $\kappa = (a_i, -a_i, -a,a-k,-a,a-k)$. Thus, consider merging $a_i$ with $-a$ instead. This is nonhyperelliptic after forgetting marked points unless $a_i = a$ and so $\kappa = (a,-a,-a, a-k, -a, a-k)$ and primitivity implies that $\gcd(a,k) = 1$. So consider merging $a_i$ and $a-k$. We are done unless the stratum is $(a,-a,-a,-a,-a, -a)$ where $a = \frac{k}{2}$. Since $\gcd(a,k) = 1$, this only occurs when $k=2$ and $a=1$. However, Theorem \ref{T2} is trivial when $k=2$.

Suppose now that $n=5$ and that some merge is hyperelliptic, i.e. the new stratum is $(-a,-a,a-k,a-k)$ where $1 \leq a \leq k-1$ and where $a$ is the singularity formed by merging. Write $\kappa = (x, -x-a, -a, a-k, a-k)$ where we choose $x$ so that $x > -a$. Merge $x$ and $a-k$ to form $(x+a-k, -x-a, -a, a-k)$, which is nonhyperelliptic unless $x = k-2a$ (which is our exception). The result is a finite area since $x+a-k > -k$.
\end{proof}

There is a subtle point that merging the singularities may cause the stratum to cease to be primitive. 

\begin{sublem}\label{SL:DisconnectedDegeneration}
Suppose that merging two singularities forms a cylindrical stratum $\cD$ of primitive finite area $d$-differentials\footnote{Technically, these boundary objects are still $k$-differentials; just ones that are $(k/d)$th powers of $d$-differentials.} where $1 \ne d \mid k$. Then the corresponding boundary of $\cM_{\cC}$ contains $(\cM_\cD)^{k/d} = \cM_\cD \times \times \hdots \cM_\cD$ or $k$ is even, $d$ is odd, and it contains the sublocus where surfaces with indices $i$ and $i+k/2$ mod $k/d$ are identical.  
\end{sublem}
\begin{proof}
The boundary of $\cM_\cC$ then contains $\{(Z, \zeta Z, \hdots, \zeta^{k/d-1}Z) : Z \in \cN_{\cD} \}$ where $\zeta$ is a primitive $k$th root of unity and $\cN_{\cD}$ is the locus of holonomy covers of surfaces in $\cD$. Let $\cL$ be the boundary component of $\cM_{\cC}$ containing this locus. The projection of $\cL$ to each factor is $\cM_{\cD}$ (by the induction hypothesis on $k$). Two indices $i, j \in \{0, \hdots, k/d-1\}$ are \emph{in the same prime factor} if all the absolute periods on the projection of $\cL$ to the $i$th factor determine all the absolute periods on the projection of $\cL$ to the $j$th factor. By assumption, we may choose $Z \in \cN_{\cD}$ so that it contains the preimage of a Euclidean cylinder under the holonomy cover. By Lemma \ref{L:Mfreecyl}, $(Z, \zeta Z, \hdots, \zeta^{k/d-1}Z)$ either has an $\cL$-free cylinder or a cylinder with a twin on $t^{k/2} Z$, which is a different surface only when $k$ is even and $d$ is odd. Therefore, $i= j$ are in the same prime factor if and only if $i=j$ or $k$ is even, $d$ is odd, and $j = i + k/2$ mod $k/d$. The result now follows from the prime decomposition of Chen-Wright \cite[Theorem 1.3 (2)]{ChenWright}.
\end{proof} 

\begin{sublem}
The claim holds when $\kappa = (-a,2a-k,-a,-a, a-k)$ for some $a \in \{1,\hdots, k-1\}$ with $\gcd(a,k)=1$.
\end{sublem}
\begin{proof}
Label the singularities as $(p_1, \hdots, p_5)$. Merging $p_1$ and $p_2$ passes to the primitive hyperelliptic component $(a-k, -a, -a, a-k)$. Let $s$ be the vanishing cycle for the merge. At some point in the original stratum, the vanishing cycle is crossed by the cylinder core curve that separates $\{p_1, p_5\}$ from $\{p_2, p_3, p_4\}$. We are done by Lemma \ref{L:CertificateToConclude}.
\end{proof}

There is a basis for $\pi_1(X-P)$ consisting of loops $a_i$ around the $i$th puncture excluding the last one. These generate a free group $F_{n-1}$. Identify $X-P$ with its deformation retract to the rose $B$ with one basepoint $p$ and $n-1$ petals $a_1, \hdots, a_{n-1}$. The universal abelian cover $A$ of $B$ is the one specified by the map from $F_{n-1}$ to its abelianization. $A$ is a graph whose vertices can be identified with $\bZ^{n-1}$ and on which $\bZ^{n-1}$ acts. Let $e_i$ be the lifts of $a_i$ to the origin. The $1$-chains are $C_1(A) = \bigoplus_{i=1}^{n-1} \bZ[\bZ^{n-1}] \cdot e_i$. Write $\bZ[\bZ^{n-1}]$ as Laurent polynomials in $t_1, \hdots, t_{n-1}$. The chain map $C_1(A) \ra C_0(A)$ is the $\bZ[\bZ^{n-1}]$-module homomorphism sending $e_i$ to $1-t_i$. 

Let $s_1, s_2$ be the two (simple) vanishing cycles produced by the merges in Sublemma \ref{SL:Genus0Merge}. These do not intersect one another except at endpoints. Up to relabeling the punctures, we may suppose that $s_1$ joins the first and second puncture and therefore positively (resp. negatively) intersects $a_1$ (resp. $a_2$) exactly once and intersects none of $a_3, \hdots, a_{n-1}$. We make the analogous supposition on $s_2$ which will join either the second and third or third and fourth punctures. Let $S_{i,A}$ be the lifts of $s_i$ to the universal abelian cover of $X-P$. These elements have well-defined intersections pairings with elements of $C_1(A)$. For instance, there is a lift of $s_1$ whose intersection pairing sends $e_1$ to $1$ and $e_2$ to $-1$. In other words, this lift determines the map $\bigoplus_{i=1}^{n-1} \bZ[\bZ^{n-1}] \cdot e_i$ that sends $\sum_i p_i(\vec{t})e_i$ to $p_1(0)-p_2(0)$. Let $\mathrm{Ann}(S_{i,A})$ be the chains in $C_1(A)$ that pair trivially with every element of $S_{i,A}$. When $i=1$, these are precisely $\{ \sum_i p_i(\vec{t})e_i : p_1 = p_2\}$. 

Let $C$ be the $(\bZ/k)$-regular cover of the rose determined by the holonomy homomorphism $\rho$. The map on $1$-chains $C_1(A) \ra C_1(C)$ is then the map induced by $\rho$ from $\bigoplus_{i=1}^{n-1} \bZ[\bZ^{n-1}] \cdot e_i$ to $\bigoplus_{i=1}^{n-1} \bZ[\bZ/k] \cdot e_i$. We will work rationally in order to use that $\bQ[\bZ/k] = \bigoplus_{d \mid k} \bQ(\zeta_d)$ where $\zeta_d := \exp\left( \frac{2 \pi i}{d} \right)$. In other words $C_1(C) = \bigoplus_{d \mid k} C_1(C)_d$ where $C_1(C)_d := \bigoplus_{i=1}^{n-1} \bQ(\zeta_d) \cdot e_i$. Let $H_1(C)_d$ be the kernel of the map on $C_1(C)_d$ that sends $e_i$ to $(1-\zeta_d^{k_i})$ where $\rho(a_i) = \zeta_k^{k_i}$. Then $H_1(C) = \bigoplus_{d \mid k} H_1(C)_d$. The annihilator $\mathrm{Ann}_{1,d}$ in $H_1(C)_d$ of the lifts of $s_1$ to $Y-Q$ is the kernel of the morphism that sends $e_1$ to $1$, $e_2$ to $-1$, and all other $e_i$ to $0$.

$\mathrm{Ann}_{1,d}$ has codimension $1$ in $H_1(C)_d$ unless $(1,-1,0, \hdots)$ and $(1-\zeta_d^{k_1}, \hdots, 1-\zeta_d^{k_{n-1}})$ both lie on the same line and the second vector is nonzero. The collinearity condition forces the second vector to be zero, so $\mathrm{Ann}_{1,d}$ is always codimension $1$. 

The smallest subspace of $H_1(C)_d$ containing both $\mathrm{Ann}_{1,d}$ and $\mathrm{Ann}_{2,d}$ is $H_1(C)_d$ unless the two subspaces coincide. This implies, depending on whether $s_2$ joins the second and third or third and fourth puncture, that $(1-\zeta_d^{k_1}, \hdots, 1-\zeta_d^{k_{n-1}})$ is nonzero and in the span of $(1, -1, 0, \hdots)$ and either $(0, 1, -1, 0, \hdots)$ or $(0,0,1,-1, 0, \hdots)$. When the second vector is $(0,0,1,-1,0, \hdots)$ this forces $d$ to divide all $k_i$, contrary to our assumption. When the second vector is $(0,1,-1,0, \hdots)$ we have that $d \mid k_i$ for $4 \leq i \leq n-1$ and that $3 - \sum_{i=1}^3 \zeta_d^{k_i} = 0$, which again implies that $(1-\zeta_d^{k_1}, \hdots, 1-\zeta_d^{k_{n-1}}) = 0$. In particular, $\mathrm{Ann}_{1,d}$ and $\mathrm{Ann}_{2,d}$ always generate $H_1(C)_d$. 

By the induction hypothesis, Sublemma \ref{SL:DisconnectedDegeneration}, and the Mirzakhani-Wright boundary formula \cite{MirWri}, the tangent space to $\cM_{\cC}$ contains $\mathrm{Ann}_{1,d}$ and $\mathrm{Ann}_{2,d}$ for all $d \mid k$ when $k$ is odd and the antisymmetriziations of these when $k$ is even, i.e. the result of multiplication by $(1-t^{k/2})$. This shows that the tangent space is all of $H_1(C)$ when $k$ is odd and is $(1-t^{k/2})H_1(C)$ when $k$ is even. 
\end{proof}

We sketch the following proof since it is standard. 

\begin{proof}[Proof of Corollary \ref{C:Asymptotics}:] Counting trajectories on translation surfaces using the Siegel-Veech constant of its orbit closure is the work of Eskin-Masur \cite{EMa} and Eskin-Mirzakhani-Mohammadi \cite[Theorem 2.12]{EMM}. The passage from counting on the holonomy cover to counting on the underlying $k$-differential appears in various references, e.g. \cite[Lemma 8.6]{Apisa-Codim1Hyp}. The holonomy cover of a generic surface in $\cC$ has orbit closure $\cM_{\cC}$ by \cite[Lemma 3.1]{Aygun2025_PrimeOrderkDifferentials}. The large genus asymptotics of Siegel-Veech constants for connected strata of abelian differentials is Aggarwal \cite{Aggarwal-LargeGenus}, whose estimates involve an additive $O(1/h)$ term where $h$ is the genus of the abelian differential. The genus $h$ of the holonomy cover of a genus $g$ finite-area primitive $k$-differential satisfies $h \geq Ck(g+1)$ where $C$ is a constant independent of $k$ and $g$. In light of Theorem \ref{T2}, the hypotheses placed on $\cC$ are precisely those needed to guarantee that $\cM_{\cC}$ is a connected stratum of abelian differentials. The asymptotics in the statement of the corollary are, a priori, ``weak asymptotics'' in the sense of Eskin-Masur \cite{EMa}, but, using recently announced work of Solan can be upgraded to usual asymptotics.
\end{proof}

\bibliography{mybib}{}

@article {Tahar-Chamber,
    AUTHOR = {Tahar, Guillaume},
     TITLE = {Chamber structure of modular curves {$X_1(N)$}},
   JOURNAL = {Arnold Math. J.},
  FJOURNAL = {Arnold Mathematical Journal},
    VOLUME = {4},
      YEAR = {2018},
    NUMBER = {3-4},
     PAGES = {459--481},
      ISSN = {2199-6792,2199-6806},
   MRCLASS = {14G35 (14H15)},
  MRNUMBER = {3949813},
MRREVIEWER = {Gregorio\ Baldi},
       DOI = {10.1007/s40598-019-00099-7},
       URL = {https://doi-org.ezproxy.library.wisc.edu/10.1007/s40598-019-00099-7},
}

@article {EskinOkounkov,
    AUTHOR = {Eskin, Alex and Okounkov, Andrei},
     TITLE = {Asymptotics of numbers of branched coverings of a torus and
              volumes of moduli spaces of holomorphic differentials},
   JOURNAL = {Invent. Math.},
  FJOURNAL = {Inventiones Mathematicae},
    VOLUME = {145},
      YEAR = {2001},
    NUMBER = {1},
     PAGES = {59--103},
      ISSN = {0020-9910,1432-1297},
   MRCLASS = {32G15 (05A17 11F23 37A25 57M12)},
  MRNUMBER = {1839286},
MRREVIEWER = {Christopher\ M.\ Judge},
       DOI = {10.1007/s002220100142},
       URL = {https://doi-org.ezproxy.library.wisc.edu/10.1007/s002220100142},
}

@article {AEZ,
    AUTHOR = {Athreya, Jayadev S. and Eskin, Alex and Zorich, Anton},
     TITLE = {Right-angled billiards and volumes of moduli spaces of
              quadratic differentials on {$\Bbb C\rm P^1$}},
      NOTE = {With an appendix by Jon Chaika},
   JOURNAL = {Ann. Sci. \'Ec. Norm. Sup\'er. (4)},
  FJOURNAL = {Annales Scientifiques de l'\'Ecole Normale Sup\'erieure.
              Quatri\`eme S\'erie},
    VOLUME = {49},
      YEAR = {2016},
    NUMBER = {6},
     PAGES = {1311--1386},
      ISSN = {0012-9593,1873-2151},
   MRCLASS = {32G15 (30F30 37D40 37D50)},
  MRNUMBER = {3592359},
MRREVIEWER = {Serge\ E.\ Troubetzkoy},
       DOI = {10.24033/asens.2310},
       URL = {https://doi-org.ezproxy.library.wisc.edu/10.24033/asens.2310},
}

@article {Aggarwal-LargeGenus,
    AUTHOR = {Aggarwal, Amol},
     TITLE = {Large genus asymptotics for {S}iegel-{V}eech constants},
   JOURNAL = {Geom. Funct. Anal.},
  FJOURNAL = {Geometric and Functional Analysis},
    VOLUME = {29},
      YEAR = {2019},
    NUMBER = {5},
     PAGES = {1295--1324},
      ISSN = {1016-443X,1420-8970},
   MRCLASS = {32G15 (37D40)},
  MRNUMBER = {4025514},
MRREVIEWER = {Ioannis\ D.\ Platis},
       DOI = {10.1007/s00039-019-00509-0},
       URL = {https://doi-org.ezproxy.library.wisc.edu/10.1007/s00039-019-00509-0},
}

@article {Mumford-theta,
    AUTHOR = {Mumford, David},
     TITLE = {Theta characteristics of an algebraic curve},
   JOURNAL = {Ann. Sci. \'Ecole Norm. Sup. (4)},
  FJOURNAL = {Annales Scientifiques de l'\'Ecole Normale Sup\'erieure.
              Quatri\`eme S\'erie},
    VOLUME = {4},
      YEAR = {1971},
     PAGES = {181--192},
      ISSN = {0012-9593},
   MRCLASS = {14H99},
  MRNUMBER = {292836},
MRREVIEWER = {P.\ E.\ Newstead},
       URL = {http://www.numdam.org/item?id=ASENS_1971_4_4_2_181_0},
}

@article {Atiyah-spin,
    AUTHOR = {Atiyah, Michael F.},
     TITLE = {Riemann surfaces and spin structures},
   JOURNAL = {Ann. Sci. \'Ecole Norm. Sup. (4)},
  FJOURNAL = {Annales Scientifiques de l'\'Ecole Normale Sup\'erieure.
              Quatri\`eme S\'erie},
    VOLUME = {4},
      YEAR = {1971},
     PAGES = {47--62},
      ISSN = {0012-9593},
   MRCLASS = {57.50 (30.00)},
  MRNUMBER = {286136},
MRREVIEWER = {J.\ Eells},
       URL = {http://www.numdam.org/item?id=ASENS_1971_4_4_1_47_0},
}

@article{chen2026algebrogeometric,
  title         = {An algebro-geometric perspective on the topology of moduli spaces of differentials},
  author        = {Chen, Dawei and Yu, Fei},
  journal       = {arXiv preprint arXiv:2601.13127},
  year          = {2026},
  eprint        = {2601.13127},
  archivePrefix = {arXiv},
  primaryClass  = {math.AG},
  url           = {https://arxiv.org/abs/2601.13127}
}

@article {Johnson-Spin,
    AUTHOR = {Johnson, Dennis},
     TITLE = {Spin structures and quadratic forms on surfaces},
   JOURNAL = {J. London Math. Soc. (2)},
  FJOURNAL = {Journal of the London Mathematical Society. Second Series},
    VOLUME = {22},
      YEAR = {1980},
    NUMBER = {2},
     PAGES = {365--373},
      ISSN = {0024-6107,1469-7750},
   MRCLASS = {57R15 (10C05)},
  MRNUMBER = {588283},
MRREVIEWER = {Neal\ W.\ Stoltzfus},
       DOI = {10.1112/jlms/s2-22.2.365},
       URL = {https://doi-org.ezproxy.library.wisc.edu/10.1112/jlms/s2-22.2.365},
}

@article{ApisaSalter,
  author       = {Paul Apisa and Nick Salter},
  title        = {Invariant measures on moduli spaces of twisted holomorphic 1-forms and strata of dilation surfaces},
  journal      = {arXiv preprint arXiv:2507.10685},
  year         = {2025},
  archivePrefix= {arXiv},
  eprint       = {2507.10685},
  primaryClass = {math.GT},
  url          = {https://arxiv.org/abs/2507.10685},
  doi          = {10.48550/arXiv.2507.10685}
}

@article {Randal-Williams-RelativeArf,
    AUTHOR = {Randal-Williams, Oscar},
     TITLE = {Homology of the moduli spaces and mapping class groups of
              framed, {$r$}-{S}pin and {P}in surfaces},
   JOURNAL = {J. Topol.},
  FJOURNAL = {Journal of Topology},
    VOLUME = {7},
      YEAR = {2014},
    NUMBER = {1},
     PAGES = {155--186},
      ISSN = {1753-8416,1753-8424},
   MRCLASS = {55R40 (57R15)},
  MRNUMBER = {3180616},
MRREVIEWER = {J.\ M.\ Boardman},
       DOI = {10.1112/jtopol/jtt029},
       URL = {https://doi-org.ezproxy.library.wisc.edu/10.1112/jtopol/jtt029},
}

@article {HumphriesJohnson,
    AUTHOR = {Humphries, Stephen P. and Johnson, Dennis},
     TITLE = {A generalization of winding number functions on surfaces},
   JOURNAL = {Proc. London Math. Soc. (3)},
  FJOURNAL = {Proceedings of the London Mathematical Society. Third Series},
    VOLUME = {58},
      YEAR = {1989},
    NUMBER = {2},
     PAGES = {366--386},
      ISSN = {0024-6115,1460-244X},
   MRCLASS = {57N05 (20F34 55M25)},
  MRNUMBER = {977482},
MRREVIEWER = {Ruth\ Charney},
       DOI = {10.1112/plms/s3-58.2.366},
       URL = {https://doi-org.ezproxy.library.wisc.edu/10.1112/plms/s3-58.2.366},
}

@article {Bogatyrev-Gendron,
    AUTHOR = {Bogatyr\"ev, Andrei and Gendron, Quentin},
     TITLE = {The space of solvable {P}ell-{A}bel equations},
   JOURNAL = {Compos. Math.},
  FJOURNAL = {Compositio Mathematica},
    VOLUME = {161},
      YEAR = {2025},
    NUMBER = {7},
     PAGES = {1483--1511},
      ISSN = {0010-437X,1570-5846},
   MRCLASS = {30F30 (14H15 14H45 32G15)},
  MRNUMBER = {4954107},
       DOI = {10.1112/s0010437x25007158},
       URL = {https://doi-org.ezproxy.library.wisc.edu/10.1112/s0010437x25007158},
}

@article {BCGGM-k-diff,
    AUTHOR = {Bainbridge, Matt and Chen, Dawei and Gendron, Quentin and
              Grushevsky, Samuel and M\"oller, Martin},
     TITLE = {Strata of {$k$}-differentials},
   JOURNAL = {Algebr. Geom.},
  FJOURNAL = {Algebraic Geometry},
    VOLUME = {6},
      YEAR = {2019},
    NUMBER = {2},
     PAGES = {196--233},
      ISSN = {2313-1691,2214-2584},
   MRCLASS = {30F30 (14H15 32J05)},
  MRNUMBER = {3914751},
MRREVIEWER = {Sebasti\'an\ Reyes-Carocca},
       DOI = {10.14231/ag-2019-011},
       URL = {https://doi-org.ezproxy.library.wisc.edu/10.14231/ag-2019-011},
}

@article{Aygun2025_PrimeOrderkDifferentials,
  author       = {Juliet Aygun},
  title        = {Counting geodesics on prime-order $k$-differentials},
  journal      = {arXiv preprint arXiv:2506.24084},
  year         = {2025},
  archivePrefix= {arXiv},
  eprint       = {2506.24084},
  primaryClass = {math.DS},
  url          = {https://arxiv.org/abs/2506.24084},
  doi          = {10.48550/arXiv.2506.24084}
}

@article {Chen-AffineGeometry,
    AUTHOR = {Chen, Dawei},
     TITLE = {Affine geometry of strata of differentials},
   JOURNAL = {J. Inst. Math. Jussieu},
  FJOURNAL = {Journal of the Institute of Mathematics of Jussieu. JIMJ.
              Journal de l'Institut de Math\'ematiques de Jussieu},
    VOLUME = {18},
      YEAR = {2019},
    NUMBER = {6},
     PAGES = {1331--1340},
      ISSN = {1474-7480,1475-3030},
   MRCLASS = {14H15 (14H10)},
  MRNUMBER = {4021107},
MRREVIEWER = {Milagros\ Izquierdo},
       DOI = {10.1017/s1474748017000445},
       URL = {https://doi-org.ezproxy.library.wisc.edu/10.1017/s1474748017000445},
}

@book {FarbMargalit-Primer,
    AUTHOR = {Farb, Benson and Margalit, Dan},
     TITLE = {A primer on mapping class groups},
    SERIES = {Princeton Mathematical Series},
    VOLUME = {49},
 PUBLISHER = {Princeton University Press, Princeton, NJ},
      YEAR = {2012},
     PAGES = {xiv+472},
      ISBN = {978-0-691-14794-9},
   MRCLASS = {57M50 (20F36 20F65 57M07 57N05)},
  MRNUMBER = {2850125},
MRREVIEWER = {Stephen\ P.\ Humphries},
}

@article{ABW2,
  author       = {Paul Apisa and Matt Bainbridge and Jane Wang},
  title        = {Holonomy of affine surfaces},
  journal      = {arXiv preprint arXiv:2508.00100},
  year         = {2025},
  archivePrefix= {arXiv},
  eprint       = {2508.00100},
  primaryClass = {math.AG},
  url          = {https://arxiv.org/abs/2508.00100},
  doi          = {10.48550/arXiv.2508.00100}
}

@InCollection{thurstonshapes,
  author     = {Thurston, W.},
  title      = {Shapes of polyhedra and triangulations of the sphere},
  booktitle  = {The {E}pstein birthday schrift},
  publisher  = {Geom. Topol. Publ., Coventry},
  year       = {1998},
  volume     = {1},
  series     = {Geom. Topol. Monogr.},
  pages      = {511--549},
  doi        = {10.2140/gtm.1998.1.511},
  mrclass    = {57M50 (20F67 22E40 53C21)},
  mrnumber   = {1668340},
  mrreviewer = {Feng Luo},
  url        = {https://doi.org/10.2140/gtm.1998.1.511},
}

@article {ApisaWrightHighRank,
    AUTHOR = {Apisa, Paul and Wright, Alex},
     TITLE = {High rank invariant subvarieties},
   JOURNAL = {Ann. of Math. (2)},
  FJOURNAL = {Annals of Mathematics. Second Series},
    VOLUME = {198},
      YEAR = {2023},
    NUMBER = {2},
     PAGES = {657--726},
      ISSN = {0003-486X,1939-8980},
   MRCLASS = {32G15 (14H15 37D40)},
  MRNUMBER = {4635302},
       DOI = {10.4007/annals.2023.198.2.4},
       URL = {https://doi.org/10.4007/annals.2023.198.2.4},
}

@article {CMZ,
    AUTHOR = {Costantini, Matteo and M\"oller, Martin and Zachhuber,
              Jonathan},
     TITLE = {The area is a good enough metric},
   JOURNAL = {Ann. Inst. Fourier (Grenoble)},
  FJOURNAL = {Universit\'e{} de Grenoble. Annales de l'Institut Fourier},
    VOLUME = {74},
      YEAR = {2024},
    NUMBER = {3},
     PAGES = {1017--1059},
      ISSN = {0373-0956,1777-5310},
   MRCLASS = {14H15 (30F30 32C30 32G15 32J05 53C07)},
  MRNUMBER = {4770336},
MRREVIEWER = {R.\ Michael\ Porter},
       DOI = {10.5802/aif.3592},
       URL = {https://doi.org/10.5802/aif.3592},
}

@article {EMMW,
    AUTHOR = {Eskin, Alex and McMullen, Curtis T. and Mukamel, Ronen E. and
              Wright, Alex},
     TITLE = {Billiards, quadrilaterals and moduli spaces},
   JOURNAL = {J. Amer. Math. Soc.},
  FJOURNAL = {Journal of the American Mathematical Society},
    VOLUME = {33},
      YEAR = {2020},
    NUMBER = {4},
     PAGES = {1039--1086},
      ISSN = {0894-0347},
   MRCLASS = {32G15 (14C30 14H52 37C83)},
  MRNUMBER = {4155219},
       DOI = {10.1090/jams/950},
       URL = {https://doi-org.proxy.lib.umich.edu/10.1090/jams/950},
}

@unpublished{MMW,
	author = "McMullen, Curtis T. and Mukamel, Ronen E. and Wright, Alex",
	howpublished = "\url{http://math.harvard.edu/~ctm/papers/home/text/papers/gothic/gothic.pdf}",
	note = "preprint",
	title = "{Cubic curves and totally geodesic subvarieties of moduli space}",
	year = "2016"
}

@article {Boissy,
    AUTHOR = {Boissy, Corentin},
     TITLE = {Configurations of saddle connections of quadratic
              differentials on {$\Bbb{CP}^1$} and on hyperelliptic {R}iemann
              surfaces},
   JOURNAL = {Comment. Math. Helv.},
  FJOURNAL = {Commentarii Mathematici Helvetici. A Journal of the Swiss
              Mathematical Society},
    VOLUME = {84},
      YEAR = {2009},
    NUMBER = {4},
     PAGES = {757--791},
}

@article {Boissy-Components,
    AUTHOR = {Boissy, Corentin},
     TITLE = {Connected components of the strata of the moduli space of
              meromorphic differentials},
   JOURNAL = {Comment. Math. Helv.},
  FJOURNAL = {Commentarii Mathematici Helvetici. A Journal of the Swiss
              Mathematical Society},
    VOLUME = {90},
      YEAR = {2015},
    NUMBER = {2},
     PAGES = {255--286},
      ISSN = {0010-2571,1420-8946},
   MRCLASS = {32G15 (30F30 57R30)},
  MRNUMBER = {3351745},
MRREVIEWER = {Athanase\ Papadopoulos},
       DOI = {10.4171/CMH/353},
       URL = {https://doi-org.ezproxy.library.wisc.edu/10.4171/CMH/353},
}

@article {ChenMoller-Exceptional,
    AUTHOR = {Chen, Dawei and M\"oller, Martin},
     TITLE = {Quadratic differentials in low genus: exceptional and
              non-varying strata},
   JOURNAL = {Ann. Sci. \'Ec. Norm. Sup\'er. (4)},
  FJOURNAL = {Annales Scientifiques de l'\'Ecole Normale Sup\'erieure.
              Quatri\`eme S\'erie},
    VOLUME = {47},
      YEAR = {2014},
    NUMBER = {2},
     PAGES = {309--369},
      ISSN = {0012-9593,1873-2151},
   MRCLASS = {30F30 (14F10 14H15 30F60 32G15)},
  MRNUMBER = {3215925},
MRREVIEWER = {Andreas\ H\"oring},
       DOI = {10.24033/asens.2216},
       URL = {https://doi-org.ezproxy.library.wisc.edu/10.24033/asens.2216},
}

@article {ChenGendron,
    AUTHOR = {Chen, Dawei and Gendron, Quentin},
     TITLE = {Towards a classification of connected components of the strata
              of {$k$}-differentials},
   JOURNAL = {Doc. Math.},
  FJOURNAL = {Documenta Mathematica},
    VOLUME = {27},
      YEAR = {2022},
     PAGES = {1031--1100},
      ISSN = {1431-0635,1431-0643},
   MRCLASS = {30F30 (14H10 14H15 32G15)},
  MRNUMBER = {4452231},
}

@Article{troyanov,
  author     = {Troyanov, M.},
  title      = {Les surfaces euclidiennes \`a singularit\'{e}s coniques},
  journal    = {Enseign. Math. (2)},
  year       = {1986},
  volume     = {32},
  number     = {1-2},
  pages      = {79--94},
  issn       = {0013-8584},
  fjournal   = {L'Enseignement Math\'{e}matique. Revue Internationale. 2e S\'{e}rie},
  mrclass    = {30F20},
  mrnumber   = {850552},
  mrreviewer = {E. Bujalance},
}

@Article{veech93,
  author     = {Veech, W. A.},
  title      = {Flat surfaces},
  journal    = {Amer. J. Math.},
  year       = {1993},
  volume     = {115},
  number     = {3},
  pages      = {589--689},
  issn       = {0002-9327},
  doi        = {10.2307/2375075},
  fjournal   = {American Journal of Mathematics},
  mrclass    = {30F60 (32G15)},
  mrnumber   = {1221838},
  mrreviewer = {Yoichi Imayoshi},
  url        = {https://doi-org.proxyiub.uits.iu.edu/10.2307/2375075},
}

@article {MZ,
    AUTHOR = {Masur, Howard and Zorich, Anton},
     TITLE = {Multiple saddle connections on flat surfaces and the principal
              boundary of the moduli spaces of quadratic differentials},
   JOURNAL = {Geom. Funct. Anal.},
  FJOURNAL = {Geometric and Functional Analysis},
    VOLUME = {18},
      YEAR = {2008},
    NUMBER = {3},
     PAGES = {919--987},
}

@article {ApisaWrightDiamonds,
    AUTHOR = {Apisa, Paul and Wright, Alex},
     TITLE = {Reconstructing orbit closures from their boundaries},
   JOURNAL = {Mem. Amer. Math. Soc.},
  FJOURNAL = {Memoirs of the American Mathematical Society},
    VOLUME = {298},
      YEAR = {2024},
    NUMBER = {1487},
     PAGES = {v+141},
      ISSN = {0065-9266,1947-6221},
      ISBN = {978-1-4704-6911-5; 978-1-4704-7853-7},
   MRCLASS = {37D40 (14H15 32G15)},
  MRNUMBER = {4772262},
       DOI = {10.1090/memo/1487},
       URL = {https://doi-org.ezproxy.library.wisc.edu/10.1090/memo/1487},
}

@misc{Apisa-Codim1Hyp,
Author = {Paul Apisa},
Title = {Billiards in right triangles and orbit closures in genus zero strata},
Year = {2021},
Note = {preprint, arXiv:2110.07540 (2021)},
}

@article {AEM,
    AUTHOR = {Avila, Artur and Eskin, Alex and M\"{o}ller, Martin},
     TITLE = {Symplectic and isometric {${\rm SL}(2,\Bbb R)$}-invariant
              subbundles of the {H}odge bundle},
   JOURNAL = {J. Reine Angew. Math.},
  FJOURNAL = {Journal f\"{u}r die Reine und Angewandte Mathematik. [Crelle's
              Journal]},
    VOLUME = {732},
      YEAR = {2017},
     PAGES = {1--20},
}

@misc{B,
   AUTHOR = {Beukers,Frits},
   TITLE = {Gauss' hypergeometric functions},
    howpublished = "\url{www.staff.science.uu.nl/~beuke106/MRIcourse93.ps}"
}

@book {C,
    AUTHOR = {Carlson, James and M{\"u}ller-Stach, Stefan and Peters, Chris},
     TITLE = {Period mappings and period domains},
    SERIES = {Cambridge Studies in Advanced Mathematics},
    VOLUME = {85},
 PUBLISHER = {Cambridge University Press},
   ADDRESS = {Cambridge},
      YEAR = {2003},
}

@article {ChenWright,
    AUTHOR = {Chen, Dawei and Wright, Alex},
     TITLE = {The {WYSIWYG} compactification},
   JOURNAL = {J. Lond. Math. Soc. (2)},
  FJOURNAL = {Journal of the London Mathematical Society. Second Series},
    VOLUME = {103},
      YEAR = {2021},
    NUMBER = {2},
     PAGES = {490--515},
      ISSN = {0024-6107,1469-7750},
   MRCLASS = {14H15 (14H10 30F30 32G15 32G20)},
  MRNUMBER = {4230909},
MRREVIEWER = {Eugenii\ Shustin},
       DOI = {10.1112/jlms.12382},
       URL = {https://doi-org.ezproxy.library.wisc.edu/10.1112/jlms.12382},
}

@article {DM,
    AUTHOR = {Deligne, P. and Mostow, G. D.},
     TITLE = {Monodromy of hypergeometric functions and nonlattice integral
              monodromy},
   JOURNAL = {Inst. Hautes \'Etudes Sci. Publ. Math.},
  FJOURNAL = {Institut des Hautes \'Etudes Scientifiques. Publications
              Math\'ematiques},
    NUMBER = {63},
      YEAR = {1986},
     PAGES = {5--89},
}

@article {EKZbig,
    AUTHOR = {Eskin, Alex and Kontsevich, Maxim and Zorich, Anton},
     TITLE = {Sum of {L}yapunov exponents of the {H}odge bundle with respect
              to the {T}eichm\"uller geodesic flow},
   JOURNAL = {Publ. Math. Inst. Hautes \'Etudes Sci.},
  FJOURNAL = {Publications Math\'ematiques. Institut de Hautes \'Etudes
              Scientifiques},
    VOLUME = {120},
      YEAR = {2014},
     PAGES = {207--333},
}

@article {EMa,
    AUTHOR = {Eskin, Alex and Masur, Howard},
     TITLE = {Asymptotic formulas on flat surfaces},
   JOURNAL = {Ergodic Theory Dynam. Systems},
  FJOURNAL = {Ergodic Theory and Dynamical Systems},
    VOLUME = {21},
      YEAR = {2001},
    NUMBER = {2},
     PAGES = {443--478},
}

@article {EMZboundary,
    AUTHOR = {Eskin, Alex and Masur, Howard and Zorich, Anton},
     TITLE = {Moduli spaces of abelian differentials: the principal
              boundary, counting problems, and the {S}iegel-{V}eech
              constants},
   JOURNAL = {Publ. Math. Inst. Hautes \'Etudes Sci.},
  FJOURNAL = {Publications Math\'ematiques. Institut de Hautes \'Etudes
              Scientifiques},
    NUMBER = {97},
      YEAR = {2003},
     PAGES = {61--179},
}

@article {EM,
    AUTHOR = {Eskin, Alex and Mirzakhani, Maryam},
     TITLE = {Invariant and stationary measures for the {${\rm SL}(2,\Bbb
              R)$} action on moduli space},
   JOURNAL = {Publ. Math. Inst. Hautes \'Etudes Sci.},
  FJOURNAL = {Publications Math\'ematiques. Institut de Hautes \'Etudes
              Scientifiques},
    VOLUME = {127},
      YEAR = {2018},
     PAGES = {95--324},
      ISSN = {0073-8301,1618-1913},
   MRCLASS = {37D40 (22E50 37C85)},
  MRNUMBER = {3814652},
MRREVIEWER = {Thomas\ Ward},
       DOI = {10.1007/s10240-018-0099-2},
       URL = {https://doi.org/10.1007/s10240-018-0099-2},
}

@article {EMM,
    AUTHOR = {Eskin, Alex and Mirzakhani, Maryam and Mohammadi, Amir},
     TITLE = {Isolation, equidistribution, and orbit closures for the {${\rm
              SL}(2,\Bbb R)$} action on moduli space},
   JOURNAL = {Ann. of Math. (2)},
  FJOURNAL = {Annals of Mathematics. Second Series},
    VOLUME = {182},
      YEAR = {2015},
    NUMBER = {2},
     PAGES = {673--721},
}

@article {F,
    AUTHOR = {Forni, Giovanni},
     TITLE = {Deviation of ergodic averages for area-preserving flows on
              surfaces of higher genus},
   JOURNAL = {Ann. of Math. (2)},
    VOLUME = {155},
      YEAR = {2002},
    NUMBER = {1},
     PAGES = {1--103},
}

@article {Fi1,
    AUTHOR = {Filip, Simion},
     TITLE = {Splitting mixed {H}odge structures over affine invariant
              manifolds},
   JOURNAL = {Ann. of Math. (2)},
  FJOURNAL = {Annals of Mathematics. Second Series},
    VOLUME = {183},
      YEAR = {2016},
    NUMBER = {2},
     PAGES = {681--713},
}

@article {H,
    AUTHOR = {Hooper, W. Patrick},
     TITLE = {Grid graphs and lattice surfaces},
   JOURNAL = {Int. Math. Res. Not. IMRN},
  FJOURNAL = {International Mathematics Research Notices. IMRN},
      YEAR = {2013},
    NUMBER = {12},
     PAGES = {2657--2698},
}

@article {KZ,
    AUTHOR = {Kontsevich, Maxim and Zorich, Anton},
     TITLE = {Connected components of the moduli spaces of {A}belian
              differentials with prescribed singularities},
   JOURNAL = {Invent. Math.},
  FJOURNAL = {Inventiones Mathematicae},
    VOLUME = {153},
      YEAR = {2003},
    NUMBER = {3},
     PAGES = {631--678},
}

@incollection {K,
    AUTHOR = {Kontsevich, M.},
     TITLE = {Lyapunov exponents and {H}odge theory},
 BOOKTITLE = {The mathematical beauty of physics ({S}aclay, 1996)},
    SERIES = {Adv. Ser. Math. Phys.},
    VOLUME = {24},
     PAGES = {318--332},
 PUBLISHER = {World Sci. Publ., River Edge, NJ},
      YEAR = {1997},
}

@article {KS,
    AUTHOR = {Kenyon, Richard and Smillie, John},
     TITLE = {Billiards on rational-angled triangles},
   JOURNAL = {Comment. Math. Helv.},
  FJOURNAL = {Commentarii Mathematici Helvetici},
    VOLUME = {75},
      YEAR = {2000},
    NUMBER = {1},
     PAGES = {65--108},
}

@article {Lconn,
    AUTHOR = {Lanneau, Erwan},
     TITLE = {Connected components of the strata of the moduli spaces of
              quadratic differentials},
   JOURNAL = {Ann. Sci. \'Ec. Norm. Sup\'er. (4)},
  FJOURNAL = {Annales Scientifiques de l'\'Ecole Normale Sup\'erieure.
              Quatri\`eme S\'erie},
    VOLUME = {41},
      YEAR = {2008},
    NUMBER = {1},
     PAGES = {1--56},
}

@article {L,
    AUTHOR = {Lochak, Pierre},
     TITLE = {On arithmetic curves in the moduli spaces of curves},
   JOURNAL = {J. Inst. Math. Jussieu},
  FJOURNAL = {Journal of the Institute of Mathematics of Jussieu. JIMJ.
              Journal de l'Institut de Math\'ematiques de Jussieu},
    VOLUME = {4},
      YEAR = {2005},
    NUMBER = {3},
     PAGES = {443--508},
}

@article {Masur-cylinder,
    AUTHOR = {Masur, Howard},
     TITLE = {Closed trajectories for quadratic differentials with an
              application to billiards},
   JOURNAL = {Duke Math. J.},
  FJOURNAL = {Duke Mathematical Journal},
    VOLUME = {53},
      YEAR = {1986},
    NUMBER = {2},
     PAGES = {307--314},
      ISSN = {0012-7094,1547-7398},
   MRCLASS = {30F30 (30C60 58F17)},
  MRNUMBER = {850537},
MRREVIEWER = {H.\ Renelt},
       DOI = {10.1215/S0012-7094-86-05319-6},
       URL = {https://doi-org.ezproxy.library.wisc.edu/10.1215/S0012-7094-86-05319-6},
}

@article {MirWri,
    AUTHOR = {Mirzakhani, Maryam and Wright, Alex},
     TITLE = {The boundary of an affine invariant submanifold},
   JOURNAL = {Invent. Math.},
  FJOURNAL = {Inventiones Mathematicae},
    VOLUME = {209},
      YEAR = {2017},
    NUMBER = {3},
     PAGES = {927--984},
}

@article {MirWri2,
    AUTHOR = {Mirzakhani, Maryam and Wright, Alex},
     TITLE = {Full-rank affine invariant submanifolds},
   JOURNAL = {Duke Math. J.},
  FJOURNAL = {Duke Mathematical Journal},
    VOLUME = {167},
      YEAR = {2018},
    NUMBER = {1},
     PAGES = {1--40},
}

@article {M,
    AUTHOR = {M{\"o}ller, Martin},
     TITLE = {Variations of {H}odge structures of a {T}eichm\"uller curve},
   JOURNAL = {J. Amer. Math. Soc.},
  FJOURNAL = {Journal of the American Mathematical Society},
    VOLUME = {19},
      YEAR = {2006},
    NUMBER = {2},
     PAGES = {327--344 (electronic)},
}

@article {N,
    AUTHOR = {Nguyen, Duc-Manh},
     TITLE = {Parallelogram decompositions and generic surfaces in {$
              H^{\rm hyp}(4)$}},
   JOURNAL = {Geom. Topol.},
  FJOURNAL = {Geometry \& Topology},
    VOLUME = {15},
      YEAR = {2011},
    NUMBER = {3},
     PAGES = {1707--1747},
}

@article {T,
    AUTHOR = {Takeuchi, Kisao},
     TITLE = {Arithmetic triangle groups},
   JOURNAL = {J. Math. Soc. Japan},
  FJOURNAL = {Journal of the Mathematical Society of Japan},
    VOLUME = {29},
      YEAR = {1977},
    NUMBER = {1},
     PAGES = {91--106},
}

@article {V,
    AUTHOR = {Veech, W. A.},
     TITLE = {Teichm\"uller curves in moduli space, {E}isenstein series and
              an application to triangular billiards},
   JOURNAL = {Invent. Math.},
  FJOURNAL = {Inventiones Mathematicae},
    VOLUME = {97},
      YEAR = {1989},
    NUMBER = {3},
     PAGES = {553--583},
}

@article {Vo,
    AUTHOR = {Vorobets, Ya. B.},
     TITLE = {Plane structures and billiards in rational polygons: the
              {V}eech alternative},
   JOURNAL = {Uspekhi Mat. Nauk},
  FJOURNAL = {Rossi\u\i skaya Akademiya Nauk. Moskovskoe Matematicheskoe
              Obshchestvo. Uspekhi Matematicheskikh Nauk},
    VOLUME = {51},
      YEAR = {1996},
    NUMBER = {5(311)},
     PAGES = {3--42},
}

@article {Ward,
    AUTHOR = {Ward, Clayton C.},
     TITLE = {Calculation of {F}uchsian groups associated to billiards in a
              rational triangle},
   JOURNAL = {Ergodic Theory Dynam. Systems},
  FJOURNAL = {Ergodic Theory and Dynamical Systems},
    VOLUME = {18},
      YEAR = {1998},
    NUMBER = {4},
     PAGES = {1019--1042},
}

@book {Y,
    AUTHOR = {Yoshida, Masaaki},
     TITLE = {Fuchsian differential equations},
 PUBLISHER = {Friedr. Vieweg \& Sohn},
   ADDRESS = {Braunschweig},
      YEAR = {1987},
}

@incollection {Z,
    AUTHOR = {Zorich, Anton},
     TITLE = {Flat surfaces},
 BOOKTITLE = {Frontiers in number theory, physics, and geometry. {I}},
     PAGES = {437--583},
 PUBLISHER = {Springer},
   ADDRESS = {Berlin},
      YEAR = {2006},
}
\bibliographystyle{amsalpha}
\end{document}